\documentclass[10.5pt,amsfonts]{amsart}

\usepackage{amsmath,amssymb,amsthm}
\usepackage{enumerate}
\usepackage[all]{xy}
\usepackage{graphicx}
\usepackage{xcolor}
\usepackage[colorlinks=true]{hyperref}
\usepackage{comment}
\usepackage[T1]{fontenc}
\usepackage{enumitem}
\usepackage{mathrsfs}

\SelectTips{eu}{12}

\newtheorem{thm}{Theorem}[section]
\newtheorem{prop}[thm]{Proposition}
\newtheorem{lem}[thm]{Lemma}
\newtheorem{cor}[thm]{Corollary}

\newtheorem{setting}[thm]{Setting}
\newtheorem{mthm}{Theorem}

\theoremstyle{definition}
\newtheorem{defn}[thm]{Definition}
\newtheorem{definition}[thm]{Definition}
\newtheorem{exa}[thm]{Example}

\theoremstyle{remark}
\newtheorem{rem}[thm]{Remark}

\numberwithin{equation}{section}

\newcommand{\RR}{\mathbb{R}}
\newcommand{\ZZ}{\mathbb{Z}}
\newcommand{\NN}{\mathbb{N}}
\newcommand{\CC}{\mathbb{C}}

\DeclareMathOperator{\Ker}{\mathrm{Ker}}

\newcommand{\tg}{\widetilde{h}}

\newcommand{\Mm}{M}
\newcommand{\Nm}{N}
\newcommand{\Xm}{X}
\newcommand{\Qm}{Q}

\newcommand{\xp}{x}

\newcommand{\Gg}{G}
\newcommand{\Ng}{N}
\newcommand{\cGg}{\check{G}}
\newcommand{\cNg}{\check{N}}
\newcommand{\cGam}{\check{\Gamma}}
\newcommand{\Pg}{P}
\newcommand{\cPg}{\check{P}}
\newcommand{\Fg}{F}
\newcommand{\Mg}{M}

\newcommand{\Zg}{Z}
\newcommand{\Gamg}{\Gamma}
\newcommand{\cGamg}{\check{\Gamma}}

\newcommand{\al}{a}

\newcommand{\xl}{x}
\newcommand{\yl}{y}
\newcommand{\zl}{z}
\newcommand{\vl}{v}
\newcommand{\wl}{w}
\newcommand{\gl}{g}
\newcommand{\fl}{f}
\newcommand{\rl}{r}

\newcommand{\gaml}{\gamma}
\newcommand{\laml}{\lambda}

\newcommand{\nuf}{\nu}
\newcommand{\muf}{\mu}
\newcommand{\hmuf}{\hat{\mu}}
\newcommand{\cmuf}{\check{\mu}}
\newcommand{\phf}{\phi}

\newcommand{\hf}{h}
\newcommand{\ff}{f}

\newcommand{\mushf}{\mu_{\rm Sh}}
\newcommand{\mupyf}{\mu_{\rm Py}}
\newcommand{\mubrf}{\mu_{\rm Br}}

\newcommand{\bb}{\mathfrak{b}}
\newcommand{\hbb}{\hat{\mathfrak{b}}}
\newcommand{\cbb}{\check{\mathfrak{b}}}
\newcommand{\Bb}{\mathcal{B}}
\newcommand{\BbG}{\mathcal{B}_{\Gg}}
\newcommand{\bBb}{\bar{\mathcal{B}}}
\newcommand{\bBbG}{\bar{\mathcal{B}}_{\Gg}}
\newcommand{\BbSymp}{\mathcal{B}_{\widetilde{\mathrm{Symp}_0^c}(\Mm,\omega)}}
\newcommand{\bBbSymp}{\bar{\mathcal{B}}_{\widetilde{\mathrm{Symp}_0^c}(\Mm,\omega)}}

\newcommand{\BbcG}{\mathcal{B}_{\cGg}}
\newcommand{\im}{i}
\newcommand{\jm}{j}
\newcommand{\ppm}{p}
\newcommand{\cppm}{\check{p}}

\newcommand{\pim}{\pi}
\newcommand{\qqm}{q}

\newcommand{\mom}{\Phi} 

\newcommand{\vv}{v}
\newcommand{\wv}{w}
\newcommand{\Vv}{V}
\newcommand{\Wv}{W}
\newcommand{\Qv}{Q}

\newcommand{\AAA}{\mathcal{A}}
\newcommand{\Rel}{R}
\newcommand{\VV}{\mathrm{V}}

\newcommand{\Rdim}{\dim_{\mathbb{R}}}
\newcommand{\DD}{D}

\newcommand{\ee}{\epsilon}
\newcommand{\ioP}{\iota}
\newcommand{\iPN}{\iota}
\newcommand{\nn}{l}

\newcommand{\Sympc}{\mathrm{Symp}^c}
\newcommand{\Hamc}{\mathrm{Ham}^c}
\newcommand{\tHamc}{\widetilde{\mathrm{Ham}^c}}
\newcommand{\tSympc}{\widetilde{\mathrm{Symp}_0^c}}
\newcommand{\flux}{\mathrm{Flux}}
\newcommand{\tflux}{\widetilde{\mathrm{Flux}}}
\newcommand{\cal}{\mathrm{Cal}}

\newcommand{\QQQ}{\mathrm{Q}}

\newcommand{\HHH}{\mathrm{H}}

\newcommand{\Sp}{\mathrm{Sp}}
\newcommand{\UU}{\mathrm{U}}

\newcommand{\id}{\mathrm{id}}

\newcommand{\vol}{\mathrm{Vol}}

\newcommand{\qd}{\mathscr{A}}
\newcommand{\XX}{X} 
\newcommand{\fk}{f_1^{(k)}}
\newcommand{\gk}{f_2^+}
\newcommand{\hk}{f_2^-}
\newcommand{\fkp}{\bar{f}_1^{(k)}}
\newcommand{\gkp}{\bar{f}_2^{+}}
\newcommand{\hkp}{\bar{f}_2^{-}}

\makeatletter
\newsavebox{\@brx}
\newcommand{\llangle}[1][]{\savebox{\@brx}{\(\m@th{#1\langle}\)}%
  \mathopen{\copy\@brx\kern-0.5\wd\@brx\usebox{\@brx}}}
\newcommand{\rrangle}[1][]{\savebox{\@brx}{\(\m@th{#1\rangle}\)}%
  \mathclose{\copy\@brx\kern-0.5\wd\@brx\usebox{\@brx}}}
\makeatother

\allowdisplaybreaks[2]

\makeatletter
\@namedef{subjclassname@2020}{%
  \textup{2020} Mathematics Subject Classification}
\makeatother


\begin{document}

\title[Flux homomorphism and bilinear form]{Flux homomorphism and bilinear form constructed from Shelukhin's quasimorphism}

\author[M. Kawasaki]{Morimichi Kawasaki}
\address[Morimichi Kawasaki]{Department of Mathematics, Faculty of Science, Hokkaido University, North 10, West 8, Kita-ku, Sapporo, Hokkaido 060-0810, Japan}
\email{kawasaki@math.sci.hokudai.ac.jp}

\author[M.~Kimura]{Mitsuaki Kimura}
\address[Mitsuaki Kimura]{Department of Mathematics, Osaka Dental University, 8-1 Kuzuha Hanazono-cho, Hirakata, Osaka 573-1121, Japan}
\email{kimura-m@cc.osaka-dent.ac.jp}

\author[S. Maruyama]{Shuhei Maruyama}
\address[Shuhei Maruyama]{School of Mathematics and Physics, College of Science and Engineering, Kanazawa University, Kakuma-machi, Kanazawa, Ishikawa, 920-1192, Japan}
\email{smaruyama@se.kanazawa-u.ac.jp}

\author[T. Matsushita]{Takahiro Matsushita}
\address[Takahiro Matsushita]{Department of Mathematical Sciences, Faculty of Science, Shinshu University, Matsumoto, Nagano, 390-8621, Japan}
\email{matsushita@shinshu-u.ac.jp}

\author[M. Mimura]{Masato Mimura}
\address[Masato Mimura]{Mathematical Institute, Tohoku University, 6-3, Aramaki Aza-Aoba, Aoba-ku, Sendai 980-8578, Japan}
\email{m.masato.mimura.m@tohoku.ac.jp}

\begin{abstract}
Given a closed connected symplectic manifold $(M,\omega)$, we construct an alternating $\mathbb{R}$-bilinear form $\mathfrak{b}=\mathfrak{b}_{\mu_{\mathrm{Sh}}}$ on the real first cohomology of $M$ from Shelukhin's quasimorphism $\mu_{\mathrm{Sh}}$. Here $\mu_{\mathrm{Sh}}$ is defined on the universal cover of the group of Hamiltonian diffeomorphisms on $(M,\omega)$. This bilinear form is invariant under the symplectic mapping class group action, and $\mathfrak{b}$ yields a constraint on the fluxes of commuting two elements in the group of symplectomorphisms on $(M,\omega)$. These results might be seen as an analog of Rousseau's result for an open connected symplectic manifold, where he recovered the symplectic pairing from the Calabi homomorphism. Furthermore, $\mathfrak{b}$ controls the extendability of Shelukhin's quasimorphisms, as well as the triviality of a characteristic class of Reznikov. To construct $\mathfrak{b}$, we build general machinery for a group $G$ of producing a real-valued $\mathbb{Z}$-bilinear form $\mathfrak{b}_{\mu}$
from a $G$-invariant quasimorphism $\mu$ on the commutator subgroup of $G$.
\end{abstract}

\maketitle

\tableofcontents

\section{Introduction}

\subsection{Motivation of our results: constructing bilinear forms out of invariant quasimorphisms}\label{subsec=RousseauPy}

 First, we recall some basic notions in symplectic geometry.
For more precise definitions, see Section \ref{sec=prelim}.
Let $(\Mm, \omega)$ be a connected symplectic manifold,  and $\Sympc(\Mm, \omega)$ the group of symplectomorphisms with compact supports on $\Mm$.
Let $\Sympc_0(\Mm, \omega)$ denote the identity component of $\Sympc(\Mm, \omega)$, and $\tSympc(\Mm, \omega)$ its universal covering (here, $\Sympc(\Mm, \omega)$ is endowed with the $C^\infty$-topology). 
Let $\tHamc(\Mm,\omega)$ denote the universal covering of the group $\Hamc(\Mm,\omega)$ of Hamiltonian diffeomorphisms.
Then, we can regard $\tHamc(\Mm,\omega)$ as a normal subgroup of $\tSympc(\Mm,\omega)$ (see Proposition \ref{survey on flux}).
The {\it flux homomorphism} $\tflux_\omega \colon \tSympc(\Mm, \omega) \to \HHH_c^1(\Mm; \RR)$ is defined by
\[\tflux_\omega([\{\psi^t\}_{t\in[0,1]}])=\int_0^1[\iota_{X_t}\omega]dt,\]
where  $X_t$ is the vector field generating the symplectic isotopy $\{\psi^t\}$. It is known that the flux homomorphism provides the following short exact sequence
\[
1 \longrightarrow \tHamc(\Mm,\omega) \longrightarrow \tSympc(\Mm,\omega) \xrightarrow{\tflux_{\omega}} \HHH^1_c(\Mm;\RR) \longrightarrow 1.
\]
In particular, we have $\tHamc(\Mm,\omega)\geqslant [\tSympc(\Mm,\omega),\tSympc(\Mm,\omega)]$. Here, our convention of the group commutator is: $[\gl_1,\gl_2]=\gl_1\gl_2\gl_1^{-1}\gl_2^{-1}$.

Our main results are motivated by a result of Rousseau: in \cite{Rou}, Rousseau  made a correspondence between the Calabi homomorphism $\cal\colon \tHamc(\Mm,\omega)\to \RR$ (see Subsection \ref{subsec=symp}) and a certain alternating $\RR$-bilinear form on $\HHH_c^1(\Mm;\RR)$ when $(\Mm,\omega)$ is an \emph{open} connected symplectic manifold. The bilinear form appearing in this correspondence is the \textit{symplectic pairing} $\bb_\omega$ (see \cite{Cala}, see also \cite[Definition 4.2.10]{Ban97}), which is the bilinear form on $\HHH_c^1(\Mm;\RR)$ defined by
\begin{equation}\label{eq=symplecticpairing}
\bb_\omega([\alpha],[\beta])=\int_{\Mm} \alpha \wedge \beta \wedge\omega^{n-1} \in\RR,
\end{equation}
where $2n$ is the dimension of $\Mm$ and we identify $\HHH_c^1(\Mm;\RR)$ with the first de Rham cohomology with compact support of $\Mm$. 

\begin{thm}[Rousseau {\cite[Proposition 4.1]{Rou}}] \label{thm=rousseau}
Let $(M,\omega)$ be a $2n$-dimensional open connected symplectic manifold.
Then, for all $\fl,\gl\in \tSympc(\Mm,\omega)$,  we have
\[
\cal([\fl,\gl]) = n \cdot \bb_\omega(\tflux_{\omega}(\fl), \tflux_{\omega}(\gl)).
\]
\end{thm}

Our main goal in the present paper is to obtain an analog of Rousseau's correspondence in Theorem~\ref{thm=rousseau} for a \emph{closed} symplectic manifold. More precisely, we regard Theorem~\ref{thm=rousseau} as a construction of an alternating $\RR$-bilinear form $\bb_{\omega}$ out of the Calabi homomorphism $\cal$. However, for a closed connected symplectic manifold $(\Mm,\omega)$, the group $\tHamc(\Mm,\omega)$ does \emph{not} admit a non-zero homomorphism to $\RR$ since $\tHamc(\Mm,\omega)$ is perfect (see Banyaga \cite{Ban}). Our idea is to employ a homogeneous \emph{quasimorphism} on $\tHamc(\Mm,\omega)$ that is invariant under the adjoint action by $\tSympc(\Mm,\omega)$. Here, for a group $\Ng$, a function $\muf\colon \Ng\to \RR$ is called a \emph{quasimorphism} if the defect 
\[
\DD(\muf)=\sup\{|\muf(\xl_1\xl_2)-\muf(\xl_1)-\muf(\xl_2)|\;|\;\xl_1,\xl_2\in \Ng\}
\]
of $\muf$ is finite. A quasimorphism $\muf\colon \Ng\to \RR$ is said to be \emph{homogeneous} if $\muf(\xl^k)=k\muf(\xl)$ for every $\xl\in \Ng$ and $k\in \ZZ$. If $\Ng$ is a normal subgroup of a group $\Gg$, then the concept of $\Gg$-invariance of $\muf$ is defined as the invariance under the adjoint $\Gg$-action on $\Ng$. We will recall these definitions in Subsection~\ref{subsec=qm} in more detail. For an open connected symplectic manifold $(\Mm,\omega)$, the Calabi homomorphism $\cal\colon \tHamc(\Mm,\omega)\to \RR$ can be regarded as a $\tSympc(\Mm,\omega)$-invariant homogeneous \emph{quasimorphism with defect $0$}.

For  a closed connected symplectic surface $(\Mm,\omega)$ of genus at least two, Py \cite{Py06} constructed a $\tSympc(\Mm,\omega)$-invariant homogeneous quasimorphism $\mupyf$ on $\tHamc(\Mm,\omega)$, which is called \emph{Py's Calabi quasimorphism.} 
In this setting, we obtain the following result. Here, for $a,b\in \RR$ and $C\in\RR_{\geq 0}$, we write $a\sim_C b$ if $|a-b|\leq C$. 
\begin{thm}[constructing $\bb_{\omega}$ out of $\mupyf$]\label{thm=KKMM}
Let $(\Mm,\omega)$ be a closed connected surface of genus at least two equipped with a symplectic form. Let $\mupyf$ be Py's Calabi quasimorphism on $\tHamc(\Mm,\omega)$. Then, for all $\fl,\gl\in \tSympc(\Mm,\omega)$, we have
\begin{equation}\label{eq=Py_form}
\mupyf([\fl,\gl])\sim_{\DD(\mupyf)} \bb_{\omega}(\tflux_{\omega}(\fl),\tflux_{\omega}(\gl)).
\end{equation}
\end{thm}

Theorem~\ref{thm=KKMM} may be seen as an analog of the construction of $\bb_{\omega}$ out of $\cal$ in Rousseau's theorem (Theorem~\ref{thm=rousseau}). Indeed, we note that $\DD(\cal)=0$ for the case where $(\Mm,\omega)$ is open. We also remark that \eqref{eq=Py_form}  yields the following \emph{limit formula}:
\begin{equation}\label{eq=Py_form_limit}
\bb_{\omega}(\tflux_{\omega}(\fl),\tflux_{\omega}(\gl))=\lim_{k\to \infty}\frac{\mupyf([\fl^k,\gl])}{k}=\lim_{l\to \infty}\frac{\mupyf([\fl,\gl^l])}{l};
\end{equation}
we will argue in a more general setting in Subsection~\ref{subsec=limit_formula}.

\begin{rem}\label{rem=pi1trivial}
In the setting of Theorem~\ref{thm=KKMM}, it is known that $\pi_1(\Sympc_0(\Mm,\omega))$ and $\pi_1(\Hamc(\Mm,\omega))$ are both trivial (see Proposition~\ref{survey on flux}~(6)), so that $\tSympc(\Mm,\omega)=\Sympc_0(\Mm,\omega)$ and $\tHamc(\Mm,\omega)=\Hamc(\Mm,\omega)$. 
\end{rem}

Limit formula \eqref{eq=Py_form_limit}, together with Remark~\ref{rem=pi1trivial}, immediately implies the following result  (see also Remark~\ref{rem=intersection}). This was the main theorem in the previous work \cite{KKMM2} of some of the authors, and  at the moment they were unaware of the theory behind it of  constructing $\bb_{\omega}$ out of $\mupyf$.

\begin{cor}[{\cite[Theorem~1.1]{KKMM2}}]\label{cor=KKMM}
Let $(\Mm,\omega)$ be a closed connected surface of genus at least two equipped with a symplectic form. 
For all $\fl,\gl\in \Sympc_0(\Mm,\omega)$ satisfying $\fl\gl=\gl\fl$, we have
\[
\flux_{\omega}(\fl)\smile \flux_{\omega}(\gl)=0,
\]
where $\smile$ denotes the cup product on the  cohomology ring $\HHH^{\bullet}(\Mm;\RR)$ of $\Mm$.
\end{cor}

%

Our main result employs \emph{Shelukhin's quasimorphism} to construct an alternating $\RR$-bilinear form $\bb_{\mushf}$ on $\HHH^1_c(\Mm;\RR)$ for a general closed connected symplectic manifold $(\Mm,\omega)$. We will go into greater detail in the next subsection. We summarize the results of constructing bilinear forms $\bb_{\muf}$ out of invariant quasimorphisms $\muf$ in Table~\ref{table=bilinear}.

\begin{table}[h]
\caption{constructing bilinear forms out of invariant quasimorphisms}
\label{table=bilinear}
\begin{tabular}{c||c|c|c}
 \begin{tabular}{c}
$\Mm^{2n}$ \\ (symplectic manifold)
\end{tabular}
& open & 
\begin{tabular}{c}closed surface \\ of genus $\geq 2$ 
\end{tabular}

& closed \\
 \hline 
 \begin{tabular}{c}
$\muf$ \\ ($\tSympc(\Mm)$-invariant \\ quasimorphism  \\ on $\tHamc(\Mm)$)
\end{tabular} & 
\begin{tabular}{c}
$\cal$\medskip  \\ Calabi \\ homomorphsm 
\end{tabular}
& 
\begin{tabular}{c}
$\mupyf$ \medskip \\ Py's Calabi \\ quasimomorphsm \cite{Py06}
\end{tabular}
&  
\begin{tabular}{c}
$\mushf$ \medskip \\  Shelukhin's \\ quasimomorphsm \cite{Shelukhin}
\end{tabular} \\ \hline
\begin{tabular}{c}
$\bb_{\muf}$ \\ (bilinear form on $\HHH^1_c(\Mm)$ \\ constructed out of $\muf$)
\end{tabular}
&
\begin{tabular}{c}
$n\cdot \bb_{\omega}$ \\ Rousseau \cite{Rou} \\ {[Theorem~\ref{thm=rousseau}]}
\end{tabular}
&
\begin{tabular}{c}
$ \bb_{\omega}$ \\ {[Theorem~\ref{thm=KKMM}]}
\end{tabular}
&
\begin{tabular}{c}
$\Sympc(\Mm)$-invariant \\ $\RR$-bilinear form $\bb_{\mushf}$ \\ {[Theorem~\ref{mthm=RRbilinear}]} \\ \hline
$\bb_{\mushf}$ is determined  \\ fully or partially  \\ for certain $\Mm$ \\ {[Theorem~\ref{mthm=explicitShelukhin}]}

\end{tabular}

\end{tabular}
\end{table}

\subsection{Main result~1: constructing an $\RR$-bilinear form $\bb_{\mushf}$ from Shelukhin's quasimorphism $\mushf$} \label{subsec=Shelukhin1}

Our main results develop the theory of constructing bilinear forms on $\HHH^1_c(\Mm;\RR)$ out of $\tSympc(\Mm,\omega)$-invariant homogeneous quasimorphisms on $\tHamc(\Mm,\omega)$ for a \emph{closed} connected symplectic manifold $(\Mm,\omega)$. The invariant quasimorphism we mainly treat in the present paper is \emph{Shelukhin's quasimorphism} $\mushf$, which was constructed by Shelukhin in \cite{Shelukhin}. We note that $\mushf$ was written as $\mathfrak{S}_{\Mm}$ in \cite{Shelukhin}. We will briefly recall the construction of $\mushf$ in Subsection~\ref{subsec:sh_qm}.

The group  $\tSympc(\Mm,\omega)$  admits an action of $\Sympc(\Mm,\omega)$ by conjugation: for $\theta\in \Sympc(\Mm,\omega)$ and for $[\{\psi^t\}_{t\in [0,1]}]\in \tSympc(\Mm,\omega)$, set
\begin{equation}\label{eq=Sympc-action}
\theta\cdot [\{\psi^t\}_{t\in [0,1]}]=[\{\theta\psi^t\theta^{-1}\}_{t\in [0,1]}].
\end{equation}
This action is by group automorphisms, and it leaves $\tHamc(\Mm,\omega)$  invariant (setwise). 
It induces a $\Sympc(\Mm,\omega)$-action on $\HHH^1_c(\Mm;\RR)=\tSympc(\Mm,\omega)/\tHamc(\Mm,\omega)$. In this manner, we can define the notion of $\Sympc(\Mm,\omega)$-invariance for a bilinear form on $\HHH^1_c(\Mm;\RR)$ (we provide the precise definition of this in a general setting in Subsection~\ref{subsec=action}). 

Our first main result states that a \emph{$\Sympc(\Mm,\omega)$-invariant} alternating \emph{$\RR$-bilinear} form on $\HHH^1_c(\Mm;\RR)$ is constructed out of $\mushf$. 
\begin{mthm}[constructing $\bb_{\mushf}$ out of $\mushf$]\label{mthm=RRbilinear}Let $(\Mm,\omega)$ be a closed connected symplectic manifold. Let $\mushf$ be Shelukhin's quasimorphism on $\tHamc(\Mm,\omega)$. Then, there exists a $\Sympc(\Mm,\omega)$-invariant  alternating $\RR$-bilinear form $\bb_{\mushf}\colon \HHH^1_c(\Mm;\RR)\times \HHH^1_c(\Mm;\RR)\to \RR$ such that for all $\fl,\gl\in \tSympc(\Mm,\omega)$, 
\begin{equation}\label{eq=Shelukhin_form}
\mushf([\fl,\gl])\sim_{\DD(\mushf)} \bb_{\mushf}(\tflux_{\omega}(\fl),\tflux_{\omega}(\gl)).
\end{equation}
\end{mthm}
Similar to deducing \eqref{eq=Py_form_limit} from \eqref{eq=Py_form}, we can obtain from \eqref{eq=Shelukhin_form} the following \emph{limit formula}:
\begin{equation}\label{eq=Shelukhin_form_limit}
\bb_{\mushf}(\tflux_{\omega}(\fl),\tflux_{\omega}(\gl))=\lim_{k\to \infty}\frac{\mushf([\fl^k,\gl])}{k}=\lim_{l\to \infty}\frac{\mushf([\fl,\gl^l])}{l},
\end{equation}
which determines the $\RR$-bilinear form $\bb_{\mushf}$.

\begin{rem}\label{rem=Symp/Symp_0}
On the $\Sympc(\Mm,\omega)$-action on $\HHH^1_c(\Mm;\RR)$, the action by the subgroup $\Sympc_0(\Mm,\omega)$ is trivial.
Hence, this $\Sympc(\Mm,\omega)$-action induces the action of the symplectic mapping class group $\Sympc(\Mm,\omega)/\Sympc_0(\Mm,\omega)$ on $\HHH^1_c(\Mm;\RR)$.
\end{rem}

The $\RR$-bilinear form $\bb_{\mushf}$ may be  of special interest in symplectic geometry, as we will see in the next two subsections. We will also discuss several applications of $\bb_{\mushf}$ in Section~\ref{sec=application}.

\subsection{Main result~2: $\bb_{\mushf}$ controls the extendability of $\mushf$}\label{subsec=Shelukhin2}

In this subsection and the next subsection, we present properties of the $\RR$-bilinear form  $\bb_{\mushf}$, obtained in Theorem~\ref{mthm=RRbilinear}. The first property asserts that $\bb_{\mushf}$  completely controls the  extendability of $\mushf$ in the following sense.

\begin{mthm}[extendability of  $\mushf$]\label{mthm=Shelukhin_extendable}
Let $(\Mm,\omega)$ be a closed connected symplectic manifold. Let $\mushf$ be Shelukhin's quasimorphism on $\tHamc(\Mm,\omega)$. Let $\Pg$ be a subgroup of $\tSympc(\Mm,\omega)$ with $\Pg\geqslant \tHamc(\Mm,\omega)$. Then the following are equivalent.
\begin{enumerate}[label=\textup{(\roman*)}]
  \item The invariant quasimorphism $\mushf$ is extendable to $\Pg$, that is, there exists a homogeneous quasimorphism $\phf$ on $\Pg$ such that $\phf|_{\tHamc(\Mm,\omega)}=\mushf$;
  \item the restriction of $\bb_{\mushf}$ to $\tflux_{\omega}(\Pg)\times \tflux_{\omega}(\Pg)$ equals the zero-form.
\end{enumerate}
\end{mthm}

\subsection{Main result~3: $\bb_{\mushf}$ controls the triviality of the Reznikov class}\label{subsec=Reznikov}

The second property of $\bb_{\mushf}$ asserts that $\bb_{\mushf}$   controls the triviality of the restriction of a characteristic class  $R\in \HHH^2(\Sympc(\Mm,\omega))$ constructed by Reznikov \cite{Reznikov} to a subgroup of $\Sympc_0(\Mm,\omega)$. Here, we consider group cohomology with trivial real coefficients. Hereafter, we call this class $R$ the Reznikov class; we will briefly recall the definition of it in Subsection~\ref{subsec:re_class}. The precise statement goes as follows. Here, the definition of the homomorphism $I_{c_1}\colon \pi_1(\Hamc(\Mm,\omega))\to \RR$ will be given by \eqref{eq=I_{c_1}} in Subsection~\ref{subsec:sh_qm}.

\begin{mthm}[characterization of the triviality of the Reznikov class]\label{mthm=Reznikov}
Let $(\Mm,\omega)$ be a closed connected symplectic manifold. Let $\Pg$ be a subgroup of $\tSympc(\Mm,\omega)$ with $\Pg\geqslant \tHamc(\Mm,\omega)$. Let $\overline{\Pg}$ be the image of $\Pg$ by the universal covering map $\tHamc(\Mm,\omega)\twoheadrightarrow \Hamc(\Mm,\omega)$. Then, the Reznikov class $R|_{\overline{\Pg}}\in \HHH^2(\overline{\Pg})$ is trivial on $\overline{\Pg}$ if and only if both of the following two conditions are satisfied.
\begin{enumerate}[label=\textup{(\arabic*)}]
 \item The homomorphism $I_{c_1}\colon \pi_1(\Hamc(\Mm,\omega))\to \RR$ is the zero map.
 \item The restriction of $\bb_{\mushf}$ to $\tflux_{\omega}(\Pg)\times \tflux_{\omega}(\Pg)$ equals the zero-form.
\end{enumerate}
\end{mthm}
We note that the validity of condition (1) above only depends on $(\Mm,\omega)$; it does not depend on the choice of $\Pg$. Hence, for $(\Mm,\omega)$ fulfilling condition (1), the triviality of the Reznikov class (restricted to a subgroup of $\Sympc_0(\Mm,\omega)$ containing $\Hamc(\Mm,\omega)$) is completely controlled by the $\RR$-bilinear form $\bb_{\mushf}$.

\subsection{Main result~4: the explicit expression of $\bb_{\mushf}$}\label{subsec=explicitexp}

Theorem~\ref{mthm=Shelukhin_extendable} and Theorem~\ref{mthm=Reznikov} might suggest the importance of gaining the explicit expression of $\bb_{\mushf}$. For certain $(\Mm, \omega)$, we succeed in obtaining the explicit expression of $\bb_{\mushf}$ fully or partially. Here, `partially' means that we have a certain $\RR$-linear subspace $\Vv$ of $\HHH^1_c(\Mm;\RR)$ and that we have the explicit expression of $\bb_{\mushf}|_{V\times V}$. Our results are summarized as follows. Here, we equip the product of symplectic manifolds with the standard product symplectic structure, as is explained in Subsection~\ref{subsec=notation} below.

\begin{mthm}[explicit expressions of $\bb_{\mushf}$]\label{mthm=explicitShelukhin}

For the following symplectic manifold $(M,\omega)$, we obtain the explicit expressions of $\bb_{\mushf}$, fully or partially, as follows:

  \begin{enumerate}[label=\textup{(\arabic*)}]
    \item If $(M,\omega)$ is the direct product $(S,\omega_S)\times (N,\omega_N)$ of a closed connected surface $(S,\omega_S)$ whose genus $l$ is at least one equipped with a symplectic form $\omega_S$, and a closed connected symplectic manifold $(N,\omega_N)$, 
   then
\[\bb_{\mushf}|_{\HHH^1_c(S;\RR)\times \HHH^1_c(S;\RR)} \equiv \vol(M,\omega)\frac{2-2l}{\mathrm{Area}(S,\omega_S)^2} \cdot \bb_{\omega_S}.\]
  Here, we regard $\HHH^1_c(S;\RR)$ as a subspace of $\HHH^1_c(M;\RR) \cong \HHH^1_c(S;\RR) \oplus \HHH^1_c(N;\RR)$.
   \item Let $(S_1,\omega_1), \dots, (S_n,\omega_n)$ be closed conncected surfaces of genus $l_1,\dots, l_n$ equipped with symplectic forms, respectively. Assume that $l_i \geq 1$ for every $i=1,\cdots,n$.
   If $(M,\omega)$ is the direct product $(S_1,\omega_1)\times\cdots\times(S_n,\omega_n)$ of the surfaces, then   
\[\bb_{\mushf}(v,w) = \vol(M,\omega) \sum_{i=1}^n  \frac{2-2l_i}{{\rm Area}(S_i,\omega_i)^2} \bb_{\omega_{S_i}}(v_i,w_i)\]
   for every
   $v=(v_1,\ldots,v_n)$ and $w=(w_1,\ldots,w_n)$ in $\HHH^1_c(M;\RR) \cong \HHH^1_c(S_1;\RR) \oplus \cdots \oplus \HHH^1_c(S_n;\RR)$.
   Here, for every $i = 1,\ldots, n$, we regard $v_i$ and $w_i$ as elements in $\HHH^1_c(S_i;\RR)$.
   \item Let $n\geq2$ and $\rho,r ,r_1,\cdots,r_n$ be real numbers with $0<\rho<r < r_1<r_2<\ldots<r_n$.
   If $(M,\omega)$ is the blow-up $(\hat{M}_r,\omega_\rho)$ of the torus $(T^{2n}(r_1,\ldots,r_n),\omega_0)$ with respect to $\iota(r)$, $J_0$ and $\rho$ $($see Subsection $\ref{subsec:blowup}$ for precise definitions$)$,
   then  
\[\bb_{\mushf}(v,w) = \vol(T^{2n}(r_1,\ldots,r_n),\omega_0) A(M, \omega) \sum_{i=1}^n \frac{1}{{\rm Area}(T^2(r_i),\omega_0)} \bb_{\omega_0}(v_i,w_i)  \]
   for every $v=(v_1,\ldots,v_n)$ and $w=(w_1,\ldots,w_n)$ in 
   \begin{align*} 
    &\HHH^1_c(M;\RR)\cong \HHH^1_c(T^{2n}(r_1,\ldots,r_n);\RR) \\ 
   \cong{}  &\HHH^1_c(T^2(r_1);\RR) \oplus \HHH^1_c(T^2(r_2);\RR) \oplus \cdots \oplus \HHH^1_c(T^2(r_n);\RR),
   \end{align*}   
   where we regard $v_i$ and $w_i$ as elements in $\HHH^1_c(T^2(r_i);\RR)$ for every $i = 1,\ldots, n$.
   Here, $A(M,\omega)$ denotes the average Hermitian scalar curvature of $(M;\omega)$ $($see Definition $\ref{defn=aHsc})$.
  \end{enumerate}
 
\end{mthm}


In the next section, before proceeding to the preliminaries, we exhibit several applications of the bilinear form $\bb_{\mushf}$.



\subsection{Strategy of the proofs of Theorem~\ref{mthm=RRbilinear} and Theorem~\ref{mthm=explicitShelukhin}}\label{subsec=outline}

Here, we describe our strategy for proving Theorem~\ref{mthm=RRbilinear}. 
First, in this paper we construct an $\RR$-valued alternating $\ZZ$-bilinear form (we discuss it in Section~\ref{sec=bilinear}) of the abelianization from an invariant quasimorphism in the following general setting: let $\Gg$ be a group, $\Ng$ the commutator subgroup of $\Gg$, and $\Gamg$ the abelianization of $\Gg$. Let $\QQQ(\Ng)^\Gg$ be the linear space consisting of the homogeneous quasimorphisms on $\Ng$ that are invariant under the adjoint action of $\Gg$. We write $i^* \QQQ(\Gg)$ to indicate the linear subspace of $\QQQ(\Ng)^\Gg$ which is extendable to $\Gg$. We refer the reader to Subsection~\ref{subsec=qm} for the more precise description. In Section~\ref{sec=bilinear}, we provide a general recipe for constructing an alternating $\ZZ$-bilinear form on $\Gamma$ from an element in $\QQQ(\Ng)^\Gg/ i^* \QQQ(\Gg)$; this gives rise to an injective $\RR$-linear map
\[
\bBbG\colon \QQQ(\Ng)^\Gg/ i^* \QQQ(\Gg)\hookrightarrow \AAA(\Gamg);\quad [\muf]\mapsto \bb_{\muf},
\]
as we will formulate in Definition~\ref{defn=bilinear_form_map}. Here, $\AAA(\Gamg)$ denotes the $\RR$-linear space of $\RR$-valued alternating $\ZZ$-bilinear forms on $\Gamg$.

Now, we return to the case that $\Gg = \tSympc(\Mm, \omega)$ and $\Ng = \tHamc(\Mm,\omega)$ for a closed connected symplectic manifold $(\Mm,\omega)$. Then the abelianization of $\Gg$ is $\HHH^1_c(\Mm ; \RR)$.  In this case, since the abelianization is now an $\RR$-linear space, it is natural to expect that the $\ZZ$-bilinear form created above is indeed $\RR$-bilinear when appropriate conditions are satisfied. Actually, we provide criteria to show that the above $\ZZ$-bilinear forms are $\RR$-bilinear in terms of topological groups (see Proposition~\ref{prop=RRbilinear} and Corollary~\ref{cor=RRbilinear}).
Using these criteria, we show that the $\ZZ$-bilinear forms constructed from Shelukhin's quasimorphisms are actually $\RR$-bilinear. This is an outline of the proof of Theorem~\ref{mthm=RRbilinear}; the actual proof will be provided in Section~\ref{sec=RRbilinear_Shelukhin}.

For the proof of Theorem~\ref{mthm=explicitShelukhin}, we will first establish an axiomatized theorem (Theorem~\ref{thm=XQ} in Subsection~\ref{subsec=axiom}) and then deduce Theorem~\ref{mthm=explicitShelukhin} as examples of it. 

\subsection{Notation and conventions}\label{subsec=notation}

\begin{enumerate}[label=\textup{(\arabic*)}]
\item We always assume symplectic manifolds to be connected. 
\item We always consider group cohomology with trivial real coefficients; thus we omit the coefficient $\RR$ in group cohomology. We in principle consider the real singular cohomology with compact support, but we use the symbol $\HHH^{\bullet}_c(\cdot;\RR)$, not $\HHH^{\bullet}_c(\cdot)$, in order to make a distinction from group cohomology at a glance.

\item For a smooth manifold $\Mm$  and for $s \ge 0$, a map $\gamma\colon [0, s] \to \mathrm{Diff}(\Mm)$ is called a \emph{smooth isotopy} if the map $M\times[0, s] \to M$ defined by $(x,t) \mapsto \gamma(t)(x)$ is a smooth map. In particular, for a symplectic manifold $(\Mm,\omega)$, the notion of smooth isotopy $\gamma \colon [0, s] \to \Hamc_0(\Mm,\omega)$ and $\gamma \colon [0,s] \to \Sympc_0(\Mm, \omega)$ (see Subsection~\ref{subsec=symp} for the definitions of $\Hamc(\Mm, \omega)$ and $\Sympc_0(\Mm, \omega)$) are defined in the same way as above. We say that $\gamma$ is a \emph{smooth isotopy from identity} if $\gamma$ is a smooth isotopy such that $\gamma(0)=\mathrm{id}_M$.

Let $G$ be  $\Hamc(M,\omega)$ or $\Sympc_0(M,\omega)$.
Then, we define the universal covering $\tilde{G}$ of $G$ by
\[\tilde{G}=\{[\gamma] \mid \gamma\colon [0,1] \to G \text{ is a smooth isotopy  from identity}\}.\]
 Here, $[\gamma]$ is the smooth homotopy class of $\gamma$ in $G$ relative to fixed ends.

 Let $s \ge 0$ and let $\gamma \colon [0,s] \to G$ be a smooth isotopy. Define the smooth isotopy $\hat{\gamma} \colon [0,1] \to G$ by $\hat{\gamma}(t) = \gamma(st)$. Then for a smooth isotopy $\gamma$ from identity, the symbol $\left[\{\gamma(t)\}_{t\in[0,s]}\right]$ denotes the element 
$[\hat{\gamma}]$ of $\tilde{G}$.

\item For a symplectic manifold $(\Mm,\omega)$, we use symbols such as $\tSympc(\Mm,\omega)$, $\tHamc(\Mm,\omega)$, $\Sympc_0(\Mm,\omega)$, $\Hamc(\Mm,\omega)$, $\Sympc(\Mm,\omega)$ and $\HHH^1_c(\Mm;\RR)$ to indicate that we consider these notions with compact support. If $\Mm$ is compact, then these symbols coincide with those without `$c$': $\widetilde{\mathrm{Symp}}(\Mm,\omega)$, $\widetilde{\mathrm{Ham}}(\Mm,\omega)$, $\mathrm{Symp}_0(\Mm,\omega)$, $\mathrm{Ham}(\Mm,\omega)$, $\mathrm{Symp}(\Mm,\omega)$ and $\HHH^1(\Mm;\RR)$, respectively. Nevertheless, we keep to use the symbols with `$c$.' This is because in Sections~\ref{sec=kari1} and \ref{sec=kari2}, we will embed 
an open symplectic manifold
into a closed symplectic manifold; there, the formulation of these notions with compact support is the right one.
\item Given a symplectic manifold $(\Mm,\omega)$ of finite volume, $\mushf^{\Mm}$ denotes Shelukhin's quasimorphism for this $(\Mm,\omega)$. When there is no risk of confusion, we abbreviate $\mushf^{\Mm}$ as $\mushf$. Thus, $\mushf$ means Shelukhin's quasimorpshism $\mushf^{\Mm}$ \emph{for the given symplectic manifold $(\Mm,\omega)$.}

\item Let $\NN$ denote the set of strictly positive integers: $\NN=\{1,2,3,\ldots\}$.
\item As we have already used in Subsection~\ref{subsec=RousseauPy}, for $a,b\in \RR$ and for $C\in \RR_{\geq 0}$, we use the convention `$a\sim_C b$' to mean $|a-b|\leq C$. For a group $\Gg$ and $\gl_1,\gl_2\in \Gg$, our convention of the group commutator is: $[\gl_1,\gl_2]=\gl_1\gl_2\gl_1^{-1}\gl_2^{-1}$.
\item In the present paper, we in principle use the symbol $M$ for a (symplectic) manifold and $N$ for a group. However, in Section~\ref{sec=bilinear}, we sometimes use the symbol $M$ for a group; in Sections~\ref{sec=kari1} and \ref{sec=kari2}, we use the symbol $N$ for a manifold. We add a description at the beginning of Section~\ref{sec=bilinear}.
\item For two symplectic manifolds $(M_1,\omega_1)$ and $(M_2,\omega_2)$, we define their product $(M_1,\omega_1) \times (M_2,\omega_2)$ as the symplectic manifold $(M_1\times M_2,\mathrm{pr}_1^\ast\omega_1+\mathrm{pr}_2^\ast\omega_2)$, where $\mathrm{pr}_1\colon M_1\times M_2\to M_1$ and $\mathrm{pr}_2\colon M_1\times M_2\to M_2$ are the first and second projections, respectively. We often write $\omega_{M_1\times M_2}$ for the symplectic form $\mathrm{pr}_1^\ast\omega_1+\mathrm{pr}_2^\ast\omega_2$.
 \item
Let $(P,\omega_P)$, $(N,\omega_N)$  and $(M,\omega)$ be symplectic manifolds and $\iPN \colon P\times N\to M$ be an open symplectic embedding of  $(P\times N,\omega_{P\times N})$ to $(M,\omega)$. Assume that $N$ is a closed manifold.
Then,  we define a homomorphism $\ioP^{P,M}_\ast\colon \tHamc \left(P,\omega_{P}\right) \to \tHamc (M,\omega)$ as follows.
 For $\tg = \left[\{h^t\}_{t\in[0,1]}\right] \in \tHamc \left(P,\omega_{P}\right)$, we define $\ioP^{P,M}_\ast(\tg) \in \tHamc (M,\omega)$ by
  \[\ioP^{P,M}_\ast(\tg) = \left[\{\iPN_\ast\left(h^t\times \id_N\right)\}_{t\in[0,1]}\right].\]
  Here, $\iPN_\ast\colon \Hamc \left(P\times N,\omega_{P\times N}\right) \to \Hamc (M,\omega)$ is the map induced from the open symplectic embedding $\ioP\colon P\times N\to M$.  We also define $\ioP^{P,M}_\ast\colon \tSympc \left(P,\omega_{P}\right) \to \tSympc (M,\omega)$ similarly.
  \label{item:emb}
\end{enumerate}

\section{Applications of the $\RR$-bilinear form $\bb_{\mushf}$  and organization of the present paper}\label{sec=application}

In Subsections~\ref{subsec=Shelukhin2} and \ref{subsec=Reznikov}, we present two properties of the $\RR$-bilinear form $\bb_{\mushf}$ constructed out of Shelukhin's quasimorphism $\mushf$. In this section, we discuss more applications of $\bb_{\mushf}$.
At the end of this section, we present the organization of the present paper.

\subsection{Applications of the $\RR$-bilinearity of $\bb_{\mushf}$}\label{subsec=application_RR}
As is mentioned in Subsection~\ref{subsec=outline}, the $\RR$-bilinearity (not merely the $\ZZ$-bilinearity) of $\bb_{\mushf}$ is of importance. This $\RR$-bilinearlity, together with Theorem~\ref{mthm=Shelukhin_extendable} and Theorem~\ref{mthm=Reznikov}, yields the following two theorems.

\begin{thm}\label{thm=extendabilityRR}
Let $(\Mm,\omega)$ be a closed symplectic manifold. Let $\Pg$ be a subgroup of $\tSympc(\Mm,\omega)$ with $\Pg\geqslant \tHamc(\Mm,\omega)$. Set $\Vv_{\Pg}$ and $\Gg_{\Pg}^{\RR}$ by 
\begin{equation}\label{eq=VL}
\Vv_{\Pg}=\RR\textrm{-}\mathrm{span}(\tflux_{\omega}(\Pg)) \quad \textrm{and}\quad \Gg_{\Pg}^{\RR}=\tflux_{\omega}^{-1}(\Vv_{\Pg}).
\end{equation}
Then, the following are equivalent.
\begin{enumerate}[label=\textup{(\roman*)}]
 \item $\mushf$ is extendable to $\Pg$;
 \item $\mushf$ is extendable to $\Gg_{\Pg}^{\RR}$.
\end{enumerate}
\end{thm}

\begin{thm}\label{thm=ReznikovRR}
Let $(\Mm,\omega)$ be a closed symplectic manifold. Let $\Pg$ be a subgroup of $\tSympc(\Mm,\omega)$ with $\Pg\geqslant \tHamc(\Mm,\omega)$. Set $\Gg_{\Pg}^{\RR}$ by \eqref{eq=VL}. Let $\overline{\Pg}$ and $\overline{\Gg_{\Pg}^{\RR}}$ be the images of $\Pg$ and $\Gg_{\Pg}^{\RR}$ by the universal covering map $\tHamc(\Mm,\omega)\twoheadrightarrow \Hamc(\Mm,\omega)$, respectively. Then, the following are equivalent.
\begin{enumerate}[label=\textup{(\roman*)}]
 \item The Reznikov class $R$ is trivial on $\overline{\Pg}$, namely $R|_{\overline{\Pg}}=0$ in $\HHH^2(\overline{\Pg})$;
 \item the Reznikov class $R$ is trivial on $\overline{\Gg_{\Pg}^{\RR}}$.
\end{enumerate}
\end{thm}

We note that  $\tHamc(\Mm,\omega)$ is viewed as a discrete group to define $\mushf$. If we regard the  $\RR$-linear space $\HHH^1_c(\Mm;\RR)$ as a discrete additive group, then taking the $\RR$-span of a given subgroup of it may be seen as a transcendental procedure. For this reason, before we obtain the  $\RR$-bilinearity of $\bb_{\mushf}$, we had not succeeded in establishing Theorems~\ref{thm=extendabilityRR} or \ref{thm=ReznikovRR}.

\subsection{Applications of the $\RR$-bilinearity in the setting of  other invariant quasimorphisms}\label{subsec=application_RR_other}
In this subsection, we exhibit two applications of the $\RR$-bilinearity of bilinear forms not constructed from $\mushf$, but from certain invariant quasimorphisms.  The first application treats Calabi quasimorphisms for the case where $(\Mm,\omega)$ is a closed symplectic surface of genus at least two; by Remark~\ref{rem=pi1trivial} we then have $\tHamc(\Mm,\omega)=\Hamc(\Mm,\omega)$). The precise definitions around Calabi quasimorphisms will be given in Subsection~\ref{subsec=Calabi}. In this case, the $\RR$-span of Calabi quasimorphisms is known to be infinite-dimensional; see \cite{Py06} or \cite[Theorem 5]{Bra}. Nevertheless, we obtain the following result.

\begin{thm}[finite dimensionality for Calabi quasimorphisms on a surface modulo extendable quasimorphisms]\label{thm=findimCalabi}
Let $(\Mm, \omega)$ be a closed  surface equipped with a symplectic form whose genus $l$ is at least two. Let $\Gg=\Sympc_0(\Mm,\omega)$ and $\Ng=\Hamc(\Mm,\omega)$. Let $\Wv\subset \QQQ(\Ng)^{\Gg}$ be the $\RR$-span of Calabi quasimorphisms. Then we have
\[
\Rdim \Wv/(\Wv \cap i^{\ast}\QQQ(\Gg))\leq l(2l-1).
\]
\end{thm}

For the second application, we introduce the following form of continuity for $\RR$-valued functions on $\tHamc(\Mm,\omega)$. Recall that the notion of smooth isotopy from identity is defined in Subsection~\ref{subsec=notation}.

\begin{defn}[continuity via smooth isotopy]\label{continuity}
Let $(\Mm,\omega)$ be a symplectic manifold.
A function $\muf \colon \tHamc(\Mm,\omega) \to\RR$ is said to be \textit{continuous via smooth isotopy} if for every smooth isotopy $\gamma\colon [0,1] \to \Hamc(M,\omega)$ from identity, 
the function $[0,1]\to \RR;\,s \mapsto \mu\left(\left[\{\gamma(t)\}_{t\in[0,s]}\right]\right)$ is continuous. 
\end{defn}

Due to the `infinite dimensionality' of the group $\tHamc(\Mm,\omega)$, not every element in $\QQQ(\tHamc(\Mm,\omega))^{\tSympc(\Mm,\omega)}$ is continuous in a reasonable topology on $\tHamc(\Mm,\omega)$. For instance, every Calabi quasimorphism on $\tHamc(\Mm,\omega)=\Hamc(\Mm,\omega)$ for a closed symplectic surface $(\Mm,\omega)$ of genus at least two is \emph{dis}continuous in the $C^0$-topology (this can be showed by a result of Entov--Polterovich--Py \cite{EPP}; see Theorem~\ref{theorem EPP} and arguments in Subsection~\ref{subsec=Calabi}). In contrast, the notion of continuity via smooth isotopy imposes continuity on `small' subsets of $\tHamc(\Mm,\omega)$ parameterized by $[0,1]$, and this is a considerably mild condition. We might expect that quasimorphisms of  geomertric origin will have this property. Then, we have the following result.

\begin{thm}[finite dimensionality of the space of quasimorphisms with continuity via smooth isotopy modulo extendable quasimorphisms]\label{thm=findimCVSI}
Let $(\Mm, \omega)$ be a closed  symplectic manifold. Let $m=\dim_{\RR}\HHH^1_c(\Mm;\RR)$. Let $\Gg=\tSympc(\Mm,\omega)$ and $\Ng=\tHamc(\Mm,\omega)$. Let $\Qv^{\mathrm{c}}_{(\Mm,\omega)}$ be the $\RR$-linear subspace of $\QQQ(\Ng)^{\Gg}$ consisting of elements that are continuous via smooth isotopy. Then we have
\[
\Rdim \Qv^{\mathrm{c}}_{(\Mm,\omega)}\left/\middle(\Qv^{\mathrm{c}}_{(\Mm,\omega)} \cap i^{\ast}\QQQ(\Gg)\right)\leq \dfrac{m(m-1)}{2}.
\]
\end{thm}

\subsection{Applications to the flux homomorphism}\label{subsec=application_flux}

In a manner similar to deducing Corollary~\ref{cor=KKMM} from Theorem~\ref{thm=KKMM}, we obtain the following constraint on the fluxes of commuting two elements in $\tSympc(\Mm,\omega)$ from limit formula \eqref{eq=Shelukhin_form_limit}.

\begin{thm}[constraint on the fluxes of commuting two elements in $\tSympc(\Mm,\omega)$]\label{thm=commutingtSymp}
Let $(\Mm,\omega)$ be a closed symplectic manifold. Then for all $f,g\in \tSympc(\Mm,\omega)$ with $fg=gf$, we have
\[
\bb_{\mushf}(\tflux_{\omega}(f),\tflux_{\omega}(g))=0.
\]
\end{thm}

In fact, we also have the following refined result on fluxes of commuting two elements in $\Sympc_0(\Mm,\omega)$. Here, the homomorphism $I_{c_1}\colon \pi_1(\Hamc(\Mm,\omega))\to \RR$ will be defined by \eqref{eq=I_{c_1}} in Subsection~\ref{subsec:sh_qm}.

\begin{thm}[constraint on the fluxes of commuting two elements in $\Sympc_0(\Mm,\omega)$]\label{thm=commutingSymp}
Let $(\Mm,\omega)$ be a closed symplectic manifold. Let $\qqm\colon \tSympc(\Mm,\omega)\twoheadrightarrow \Sympc_0(\Mm,\omega)$ be the universal covering map. Then for all $f,g\in \tSympc(\Mm,\omega)$ with $\qqm(f) \qqm(g)=\qqm(g)\qqm(f)$ in $\Sympc_0(\Mm,\omega)$, we have
\begin{equation}\label{eq=commutingcenter}
\bb_{\mushf}(\tflux_{\omega}(f),\tflux_{\omega}(g))=I_{c_1}([f,g]).
\end{equation}
In particular, $\bb_{\mushf}(\tflux_{\omega}(f),\tflux_{\omega}(g))\in I_{c_1}(\pi_1(\Hamc(\Mm,\omega)))$.
\end{thm}

By taking the contrapositions, we may regard Theorem~\ref{thm=commutingtSymp} and Theorem~\ref{thm=commutingSymp} as obstructions to commutativity in $\tSympc(\Mm,\omega)$ and in $\Sympc_0(\Mm,\omega)$, respectively. We will discuss them in Corollary~\ref{cor=obs_flux} in detail.

\subsection{Applications of the explicit expression of $\bb_{\mushf}$}\label{subsec=application_expression}

As is mentioned in Subsection~\ref{subsec=RousseauPy}, Py's Calabi quasimorphism $\mupyf\colon \tHamc(\Mm,\omega)\to \RR$ is defined for  a closed symplectic surface $(\Mm,\omega)$ of genus at least two. Then, we have the following relation between $\mushf=\mushf^{\Mm}$ and $\mupyf$.

\begin{thm}\label{thm=ShelukhinPy}
Let $(\Mm,\omega)$ be a closed surface  equipped with a symplectic form whose genus $l$ is at least two. Then, 
\[
\mushf-\frac{2-2l}{\mathrm{Area}(\Mm,\omega)}\mupyf\in i^{\ast}\QQQ(\tSympc(\Mm,\omega)),
\]
that is, the quasimorphism on the left-hand side is extendable to $\tSympc(\Mm,\omega)$.
\end{thm}

\begin{rem}
Py's quasimorphism $\mathfrak{S}_{Py}$ appearing in \cite[Theorem~4]{Shelukhin} is different from Py's Calabi quasimorphism $\mupyf$, appearing in the present paper. We also remark that $\mushf-\dfrac{2-2l}{\mathrm{Area}(\Mm,\omega)}\mupyf$ in the setting of Theorem~\ref{thm=ShelukhinPy} is \emph{not} the zero map (see Remark~\ref{rem=local_Calabi}).
\end{rem}

We have a non-vanishing result for $\bb_{\mushf}$ as a corollary to Theorem~\ref{mthm=explicitShelukhin} (Corollary~\ref{cor=non-vanishing_Shelukhin}). As a \mbox{by-product}, we have a non-zero bound from below of the real dimension of an $\RR$-linear space related to that in Theorem~\ref{thm=findimCVSI} (Proposition~\ref{prop=dimatleast1}). We also apply Theorem~\ref{mthm=Reznikov} to obtain the characterization in terms of the intersection form of the triviality of the Reznikov class  for a closed symplectic surface of genus at least two (Corollary~\ref{cor=surfacetrivial}).

\subsection{Organization of the present paper}\label{subsec=organization}
Section~\ref{sec=prelim} provides preliminaries. In Section~\ref{sec=bilinear}, we introduce a general framework of constructing an ($\RR$-valued) alternating $\ZZ$-bilinear form out of an invariant quasimorphism for a group pair $(\Gg,\Ng)$ with $\Ng=[\Gg,\Gg]$. In Section~\ref{sec=Py}, we study bilinear forms constructed out of Calabi quasimorphisms on surfaces. In particular, we prove Theorems~\ref{thm=KKMM} and \ref{thm=findimCalabi}. In Section~\ref{sec=RRbilinear_Shelukhin}, we prove Theorem~\ref{mthm=RRbilinear}, Theorem~\ref{mthm=Shelukhin_extendable} and Theorem~\ref{thm=extendabilityRR} by employing the machinery built in Section~\ref{sec=bilinear}. The key here is the $\RR$-bilinearity of $\bb_{\mushf}$ (Proposition~\ref{prop=RRbilinear}). In addition, we prove Theorems~\ref{thm=findimCVSI}, \ref{thm=commutingtSymp} and \ref{thm=commutingSymp}. In Section~\ref{sec=Reznikov}, we prove Theorem~\ref{mthm=Reznikov} and Theorem~\ref{thm=ReznikovRR}. Section~\ref{sec=kari1} and Section~\ref{sec=kari2} are devoted to the proof of Theorem~\ref{mthm=explicitShelukhin}. In Section~\ref{sec=kari1}, we formulate and prove the axiomatized theorem for Theorem~\ref{mthm=explicitShelukhin} (Theorem~\ref{thm=XQ}). In Section~\ref{sec=kari2} we apply this theorem to the settings of Theorem~\ref{mthm=explicitShelukhin}; Theorem~\ref{thm=ShelukhinPy} is proved there.

\section{Preliminaries}\label{sec=prelim}

\subsection{Symplectic geometry}\label{subsec=symp}

In this subsection, we review some concepts in symplectic geometry which we will need in the subsequent sections.
For a more comprehensive introduction to this subject, we refer the reader to \cite{Ban97}, \cite{MS} and \cite{P01}. Recall from Subsection~\ref{subsec=notation} that we assume symplectic manifolds to be connected.

Let $(\Mm,\omega)$ be a  symplectic manifold.
In this subsection, we endow $\Sympc(\Mm,\omega)$ with the $C^\infty$-topology. 
For a smooth function $H\colon \Mm\to\mathbb{R}$, we define the \textit{Hamiltonian vector field} $X_H$ associated with $H$ by
\[\omega(X_H,V)=-dH(V)\text{ for every }V \in \mathcal{X}(\Mm),\]
where $\mathcal{X}(\Mm)$ is the set of smooth vector fields on $\Mm$.

For a smooth function $H\colon  [0,1] \times \Mm\to\mathbb{R}$ with compact support and for $t \in  [0,1] $, we define a function $H_t\colon \Mm\to\mathbb{R}$ by $H_t(x)=H(t,x)$.
Let $X_H^t$ denote the Hamiltonian vector field associated with $H_t$ and let $\{\varphi_H^t\}_{t\in\mathbb{R}}$ denote the isotopy generated by $X_H^t$ such that $\varphi^0=\mathrm{id}_{\Mm}$.
We set $\varphi_H=\varphi_H^1$ and $\varphi_H$ is called the \emph{Hamiltonian diffeomorphism generated by $H$}.
For a symplectic manifold $(\Mm,\omega)$, we define the group $\Hamc(\Mm,\omega)$ of Hamiltonian diffeomorphisms by
\[\Hamc(\Mm,\omega)=\{\varphi\in\mathrm{Diff}(\Mm)\;|\;\exists H\in C^\infty([0,1]\times \Mm)\text{ such that }\varphi=\varphi_H\}.\]
Then, $\Hamc(\Mm,\omega)$ is a normal subgroup of $\Sympc_0(\Mm,\omega)$.

A smooth function $H\colon  [0,1] \times \Mm\to\mathbb{R}$ is said to be \emph{normalized}  if 
\begin{itemize}
\item $\Mm$ is closed and $\int_{\Mm}H_t\omega^n=0$ for every $t\in [0,1]$, or 
\item $\Mm$ is open and $H$ has a compact support.
\end{itemize}

Let $\tSympc(\Mm,\omega)$ denote the universal covering of $\Sympc_0(\Mm,\omega)$.
We define the (symplectic) flux homomorphism $\widetilde{\flux}_\omega\colon\widetilde{\Sympc_0}(\Mm,\omega)\to \HHH_c^{1}(\Mm;\RR)$ by
\[\widetilde{\flux}_\omega([\{\psi^t\}_{t\in[0,1]}])=\int_0^1[\iota_{X_t}\omega]dt,\]
where $\{\psi^t\}_{t\in[0,1]}$ is a path in $\Sympc_0(\Mm,\omega)$ with $\psi^0=\id_{\Mm}$ and $[\{\psi^t\}_{t\in[0,1]}]$ is the element of the universal covering $\widetilde{\Sympc_0}(\Mm,\omega)$ represented by the path $\{\psi^t\}_{t\in[0,1]}$.
It is known that $\widetilde{\flux}_\omega$ is a well-defined homomorphism (\cite{Ban}, \cite{Ban97} and  \cite[Lemma 2.4.3]{O15a}).

We also define the descended flux homomorphism.
We set
\[
\Gamma_\omega=\widetilde{\flux}_\omega(\pi_1(\Sympc_0(\Mm,\omega))),
\]
which is called the \textit{symplectic flux group}.
Then, $\widetilde{\flux}_\omega \colon \widetilde{\Sympc_0}(\Mm, \omega) \to \HHH_c^{1}(\Mm ; \RR)$ induces a homomorphism $\Sympc_0(\Mm,\omega)\to \HHH_c^{1}(\Mm;\RR)/\Gamma_\omega$, which is denoted by $\flux_\omega$ and is called the descended flux homomorphism.
\begin{prop}[fundamental facts on the flux homomorphism and $\tSympc(\Mm,\omega)$]\label{survey on flux}
  Let $(\Mm,\omega)$ be a closed  symplectic manifold.
  Then, the following hold.
 \begin{enumerate}[label=\textup{(\arabic*)}]
  \item  The map $\tHamc(\Mm,\omega) \to \tSympc(\Mm,\omega)$ induced by the inclusion map $\Hamc(\Mm,\omega) \to \Sympc_0(\Mm,\omega)$ is an injective homomorphism.
  \item The flux homomorphism $\tflux_\omega \colon \tSympc(\Mm,\omega)\to \HHH^1_c(\Mm;\RR)$ is surjective.
  \item The group $\mathrm{Ker}(\tflux_\omega)$ equals $\tHamc(\Mm,\omega)$.
  \item The group $\tHamc(\Mm,\omega)$ equals the commutator subgroup $[\tSympc(\Mm,\omega),\tSympc(\Mm,\omega)]$ of $\tSympc(\Mm,\omega)$.
  \item  The center of $\tSympc(\Mm,\omega)$ is $\pi_1(\Sympc_0(\Mm,\omega))$.
  \item Assume that $(\Mm,\omega)$ is a surface of genus at least two equipped with a symplectic form. Then $\pi_1(\Sympc_0(\Mm,\omega))$ and $\pi_1(\Hamc(\Mm,\omega))$ are both trivial. In particular, $\tSympc(\Mm,\omega)$ coincides with $\Sympc_0(\Mm,\omega)$ and $\tHamc(\Mm,\omega)$ coincides with $\Hamc(\Mm,\omega)$ in this case.
\end{enumerate}
\end{prop}
For (1)--(4), see \cite{Ban}, \cite{Ban97} and {\cite[Proposition 10.18]{MS}}; for (4), see in particular \cite[Theorem~4.3.1]{Ban97}.
For (6), see \cite[Subsection 7.2]{P01}.

By Proposition \ref{survey on flux}~(1), we can regard $\tHamc(\Mm,\omega)$ as a subgroup of $\tSympc(\Mm,\omega)$; in fact, by (3) $\tHamc(\Mm,\omega)$ is the commutator subgroup of $\tSympc(\Mm,\omega)$.


If $(S,\omega)$ is a closed symplectic surface whose genus is at least two, then  $\Sympc_0(S,\omega)$ is simply connected by Proposition~\ref{survey on flux}~(6). In particular, the flux group $\Gamma_\omega$ of $(S, \omega)$  vanishes and the flux homomorphism $\flux_\omega$ is a homomorphism from $\Sympc_0(S, \omega)$ to $\HHH^1_c(S ; \RR)$.

Finally, we recall the Calabi homomorphism. 
Let $(\Mm,\omega)$ be a $2n$-dimensional  open symplectic manifold.
The \textit{Calabi homomorphism}
is a function $\mathrm{Cal}_{\Mm} \colon \tHamc(\Mm,\omega)\to\mathbb{R}$ defined by
\[
	\cal_{\Mm}(\varphi_H)=\int_0^1\left(\int_{\Mm} H_t \omega^n\right)dt,
\]
 where $H \colon [0,1] \times \Mm \to \RR$ is a normalized smooth function. 
 It is known that the Calabi homomorphism is a well-defined group homomorphism (see \cite{Cala}, \cite{Ban}, \cite{Ban97}, \cite{MS}, \cite{Hum} and \cite[Theorem 2.5.6]{O15a}).
Since $\int_{\Mm} H\omega^n=\int_{\Mm} (H\circ\psi)\omega^n$ holds for every smooth function $H\colon \Mm \to \RR$ and every $\psi \in \Sympc_0(\Mm,\omega)$,  the Calabi homomorphism is $\Sympc_0(\Mm,\omega)$-invariant.

\subsection{Invariant quasimorphisms}\label{subsec=qm}

We first recall the definition and properties of quasimorphisms. We refer the reader to \cite{Calegari} as a comprehensive treatise on quasimorphisms.

\begin{defn}\label{defn=qhom}
Let $\Gg$ be a group.
\begin{enumerate}[label=\textup{(\arabic*)}]
  \item A real-valued function $\phf \colon \Gg \to \RR$ on a group $\Gg$ is called a {\it quasimorphism} if
\[
 \DD(\phf):= \sup_{\gl_1,\gl_2 \in \Gg}|\mu(\gl_1\gl_2) - \mu (\gl_1) - \mu(\gl_2)|
\]
is finite. The constant $\DD(\phf)$ is called the {\it defect} of $\phf$.
  \item A quasimorphism $\phf$ on $\Gg$ is said to be {\it homogeneous} if
$\phf(\gl^k) = k \cdot \phf(\gl)$ for every $\gl \in \Gg$ and for every $k \in \ZZ$.
  \item We write $\QQQ(\Gg)$ for the $\RR$-linear space of homogeneous quasimorphisms on $\Gg$.
\end{enumerate}
\end{defn}
By employing the symbol $\sim_C$ as in Subsection~\ref{subsec=notation}, for $\phf\in \QQQ(\Gg)$ we have
\[
\phf(\gl_1\gl_2)\sim_{\DD(\phf)} \phf(\gl_1)+\phf(\gl_2)
\]
for all $\gl_1,\gl_2$ in $\Gg$. 
The following properties of homogeneous quasimorphisms are fundamental and easy to deduce from homogeneity. 

\begin{lem} \label{lem:qm}
  Let $\phf$ be a homogenous quasimorphism on a group $\Gg$. Then, for all $\gl_1,\gl_2 \in \Gg$, the following hold true:
  \begin{enumerate}[label=\textup{(\arabic*)}]
    \item $\phf(\gl_2\gl_1\gl_2^{-1})=\phf(\gl_1)$;
    \item if $\gl_1\gl_2=\gl_2\gl_1$, then $\phf(\gl_1\gl_2)=\phf(\gl_1)+\phf(\gl_2)$;
    \item $\left|\phf([\gl_1,\gl_2])\right| \leq  \DD(\phf).$
  \end{enumerate}
\end{lem}

 The homogeneity condition on quasimorphisms is not restrictive in the following sense. 

\begin{lem}[homogenization]\label{lem=homogenization}
Let  $\phf$ be a quasimorphism on a group $\Gg$. Then, the map $\phf_{\mathrm{h}}\colon \Gg\to \RR$ defined by
\begin{equation}\label{eq=homogenization}
\phf_{\mathrm{h}}(\gl)=\lim_{k\to \infty}\frac{\phf(\gl^k)}{k}.
\end{equation}
is well-defined. In other words, the limit on the right-hand side exists. Furthermore, this $\phf_{\mathrm{h}}$ satisfies the following two properties.
\begin{enumerate}[label=\textup{(\arabic*)}]
  \item $\phf_{\mathrm{h}}\in \QQQ(\Gg)$, namely, $\phf_{\mathrm{h}}$ is a homogeneous quasimorphism on $\Gg$.
  \item $\|\phf_{\mathrm{h}}-\phf\|_{\infty}\leq \DD(\phf)$. That is,
\[
\sup_{\gl\in \Gg} |\phf_{\mathrm{h}}(\gl)-\phf(\gl)|\leq \DD(\phf).
\]
\end{enumerate}
\end{lem}

The element $\phf_{\mathrm{h}}$ in $\QQQ(\Gg)$ defined by \eqref{eq=homogenization} is called the \emph{homogenization} of $\phf$.

Secondly, we recall the definition of quasi-invariance and invariance of quasimorphisms for the pair $(\Gg,\Ng)$ of a group and its normal subgroup.

\begin{defn}[invariant quasimorphism]\label{defn=inv_qm}
Let $\Gg$ be  a group and $\Ng$ a normal subgroup of $\Gg$.
\begin{enumerate}[label=\textup{(\arabic*)}]
 \item A real-valued function $\ff\colon \Ng\to \RR$ is said to be \emph{$\Gg$-invariant} if 
\[
\ff(\gl \xl\gl^{-1}) = \ff(\xl)
\]
holds for every $\xl \in \Ng$ and for every $\gl \in \Gg$.
  \item The $\RR$-vector space $\QQQ(\Ng)^{\Gg}$ is defined as the space of homogeneous quasimorphisms on $\Ng$ that are $\Gg$-invariant.
  \item The $\RR$-vector space $\HHH^1(\Ng)^{\Gg}$ is defined as the space of \textup{(}genuine\textup{)} homomorphisms $\Ng\to \RR$ that are $\Gg$-invariant.
\end{enumerate}
\end{defn}

Let $\Gg$ be  a group and $\Ng$ a normal subgroup of $\Gg$. By Lemma~\ref{lem:qm}~(1), we have $\QQQ(\Gg)=\QQQ(\Gg)^{\Gg}$. The inclusion map $\im\colon \Ng\hookrightarrow \Gg$ induces $\im^{\ast}\colon \QQQ(\Gg)\to \QQQ(\Ng)$; $\phf\mapsto \phf|_{\Ng}$. Thus, we may regard this map as
\[
\im^{\ast}\colon \QQQ(\Gg)=\QQQ(\Gg)^{\Gg}\to \QQQ(\Ng)^{\Gg}.
\]

The following $\RR$-linear space $\VV(\Gg,\Ng)$ had been introduced in \cite{KKMMM} without the symbol; the symbol $\VV(\Gg,\Ng)$ was later given in \cite{KKMMMsurvey}.

\begin{defn}[the space $\VV(\Gg,\Ng)$ of non-extendable quasimorphisms]\label{defn=Vspace}
Let $\Gg$ be a group and $\Ng$ a normal subgroup of $\Gg$. Let $\im\colon \Ng\hookrightarrow \Gg$ be the inclusion map.
\begin{enumerate}[label=\textup{(\arabic*)}]
  \item An element $\muf$ in $\QQQ(\Ng)^{\Gg}$ is said to be \emph{extendable} \textup{(}as an element in $\QQQ(\Gg)$\textup{)} if $\muf$ belongs to $\im^{\ast}\QQQ(\Gg)$, that is, there exists $\phf\in \QQQ(\Gg)$ such that $\muf=\im^{\ast}\phf$.
  \item The \emph{space $\VV(\Gg,\Ng)$ of non-extendable quasimorphisms} is defined as
\[
\VV(\Gg,\Ng)=\QQQ(\Ng)^{\Gg}/\im^{\ast}\QQQ(\Gg).
\]
\end{enumerate}
\end{defn}

At the end of this subsection, we state the following result from \cite{KKMMM} and \cite{KKMMMcoarse}, which will be employed in Section~\ref{sec=bilinear}.

\begin{prop}[{\cite[Lemma~7.12]{KKMMM}, \cite[Theorem 8.9]{KKMMMcoarse} and \cite[Theorem 1.10]{KKMMM}}]\label{prop=abelianiroiro}
Let $\Gg$ be a group and $\Ng$ a normal subgroup of $\Gg$. Let $\im\colon \Ng\hookrightarrow \Gg$ be the inclusion map. Set $\Gamg=\Gg/\Ng$. Assume that $\Gamg$ is abelian. Then, the following hold true.
\begin{enumerate}[label=\textup{(\arabic*)}]
 \item For every $\phf\in \QQQ(\Gg)$, we have $\DD(\phf)=\DD(\im^{\ast}\phf)$.
 \item Let $\muf\in \QQQ(\Ng)^{\Gg}$, Then, $\muf$ is extendable 
 to $\Gg$ if and only if
\[
\sup_{\gl_1,\gl_2\in \Gg}\left|\muf([\gl_1,\gl_2])\right|<\infty.
\]
 \item Assume that $\HHH^2(\Gg)=0$. Then,
\[
\QQQ(\Ng)^{\Gg}=\HHH^1(\Ng)^{\Gg}+\im^{\ast}\QQQ(\Gg).
\]
\end{enumerate}
\end{prop}
Here, in (1) $\DD(\im^{\ast}\phf)$ is the defect of the quasimorphism $\im^{\ast}\phf$ on $\Ng$. For (2), we note that $\Ng\geqslant [\Gg,\Gg]$ since $\Gamg$ is abelian; hence for all $\gl_1,\gl_2\in \Gg$, $[\gl_1,\gl_2]\in \Ng$. We also remark that \cite[Theorem 1.10]{KKMMM} implies (3) because $\Gamg=\Gg/\Ng$ is abelian and in particular boundedly $3$-acyclic in this setting.

\subsection{Shelukhin's quasimorphism} \label{subsec:sh_qm}
 Py's Calabi quasimorphism, constructed in \cite{Py06}, was studied in the previous work \cite{KK} and \cite{KKMM2} of some of the authors. We refer the reader to \cite[Section 8]{Rosenberg} and \cite[Subsection~2.4]{KKMM2} for an outlined construction. What we use in the present paper is Proposition~\ref{prop=KKMM}, which will be stated in Section~\ref{sec=Py}.

One of the main subjects in this paper is the $\tSympc(\Mm,\omega)$-invariant homogeneous quasimorphism on $\tHamc(\Mm,\omega)$ constructed by Shelukhin \cite{Shelukhin}. 
For a $2n$-dimensional symplectic manifold $(\Mm^{2n},\omega)$ of finite volume, Shelukhin  constructed a homogeneous quasimorphism $\mathfrak{S}_{\Mm}$ on $\tHamc(\Mm, \omega)$. As we mentioned in the introduction, we write $\mushf$ for $\mathfrak{S}_{\Mm}$ in the present paper. 
In this subsection, we briefly recall the definition and properties of Shelukhin's quasimorphism $\mushf$.

Let $\mathcal{S} \to \Mm$ be the fiber bundle whose fiber over $\xp \in \Mm$ is the space of $\omega_{\xp}$-compatible complex structures on the tangent space $T_{\xp} \Mm$.
Let $\mathcal{J}$ be the space of global sections of $\mathcal{S} \to \Mm$, that is, the space of $\omega$-compatible almost complex structures on $\Mm$.
Note that the space $\mathcal{J}$ is contractible since the fibers of $\mathcal{S} \to \Mm$ are homeomorphic to the contractible space $\Sp(2n, \RR)/\UU(n)$.
The canonical $\Sp(2n, \RR)$-invariant K\"{a}hler form, which is induced from the trace, defines the fiberwise-K\"{a}hler form $\sigma$ on $\mathfrak{S}$.
If $\Mm$ is closed,  then  a form $\Omega$ on $\mathcal{J}$ is defined by $\Omega(A, B) := \int_{\Mm}\sigma_{\xp}(A_{\xp}, B_{\xp}) \omega^n(\xp)$.
Note that the $\Hamc(\Mm, \omega)$-action on $\mathcal{J}$ preserves $\Omega$ (and indeed, so does the $\Sympc_0(\Mm, \omega)$-action).
We also note that this fiberwise-K\"{a}hler form $\sigma$ induces the metric structure on $\mathcal{J}$.
For every two points in $\mathcal{J}$, the geodesic between the two points is uniquely determined.


The $\Hamc(\Mm,\omega)$-action on $(\mathcal{J}, \Omega)$ is Hamiltonian (\cite{MR1622931} and \cite{MR1207204}), and its moment map $\mom$ is given as follows.
For an $\omega$-compatible almost complex structure $J \in \mathcal{J}$, let $h$ be the Hermitian metric induced from $J$ and $\omega$.
Let $\nabla$ be the second canonical connection of $(M,J,h)$, that is, the connection uniquely defined by the three conditions that $\nabla J = 0$, $\nabla h = 0$, and the $(1,1)$-part of $\nabla$ is everywhere vanishing (\cite[Section 2.6]{MR1456265}).
Let $\rho \in \Omega^2(M;\RR)$ be the $i$ times the curvature of the connection of the Hermitian line bundle $\Lambda_{\CC}^n(TM)$ induced from $\nabla$.
Then the Hermitian scalar curvature $S(J) \in C^{\infty}(\Mm,\RR)$ is defined by
\[
  S(J)\omega^n = n\rho\wedge\omega^{n-1},
\]
and the moment map $\mom \colon C_{\rm normal}^{\infty}(\Mm,\RR) \to C^{\infty}(\mathcal{J},\RR)$ is given by
\[
  (\mom(H))(J) = \int_{M}S(J)H\omega^n.
\]
Here $C_{\rm normal}^{\infty}(\Mm,\RR)$  is the space of normalized Hamiltonian functions, which is the Lie algebra of $\Hamc(\Mm,\omega)$.

For a symplectic manifold $(\Mm,\omega)$ of finite volume, the average Hermitian scalar curvature $A(\Mm,\omega)$ is defined as follows. 

\begin{defn}[average Hermitian scalar curvature]\label{defn=aHsc}
Let  $(\Mm,\omega)$ be a symplectic manifold of finite volume. Then the \emph{average Hermitian scalar curvature} is defined by
\begin{equation}\label{eq:aHsc}
A(\Mm,\omega)= \left. \int_{\Mm} S(J) \omega^n \middle/ \int_{\Mm} \omega^{n} \right.
= \left. n \int_{\Mm} c_1(\Mm) \omega^{n-1} \middle/ \int_{\Mm} \omega^{n}. \right. 
\end{equation}
Here, $c_1(\Mm) \in \HHH^2(M;\RR)$ denotes the first Chern class of $(\Mm,\omega)$.
\end{defn}
Here, we recall that the first Chern class can be defined for a symplectic manifold (for example, see \cite[Subsection 2.6]{MS} and \cite[Remark 2.4.6]{Ge}).

The geodesic from $J_0$ to $J_1$ will be denoted by $[J_0, J_1]$.
For all $J_0, J_1, J_2$ in $\mathcal{J}$, let $\Delta(J_0, J_1, J_2)$ denote the $2$-simplex whose restriction to each fiber $\mathcal{S}_{\xp}$ over $\xp \in \Mm$ is the geodesic triangle $\Delta((J_0)_{\xp}, (J_1)_{\xp}, (J_2)_{\xp})$ with respect to the K\"{a}hler form $\sigma_{\xp}$ with vertices $(J_0)_{\xp}$, $(J_1)_{\xp}$, and $(J_2)_{\xp}$.

We are now ready to define Shelukhin's quasimorphism $\mushf$ on $\tHamc(\Mm,\omega)$.
For $\tg \in \tHamc(\Mm,\omega)$, let $\{ h_t \}_{t \in [0,1]}$ be a path 
in $\Hamc(\Mm,\omega)$ representing $\tg$, and set $h=h_1$. 
Let $J \in \mathcal{J}$ be a basepoint and $D$ a disk in $\mathcal{J}$ whose boundary is equal to the loop $\{ h_t \cdot J \}_{t \in [0,1]} \ast [h \cdot J, J]$, where $\ast$ denotes concatenation.
Let $\{ H_t \}_{t \in [0,1]}$ be the path in 
$ C_{\rm normal}^{\infty}(\Mm,\RR)$ corresponding to the path $\{ h_t \}_{t \in [0,1]}$.
Then we define $\nuf_{J} \colon \tHamc(\Mm,\omega) \to \RR$ by
\begin{align}\label{before_homogenize}
  \nu_J(\tg) := \int_D \Omega - \int_0^1 \mom(H_t)(h_t\cdot J) dt.
\end{align}
This function $\nuf_{J}$ is a quasimorphism since the equality
\begin{align*}
  \nu_J(\widetilde{g} \widetilde{h}) - \nu_J(\widetilde{g}) - \nu_J(\widetilde{h}) = \int_{\Delta(J, gJ, ghJ)} \Omega
\end{align*}
and the inequality
\begin{align}\label{bounded_Reznikov}
  \left| \int_{\Delta(J, gJ, ghJ)} \Omega \right|
  = \left| \int_{\Mm} \middle( \int_{\Delta((J_\xp), (gJ)_{\xp}, (ghJ)_{\xp} )} \sigma_{\xp} \middle) \omega^n \right| \leq \vol(\Mm,\omega^n) \cdot C
\end{align}
hold, where $C$ is the area of an ideal triangle of $\Sp(2n,\RR)/\UU(n)$ with respect to the canonical K\"{a}hler form.

\begin{defn}[{\cite{Shelukhin}}]
Let $\Mm$ be a symplectic manifold of finite volume.  Then, the homogeneous quasimorphism
  \[
    \mushf \colon \tHamc(\Mm,\omega) \to \RR
  \]
  is defined as the homogenization (recall \eqref{eq=homogenization} in Lemma~\ref{lem=homogenization}) of $\nuf_J$; we call $\mushf$ \emph{Shelukhin's quasimorphism}.
\end{defn}

The quasimorphism $\nuf_J$ may depend on the choice of $J$. However, in \cite[Theorem 1]{Shelukhin}, it is shown that the homogenization $\mushf$ does not depend on the choice of $J$.

Let $(\Mm,\omega)$ be a closed symplectic manifold. Recall from Proposition~\ref{survey on flux}~(5) that the center of $\tSympc(\Mm,\omega)$ equals $\pi_1(\Sympc_0(\Mm,\omega))$. Shelukhin determined the values of $\mushf$ on $\tHamc(\Mm,\omega)\cap\pi_1(\Sympc_0(\Mm,\omega))=\pi_1(\Hamc(\Mm,\omega))$.
To state this, let us recall the homomorphism $I_{c_1} \colon \pi_1(\Hamc(\Mm,\omega)) \to \RR$ introduced in \cite{MR1666763}. 

First note that $\pi_1(\Hamc(M,\omega))$ is bijective to the set of isomorphism classes of fiber bundles $P \to S^2$ whose fiber is $\Mm$ and whose structure group is contained in $\Hamc(\Mm,\omega)$.
Under this bijection, for $\gamma \in \pi_1(\Hamc(\Mm,\omega))$, let $P_{\gamma} \to S^2$ denote the corresponding fiber bundle.
Then the vertical tangent bundle $T_{V} P \to P$ is a symplectic vector bundle, and hence we have the first Chern class $c_1^{V}$ of $T_{V} P \to P$.
There is another characteristic class $u \in \HHH^2(P;\RR)$, which is defined as a cohomology class satisfying $u^{n+1} = 0$ and $u|_{\text{fiber}} = [\omega]$.
Then the homomorphism $I_{c_1}$ is defined by
\begin{equation}\label{eq=I_{c_1}}
  I_{c_1}(\gamma) = \int_{P_{\gamma}} c_1^{V} u^n.
\end{equation}

\begin{prop}[{\cite[Corollary~2]{Shelukhin}}]\label{prop=ShelukhinHam}
Let $(\Mm,\omega)$ be a closed symplectic manifold. Then, the restriction of Shelukhin's quasimorphism $\mushf$ to $\pi_1(\Hamc(\Mm,\omega))$ equals $I_{c_1}$.
\end{prop}

We also recall from the introduction that $\Sympc(\Mm,\omega)$ acts on $\tHamc(\Mm,\omega)$ by \eqref{eq=Sympc-action}. 

\begin{prop}[{\cite[Proposition 1.1]{Shelukhin}}]\label{prop=Sympc-inv}
Let $(\Mm,\omega)$ be a closed symplectic manifold. Then, Shelukhin's quasimorphism $\mushf$ is $\Sympc(\Mm,\omega)$-invariant, namely, for every $\theta\in \Sympc(\Mm,\omega)$ and for every $\psi\in \tHamc(\Mm,\omega)$, we have
\[
\mushf(\theta\cdot \psi)=\mushf(\psi).
\]
\end{prop}

\subsection{Reznikov class}\label{subsec:re_class}
In the context of characteristic classes of foliated bundles such as the Bott--Thurston cocycles, Reznikov \cite{Reznikov} introduced certain cocycles of symplectomorphism groups.
One of the cocycles is defined as follows.
Let $(M, \omega), \Omega$, $J$, $\Delta(J_0, J_1, J_2)$ be as in Subsection \ref{subsec:sh_qm}.
\begin{definition}
  For $g,h \in \Sympc(M,\omega)$, define a two-cocycle $b_J$ on $\Sympc(M,\omega)$ by 
  \[
    b_J (g,h) = \int_{\Delta(J, gJ, ghJ)} \Omega.
  \] 
  The cohomology class $R = [b_J] \in \HHH^2(\Sympc(M,\omega))$ is called the \textit{Reznikov class}.
  We write $R|_G \in \HHH^2(G)$ for the class restricted to a subgroup $G$ of $\Sympc(M,\omega)$.
\end{definition}

The cohomology class $R$ of $\Sympc(M,\omega)$ is independent of the choice of $J$.
It is proved in \cite{Reznikov} that the class $R$ is non-zero when $M$ is a high-dimensional symplectic torus and a certain connected component of the smooth moduli space $\mathrm{Hom}(\pi_1(S), \mathrm{SO}(3))/\mathrm{SO}(3)$ of stable vector bundles over a Riemann surface $S$ of genus at least two.


\subsection{Symplectic blow-up}\label{subsec:blowup}

In this subsection, we review a blow-up of a symplectic manifold; throughout this subsection, we assume that the dimension of a symplectic manifold is larger than 2.
Before that, we review a blow-up of a topological manifold.

We set 
\[
  B^{2n}(R) = \{ v = (z_1, \cdots, z_n) \in \CC^{n} \mid \| v \| < R \},
\] 
where $\| v\| = \sqrt{|z_1|^2 + \cdots + |z_n|^2}$.
Let $\Mm$ be a $2n$-dimensional topological manifold and $\iota \colon B^{2n}(R) \to \Mm$ a topological embedding.
We set $\xl_0=\iota(0)$.
For every real number $r$ with $ 0 <  r < R$, we set
\[U_r=\{
(\vl,[\wl]) \in \CC^n \times \CC P^{n-1}\, | \, w\in \CC^n\setminus\{0\}, v\in \CC w, \|v \| < r
\}.\]
Here, $[w]$ denotes the image of $w$ through the natural projection $\CC^n\setminus\{0\} \to \CC P^{n-1}$. 
Then, we define the \emph{blow-up} $\hat{\Mm}_\iota$ of $\Mm$ (with respect to $\iota$) to  be
\[\hat{\Mm}_\iota=(\Mm \setminus \{\xl_0\}) \sqcup U_r/\sim,\]
where the equivalence relation $\sim$ identifies a point $(\zl,[\wl]) \in U_r$ with $\iota(\zl) \in \Mm$ for each $\zl \neq 0$ .

Let $\iota_1 \colon \Mm \setminus \{ \xl_0\} \to \hat{\Mm}_\iota$ and $\iota_2 \colon U_r \to \hat{\Mm}_\iota$ be the inclusions.
More precisely, $\iota_1$ is the composition of the map $\Mm \setminus \{ \xl_0\} \to (\Mm \setminus \{ \xl_0\}) \sqcup U_r$ and the quotient map $(\Mm \setminus \{ \xl_0\}) \sqcup U_r \to \hat{\Mm}_\iota$, and $\iota_2$ is the composition of the map $U_r \to (\Mm \setminus \{ \xl_0\}) \sqcup U_r$ and the quotient map $(\Mm \setminus \{ \xl_0\}) \sqcup U_r \to \hat{\Mm}_\iota$.

We define the map $\hat\iota \colon \CC P^{n-1} \to \hat{\Mm}_\iota$ by
\[\hat\iota ([\wl])=\iota_2(0,[\wl]).\]
Let $\pi \colon \hat{\Mm}_\iota \to \Mm$ be the projection.
Then, the following proposition is known.
\begin{prop}[for example, see {\cite[Proposition 9.3.3]{MS12}}]\label{blow_up_def}
Let $(\Mm,\omega)$ be a symplectic manifold, $\iota \colon B^{2n}(r) \to \Mm$ be an open symplectic embedding, and $J \in \mathcal{J}(\Mm,\omega)$ such that $\iota^\ast J=J_0$, where $J_0$ is the standard complex structure on $B^{2n}(r)$.
Then, for every $ 0 < \rho < r$, there exists a symplectic form $\omega_\rho$ on the blow-up $\hat{\Mm}_\iota$ with the following properties.
\begin{enumerate}[label=\textup{(\arabic*)}]
 \item The $2$-form $\pi^\ast \omega$ corresponds to $\omega_\rho$ on $\pi^{-1}(\Mm \setminus \iota(B^{2n}(r)))$.
 \item $\hat{\iota}^\ast \omega_\rho=\rho^2\omega_{FS}$, where $\omega_{FS}$ is the Fubini--Study form on $\CC P^n$.
\end{enumerate}
Here, under our convention, we set $\omega_{FS}([\CC P^1])=1$.
\end{prop}

The symplectic manifold $(\hat{\Mm}_\iota, \omega_\rho)$ appearing in Proposition \ref{blow_up_def} is called the \emph{blow-up of $(\Mm,\omega)$} (with respect to $\iota$, $J$ and $\rho$).

Let $r_1,\ldots,r_n$ be real numbers with $0<r_1<r_2<\ldots<r_n$.
Then, let $T^{2n}(r_1,\ldots,r_n)$ denote the torus $(\RR/2r_1\ZZ)^2 \times \cdots \times (\RR/2r_n\ZZ)^2$ with coordinates $(\xl_1, \yl_1, \xl_2, \yl_2, \cdots, \xl_n, \yl_n)$ and $J_0$ be the standard complex structure of $T^{2n}(r_1,\ldots,r_n)$.
We define the standard symplectic form $\omega_0$ of $T^{2n}(r_1, \ldots, r_n)$ by $\omega_0=d\xl_1\wedge d\yl_1+\cdots+d\xl_n\wedge d\yl_n$.
For a real number $r$ with $ 0 <  r < r_1$, let $\iota(r) \colon B^{2n}(r) \to T^{2n}(r_1,\ldots,r_n)$ denote an embedding defined by $\iota(r)(\xl_1, \yl_1, \xl_2, \yl_2, \ldots, \xl_n,\yl_n) = (\xl_1, \yl_1, \xl_2, \yl_2, \ldots, \xl_n, \yl_n)$.
The following lemma is proved in \cite{KKMMMReznikov}.

\begin{lem}[{\cite[Lemma 2.16]{KKMMMReznikov}}]\label{torus blow up lemma}
Let $r_1,\ldots,r_n$ be real numbers with $0<r_1<r_2<\cdots<r_n$.
For real numbers $\rho,r$ with $ 0 < \rho<r < r_1$,
let $(\hat{\Mm}_r,\omega_\rho)$ be a blow-up of $(T^{2n}(r_1,\ldots,r_n),\omega_0)$ with respect to $\iota(r)$, $J_0$ and $\rho$.
Then, $A(\hat{\Mm}_r,\omega_\rho) \neq 0$.
\end{lem}

\section{General recipe of constructing $\ZZ$-bilinear forms}\label{sec=bilinear}

In this section, except Examples~\ref{exa=symplectic} and \ref{exa=Reznikovsetting} we use the symbol $\Mg$ for a group (not for a symplectic manifold); for instance one as in diagram \eqref{diagram=core}. As we mentioned in Subsection~\ref{subsec=outline}, we build general machinery to construct an ($\RR$-valued) $\ZZ$-bilinear form out of an invariant quasimorphism for a group pair $(\Gg,\Ng)$ with $\Ng=[\Gg,\Gg]$. 

\subsection{The alternating $\ZZ$-bilinear form $\bb_{\muf}$ and the limit formulae}\label{subsec=limit_formula}

For an abelian group $\Gamg$, the notion of an ($\RR$-valued) \emph{alternating $\ZZ$-bilinear form on $\Gamg$} makes sense in a standard manner because $\Gamg$ is naturally equipped with the structure of a $\ZZ$-module. An equivalent formulation of this notion is a map $\bb\colon \Gamg\times \Gamg\to \RR$ that satisfies the following two conditions for all $\gaml_1,\gaml_2,\gaml\in \Gamg$:
\begin{enumerate}[label=\textup{(\arabic*)}]
  \item $\bb(\gaml_1+\gaml_2,\gaml)=\bb(\gaml_1,\gaml)+\bb(\gaml_2,\gaml)$,
  \item $\bb(\gaml_2,\gaml_1)=-\bb(\gaml_1,\gaml_2)$.
\end{enumerate}

In this section, we study  $(\Gg,\Ng)$, where $\Gg$ is a group and $\Ng=[\Gg,\Gg]$ is the commutator subgroup of $G$. For such a pair $(\Gg,\Ng)$, we set $\Gamg=\Gg/\Ng$, the abelianization of $\Gg$. These settings are summarized by the following short exact sequence
\begin{equation}\label{eq=shortex}
1 \longrightarrow \Ng=[\Gg,\Gg] \stackrel{\im}{\longrightarrow} \Gg \stackrel{\ppm}{\longrightarrow} \Gamg=\Gg/\Ng \longrightarrow 1.
\end{equation}
Our motivating example is the following.
\begin{exa}\label{exa=symplectic}
Let $(\Mm,\omega)$ be a closed symplectic manifold. Set $\Gg=\tSympc(\Mm,\omega)$ and $\Ng=\tHamc(\Mm,\omega)$. Then, $\Ng=[\Gg,\Gg]$ (recall Proposition~\ref{survey on flux}). In this case, $\Gamg=\Gg/\Ng$ is identical to $\HHH^1_c(\Mm;\RR)$, and the group quotient map $\ppm\colon \Gg\twoheadrightarrow \Gamg$ is the flux homomorphism
\[
\tflux_{\omega}\colon \tSympc(\Mm,\omega)\twoheadrightarrow \HHH^1_c(\Mm;\RR).
\]
\end{exa}

Recall that the space $\QQQ(\Ng)^{\Gg}$ is defined in Definition~\ref{defn=inv_qm}, and that `$a\sim_C b$' means $|a-b|\leq C$ (Subsection~\ref{subsec=notation}) for $a,b\in \RR$ and $C\in \RR_{\geq 0}$.

\begin{thm}[constructing an alternating $\ZZ$-bilinear form out of an invariant quasimorphism]\label{thm=bilinear_main}
Let $(\Gg,\Ng,\Gamg)$ be a triple fitting in \eqref{eq=shortex}. Let $\muf\in \QQQ(\Ng)^{\Gg}$. Then, there exists an alternating $\ZZ$-bilinear form $\bb_{\muf}$ on $\Gamg$ such that for all $\gl_1,\gl_2\in \Gg$,
\begin{equation}\label{eq=bilinear_main}
\muf([\gl_1,\gl_2])\sim_{\DD(\muf)} \bb_{\muf}(\ppm(\gl_1),\ppm(\gl_2)).
\end{equation}
\end{thm}

Once the existence of an alternating $\ZZ$-bilinear form $\bb_{\muf}$ satisfying \eqref{eq=bilinear_main} on $\Gamg$ is ensured, we can determine the form $\bb_{\muf}$ by  the following \emph{limit formula}.

\begin{thm}[limit formula]\label{thm=limit_formula}
Let $(\Gg,\Ng,\Gamg)$ be a triple fitting in \eqref{eq=shortex}. Let $\muf\in \QQQ(\Ng)^{\Gg}$. Then, for all $\gaml_1,\gaml_2\in \Gamg$, we have
\begin{equation}\label{eq=limit_formula}
\lim_{k\to\infty}\frac{\muf([\gl_1^k,\gl_2])}{k}=\bb_{\muf}(\gaml_1,\gaml_2)=\lim_{l\to\infty}\frac{\muf([\gl_1,\gl_2^l])}{l}.
\end{equation}
Here, $\gl_1,\gl_2$ are arbitrary elements in $\Gg$ satisfying $\ppm(\gl_1)=\gaml_1$ and $\ppm(\gl_2)=\gaml_2$.
\end{thm}

In fact, Theorem~\ref{thm=bilinear_main} produces the following refined limit formula.

\begin{thm}[refined limit formula]\label{thm=refined_limit_formula}
Let $(\Gg,\Ng,\Gamg)$ be a triple fitting in \eqref{eq=shortex}. Let $\muf\in \QQQ(\Ng)^{\Gg}$. Let $\gaml_1,\gaml_2\in \Gamg$. Let $(\gl_1^{(k)})_{k\in \NN}$ be a sequence in $\Gg$ satisfying $\ppm(\gl_1^{(k)})=k\gaml_1$ for every $k\in \NN$. Let $\gl_2\in \Gg$ with $\ppm(\gl_2)=\gaml_2$. Then, 
\begin{equation}\label{eq=refined_limit_formula_left}
\lim_{k\to\infty}\frac{\muf([\gl_1^{(k)},\gl_2])}{k}=\bb_{\muf}(\gaml_1,\gaml_2).
\end{equation}
Let $(\gl_2^{(l)})_{l\in \NN}$ be a sequence in $\Gg$ satisfying $\ppm(\gl_2^{(l)})=l\gaml_2$ for every $l\in \NN$. Let $\gl_1\in \Gg$ with $\ppm(\gl_1)=\gaml_1$. Then, 
\begin{equation}\label{eq=refined_limit_formula_right}
\lim_{l\to\infty}\frac{\muf([\gl_1,\gl_2^{(l)}])}{l}=\bb_{\muf}(\gaml_1,\gaml_2).
\end{equation}
\end{thm}

\begin{proof}[Proofs of Theorems~$\ref{thm=limit_formula}$ and $\ref{thm=refined_limit_formula}$ modulo Theorem~$\ref{thm=bilinear_main}$]
Let $\gaml_1,\gaml_2\in\Gamg$. Let $(\gl_1^{(k)})_{k\in \NN}$ be a sequence in $\Gg$ satisfying $\ppm(\gl_1^{(k)})=k\gaml_1$ for every $k\in \NN$. Let $\gl_2\in \Gg$ with $\ppm(\gl_2)=\gaml_2$. Then under \eqref{eq=bilinear_main}, we have
\[
\muf([\gl_1^{(k)},\gl_2])\sim_{\DD(\muf)} \bb_{\muf}(k\gaml_1,\gaml_2)=k\bb_{\muf}(\gaml_1,\gaml_2).
\]
Here, recall that $\bb_{\muf}$ is $\ZZ$-bilinear. By dividing the formula above by $k$ and letting $k\to \infty$, we obtain \eqref{eq=refined_limit_formula_left}. We can show \eqref{eq=refined_limit_formula_right} in a  manner similar to one above. Then, \eqref{eq=refined_limit_formula_left} and \eqref{eq=refined_limit_formula_right} imply \eqref{eq=limit_formula}.
\end{proof}

It remains to prove Theorem~\ref{thm=bilinear_main}. In Subsection~\ref{subsec=bilinear_const}, we explain the construction of $\bb_{\muf}$, and present the proof of Theorem~\ref{thm=bilinear_main}.

\subsection{Construction of the alternating $\ZZ$-bilinear form}\label{subsec=bilinear_const}

The main idea for the construction of the alternating $\ZZ$-bilinear form $\bb_{\muf}$ in Theorem~\ref{thm=bilinear_main} is to \emph{lift $\Gg$ to a free group $\Fg$}; this was one of the key ideas in \cite[Section~7]{KKMMMcoarse}. We use the following setting in this subsection.

\begin{setting}\label{setting=GNFM}
Let $\Gg$ be a group and $\Ng=[\Gg,\Gg]$. Let $\im\colon\Ng\hookrightarrow \Gg$ be the inclusion map. Fix a free group $\Fg$ and a surjective group homomorphism $\pim\colon \Fg\twoheadrightarrow \Gg$, and set $\Rel=\Ker(\pim)$. Set $\Mg=[\Fg,\Fg]$, and let $\jm\colon \Mg\hookrightarrow \Fg$ be the inclusion map.
\end{setting}

We note that such a pair $(\Fg,\pim)$ always exists: for instance, view $\Gg$ as a set, and let $\Fg$ be the free group over the set $\Gg$. We also remark that $\pim\colon \Fg\twoheadrightarrow \Gg$ induces a surjective group homomorphism $\Mg\twoheadrightarrow \Ng$, which we also write $\pim$ by abuse of notation. We employ the following result.

\begin{prop}\label{prop=decomposition}
Let $\Fg$ be a free group and $\Mg=[\Fg,\Fg]$. Let $\jm\colon \Mg\hookrightarrow \Fg$ be the inclusion map. Then the following hold.
\begin{enumerate}[label=\textup{(\arabic*)}]
  \item For every $\phf\in \QQQ(\Fg)$, $\DD(\jm^{\ast}\phf)=\DD(\phf)$.
  \item We have the following direct-sum decomposition:
\[
\QQQ(\Mg)^{\Fg}=\HHH^1(\Mg)^{\Fg}\oplus \jm^{\ast}\QQQ(\Fg).
\]
\end{enumerate}
\end{prop}

\begin{proof}
Observe that $\Fg/\Mg$ is abelian. Now, item~(1) follows from Proposition~\ref{prop=abelianiroiro}~(1). 

We will prove item (2). Since $\Fg$ is a free group, $\HHH^2(\Fg)=0$. Then, Proposition~\ref{prop=abelianiroiro}~(3) implies that
\[
\QQQ(\Mg)^{\Fg}=\HHH^1(\Mg)^{\Fg}+ \jm^{\ast}\QQQ(\Fg).
\]
Finally, we show the sum above  is a direct sum. Let $\nuf\in \HHH^1(\Mg)^{\Fg} \cap \jm^{\ast}\QQQ(\Fg)$. Take $\phf\in \QQQ(\Fg)$ such that $\jm^{\ast}\phf=\nuf$. Then, by item (1) we have $\DD(\phf)=\DD(\phf|_{\Mg})=\DD(\nuf)$, which equals $0$ since $\nuf\in \HHH^1(\Mg)^{\Fg}$. Hence, $\phf\in \HHH^1(\Fg)$. Since $\Mg=[\Fg,\Fg]$, this implies that $\nuf\equiv 0$, as desired.
\end{proof}

In Subsection~\ref{subsec=Bb}, we will also employ the following result.

\begin{prop}\label{prop=ore}
Assume Setting~$\ref{setting=GNFM}$. Then for $\muf\in \QQQ(\Ng)^{\Gg}$, the following are equivalent.
\begin{enumerate}[label=\textup{(\roman*)}]
 \item $\muf\in \im^{\ast}\QQQ(\Gg)$;
 \item $\pim^{\ast}\muf \in \jm^{\ast}\QQQ(\Fg)$.
\end{enumerate}
\end{prop}

In Proposition~\ref{prop=ore}, `(i) implies (ii)' is straightforward. The converse: `(ii) implies (i)' is non-trivial. This implication is formulated as  the \emph{descending condition $(D)$} in \cite[Definition~7.4]{KKMMMcoarse} to more general setting. 

\begin{proof}[Proof of Proposition~\textup{\ref{prop=ore}}]
We appeal to Proposition~\ref{prop=abelianiroiro}~(2). Then, (ii) holds if and only if 
\[
\sup_{\fl_1,\fl_2\in \Fg}\left|\pim^{\ast}\muf([\fl_1,\fl_2])\right|<\infty
\]
holds. Again by Proposition~\ref{prop=abelianiroiro}~(2), the last condition is equivalent to (i).
\end{proof}

Under Setting~\ref{setting=GNFM}, let $\muf\in \QQQ(\Ng)^{\Gg}$. Then, set $\hmuf=\pim^{\ast}\muf$; it lies in $\QQQ(\Mg)^{\Fg}$. By Proposition~\ref{prop=decomposition}, $\hmuf$ admits a decomposition
\begin{equation}\label{eq=core}
\hmuf=\hf+\jm^{\ast}\phf.
\end{equation}
Here $\hf\in \HHH^1(\Mg)^{\Fg}$ and $\phf\in \QQQ(\Fg)$; the pair $(\hf,\jm^{\ast}\phf)$ is uniquely determined by $\hmuf$. We illustrate this argument by the following diagram.

\begin{equation}\label{diagram=core}
 \xymatrix@C=36pt{
\Fg \ar@{->>}[d]_{\pim} \ar@{}[r]|*{\geqslant}  &  \Mg=[\Fg,\Fg] \ar@{->>}[d]_{\pim}  \\
\Gg\cong \Fg/\Rel \ar@{}[r]|*{\geqslant}  &  \Ng=[\Gg,\Gg],
} \qquad 
\xymatrix{
\QQQ(\Mg)^{\Fg}\ar@{}[r]|*{\ni} &  \hmuf=\pim^{\ast}\muf \ar@{}[r]|*{=}& \hf+\jm^{\ast}\phf \\
\QQQ(\Ng)^{\Gg}\ar[u]_{\pim^{\ast}}\ar@{}[r]|*{\ni}& \muf. \ar@{|->}[u]. &}
\end{equation}

The invariant \emph{homomorphism} $\hf\in \HHH^1(\Mg)^{\Fg}$ plays a key role to the construction of the alternating $\ZZ$-bilinear form $\bb_{\muf}$. We construct a map $\hbb_{\muf}\colon \Fg\times \Fg\to \RR$ in the following manner. The terminology of `core' in Definition~\ref{defn=core} (1) originates from \cite[Definition~7.9]{KKMMMcoarse}. 

\begin{defn}[core and $\hbb_{\muf}$]\label{defn=core}
Assume Setting~\ref{setting=GNFM}. Let $\muf\in \QQQ(\Ng)^{\Gg}$. 
\begin{enumerate}[label=\textup{(\arabic*)}]
\item The element $\hf$ in $\HHH^1(\Mg)^{\Fg}$ appearing decomposition \eqref{eq=core} of $\hmuf=\pim^{\ast}\muf$ is called the \emph{core} of $\muf$ with respect to $(\Fg,\pim)$.
\item Define $\hbb_{\muf,\Fg,\pim}=\hbb_{\muf}\colon \Fg\times \Fg\to \RR$ by 
\[
\hbb_{\muf}(\fl_1,\fl_2)=\hf([\fl_1,\fl_2])
\]
for all $\fl_1,\fl_2\in \Fg$. Here, $\hf$ is the core of $\muf$ with respect to $(\Fg,\pim)$.
\end{enumerate}
\end{defn}

In what follows, we show that this map $\hbb_{\muf}\colon \Fg\times \Fg\to \RR$ descends to a map $\Gamg\times \Gamg\to \RR$, which will turn out to be  the desired alternating $\ZZ$-bilinear form $\bb_{\muf}$.

\begin{lem}\label{lem=uptodefect}
Assume Setting~$\ref{setting=GNFM}$, and let $\muf\in \QQQ(\Ng)^{\Gg}$. Then, for all $\fl_1,\fl_2\in \Fg$ we have
\[
\muf([\pim(\fl_1),\pim(\fl_2)])\sim_{\DD(\muf)} \hbb_{\muf}(\fl_1,\fl_2).
\]
\end{lem}

\begin{proof}
By using decomposition \eqref{eq=core} of $\hmuf=\pim^{\ast}\muf$, we have
\[
\hmuf([\fl_1,\fl_2])=\hf([\fl_1,\fl_2])+\jm^{\ast}\phf([\fl_1,\fl_2]).
\]
Since $\phf\in \QQQ(\Fg)$, we have $\phf([\fl_1,\fl_2])\sim_{\DD(\phf)}0$. Moreover, by Proposition~\ref{prop=decomposition}~(1), 
\begin{align*}
\DD(\phf)&=\DD(\jm^{\ast}\phf)=\DD(\jm^{\ast}\phf+\hf)\\
&=\DD(\hmuf)=\DD(\muf).
\end{align*}
Hence, by Lemma~\ref{lem:qm} we conclude  that
\[
\muf([\pim(\fl_1),\pim(\fl_2)])=\hmuf([\fl_1,\fl_2])\sim_{\DD(\muf)}\hbb_{\muf}(\fl_1,\fl_2),
\]
as desired.
\end{proof}

\begin{lem}[properties of $\hbb_{\muf}$]\label{lem=hbb}
Assume Setting~$\ref{setting=GNFM}$, and let $\muf\in \QQQ(\Ng)^{\Gg}$. Then, the map $\hbb_{\muf}\colon \Fg\times \Fg\to \RR$ satisfies the following.
\begin{enumerate}[label=\textup{(\arabic*)}]
\item $\hbb_{\muf}(\fl_1\fl_2,\fl)=\hbb_{\muf}(\fl_1,\fl)+\hbb_{\muf}(\fl_2,\fl)$ for all $\fl_1,\fl_2,\fl\in \Fg$.
\item $\hbb_{\muf}(\fl_1,\fl_2)=-\hbb_{\muf}(\fl_2,\fl_1)$ for all $\fl_1,\fl_2\in \Fg$.
\item $\hbb_{\muf}(\fl',\fl)=0$ for every $\fl\in \Fg$ and every $\fl'\in \Mg$.
\item $\hbb_{\muf}(\rl,\fl)=0$ for every $\fl\in \Fg$ and every $\rl\in \Rel$.
\end{enumerate}
\end{lem}

\begin{proof}
Item~(1) follows from the following commutator calculus:
\begin{equation}\label{eq=commutatorcal}
[\fl_1\fl_2,\fl]=\fl_1[\fl_2,\fl]\fl_1^{-1} [\fl_1,\fl].
\end{equation}
Here, recall that $\hf\in \HHH^1(\Mg)^{\Fg}=\HHH^1([\Fg,\Fg])^{\Fg}$. Item~(2) holds since $[\fl_2,\fl_1]=[\fl_1,\fl_2]^{-1}$. Now, item~(3) follows from item (1). 

Finally, we prove item (4). By Lemma~\ref{lem=uptodefect}, we have
\[
\hbb_{\muf}(\rl,\fl)\sim_{\DD(\muf)} \muf([\pim(\rl),\pim(\fl)])=\muf([e_{\Gg},\pim(\fl)])=0,
\]
where $e_{\Gg}$ denotes the group unit of $\Gg$. 
Let $k\in \NN$. By replacing $\rl$ by $\rl^k$, we have
\begin{equation}\label{eq=Rel_k}
\hbb_{\muf}(\rl^k,\fl)\sim_{\DD(\phf)} 0.
\end{equation}
By item (1), $\hbb_{\muf}(\rl^k,\fl)=k\hbb_{\muf}(\rl,\fl)$. Therefore, by dividing \eqref{eq=Rel_k} by $k$ and letting $k\to \infty$, we have item (4). This completes the proof.
\end{proof}

Let $\Gamg=\Gg/\Ng$ and $\pim\colon \Gg\twoheadrightarrow \Gamg$ be the group quotient map. By Lemma~\ref{lem=hbb}, the map $\hbb_{\muf}\colon \Fg\times \Fg\to \RR$ descends to a map $\Gamg\times \Gamg\to \RR$. That means, the map
\[
\Gamg\times \Gamg\to \RR;\quad (\gaml_1,\gaml_2)\mapsto \hbb_{\muf}(\fl_1,\fl_2)
\]
is well-defined, where we take arbitrary elements $\fl_1,\fl_2$ in $\Fg$ satisfying $(\ppm \circ \pim)(\fl_1)=\gaml_1$ and $(\ppm \circ \pim)(\fl_2)=\gaml_2$. Again by Lemma~\ref{lem=hbb}, this map is an alternating  $\ZZ$-bilinear form on $\Gamg$, and we write $\bb_{\muf}$ for it. Now we are in a position to close up the proof of Theorem~\ref{thm=bilinear_main}.

\begin{proof}[Proof of Theorem~$\ref{thm=bilinear_main}$]
We have defined the alternating $\ZZ$-bilinear form $\bb_{\muf}\colon\Gamg\times \Gamg\to \RR$. Finally, Lemma~\ref{lem=uptodefect} shows \eqref{eq=bilinear_main}.
\end{proof}

\subsection{The maps $\BbG$ and $\bBbG$}\label{subsec=Bb}

\begin{defn}
Let $\Gamg$ be an abelian group.
\begin{enumerate}[label=\textup{(\arabic*)}]
  \item We define $\AAA(\Gamg)$ to be the $\RR$-linear space of $\RR$-valued alternating $\ZZ$-bilinear forms $\Gamg\times \Gamg\to \RR$ on $\Gamg$.
  \item Assume that $\Gamg$ is equipped with a structure of an $\RR$-linear space. Then, we define the subspace $\AAA_{\RR}(\Gamg)$ of $A(\Gamg)$ by
\[
\AAA_{\RR}(\Gamg)=\{\bb\in \AAA(\Gamma)\;|\;\bb \textrm{ is $\RR$-bilinear}\}.
\]
\end{enumerate}
\end{defn}

Here, for an $\RR$-linear space $\Gamg$, the notion of an ($\RR$-valued) \emph{alternating $\RR$-bilinear form} on $\Gamma$ is equivalent to being an element $\bb$ in $\AAA(\Gamg)$ that satisfies 
\[
\bb(t\gaml_1,\gaml_2)=t\bb(\gaml_1,\gaml_2)
\]
for all $\gaml_1,\gaml_2\in \Gamg$ and every $t\in \RR$.

Let $(\Gg,\Ng,\Gamg)$ be a triple fitting in \eqref{eq=shortex}. 
By Theorem~\ref{thm=bilinear_main}, we can define a map
\[
\BbG\colon \QQQ(\Ng)^{\Gg}\to \AAA(\Gamg);\quad \muf\mapsto \BbG(\muf)=\bb_{\muf}.
\] 
By limit formula \eqref{eq=limit_formula}, this map $\BbG$ is an $\RR$-linear map.

\begin{prop}\label{prop=V}
Under the setting above, $\Ker(\BbG)=\im^{\ast}\QQQ(\Gg)$.
\end{prop}

\begin{proof}
Fix $(\Fg,\pim)$, where $\Fg$ is a free group and $\pim\colon \Fg\twoheadrightarrow \Gg$ is a surjective group homomorphism. Let $\muf\in \QQQ(\Ng)^{\Gg}$. Then, as we argued in Subsection~\ref{subsec=bilinear_const}, $\BbG(\muf)=\bb_{\muf}$ is lifted to $\hbb_{\muf}\colon \Fg\times \Fg\to\RR$. The element $\muf$ in $\QQQ(\Ng)^{\Gg}$ belongs to $\Ker(\BbG)$ if and only if $\hbb_{\muf}\equiv 0$; it happens if and only if the core $\hf$ of $\muf$ is the zero-map (recall that $\hf$ is an invariant homomorphism). To summarize, $\muf$ is an element of $\Ker(\BbG)$ if and only if $\pim^{\ast}\muf\in \jm^{\ast}\QQQ(\Fg)$. By Proposition~\ref{prop=ore}, it happens if and only if $\muf\in \im^{\ast}\QQQ(\Gg)$, as desired.
\end{proof}

Now we are in a position to define the map $\bBbG$; recall from Definition~\ref{defn=Vspace} that $\VV(\Gg,\Ng)$ is defined as $\VV(\Gg,\Ng)=\QQQ(\Ng)^{\Gg}/\im^{\ast}\QQQ(\Gg)$. We also include the definition of $\BbG$ in Definition~\ref{defn=bilinear_form_map} as a reference.

\begin{defn}\label{defn=bilinear_form_map}
Let $(\Gg,\Ng,\Gamg)$ be a triple fitting in \eqref{eq=shortex}. 
\begin{enumerate}[label=\textup{(\arabic*)}]
\item The $\RR$-linear map $\BbG$ is defined as
\[
\BbG\colon \QQQ(\Ng)^{\Gg}\to \AAA(\Gamg); \quad \muf \mapsto \bb_{\muf}.
\] 
\item By Proposition~\ref{prop=V}, the $\RR$-linear map $\BbG$ descends to the following \emph{injective} $\RR$-linear map
\[
\bBbG\colon \VV(\Gg,\Ng)\hookrightarrow \AAA(\Gamg);\quad [\muf]\mapsto \bb_{\muf}.
\]
\end{enumerate}
\end{defn}

We note that in \eqref{eq=shortex}, the group $\Ng=[\Gg,\Gg]$ is determined by $\Gg$. Thus, once $\Gg$ is given, the two maps $\BbG$ and $\bBbG$ are defined.

\begin{prop}\label{prop=finitedim}
Let $(\Gg,\Ng,\Gamg)$ be a triple fitting in \eqref{eq=shortex}. Assume that $\Gamg$ is equipped with a structure of a finite-dimensional $\RR$-linear space, and let $m$ be the dimension. Then
\[
\dim_{\RR} \bBbG^{-1}(\AAA_{\RR}(\Gamg))\leq \frac{m(m-1)}{2}.
\]
In particular, $\dim_{\RR} \bBbG^{-1}(\AAA_{\RR}(\Gamg))<\infty$.
\end{prop}

\begin{proof}
Observe that $\dim_{\RR}\AAA_{\RR}(\Gamg)\leq \dfrac{m(m-1)}{2}$. Now, the injectivity of $\bBbG$ ends our proof.
\end{proof}

In Example~\ref{exa=symplectic}, it is unclear whether $\VV(\Gg,\Ng)=\VV(\tSympc(\Mm,\omega),\tHamc(\Mm,\omega))$ is finite-dimensional. Nevertheless, since $\Gamg=\HHH^1_c(\Mm;\RR)$ is a finite-dimensional $\RR$-linear space, the subspace $ \bBbG^{-1}(\AAA_{\RR}(\HHH^1_c(\Mm;\RR)))$ of $\VV(\Gg,\Ng)$ is finite-dimensional by Proposition~\ref{prop=finitedim}. This observation will be one of the keys to the proof of Theorem~\ref{thm=findimCalabi}.

We also have the following application of $\bb_{\muf}$ to the extension problem for invariant quasimorphisms.

\begin{prop}\label{prop=extendability}
Let $(\Gg,\Ng,\Gamg)$ be a triple fitting in \eqref{eq=shortex}. Let $\Pg$ be a subgroup of $\Gg$ with $\Pg\geqslant \Ng$. Then, the following are equivalent.
\begin{enumerate}[label=\textup{(\roman*)}]
 \item $\muf$ is extendable to an element in $\QQQ(\Pg)$;
 \item $\bb_{\muf}|_{\ppm(\Pg)\times \ppm(\Pg)}\equiv 0$.
\end{enumerate}
\end{prop}

\begin{proof}
Note that $\ppm(\Pg)=\Pg/\Ng$ is abelian since $\ppm(\Pg)\leqslant \Gamg$. 
Then, 
Proposition~\ref{prop=abelianiroiro} (2) 
implies the following: (i) happens if and only if 
\[
\sup\{|\muf([\al_1,\al_2])|\;|\; \al_1,\al_2\in \Pg\}<\infty.
\]
By \eqref{eq=bilinear_main}, this occurs exactly when (ii) is satisfied. Indeed, by $\ZZ$-bilinearity, $\bb_{\muf}$ is bounded on $\ppm(\Pg)\times \ppm(\Pg)$ if and only if $\bb_{\muf}|_{\ppm(\Pg)\times \ppm(\Pg)}\equiv 0$.
\end{proof}

\subsection{Invariance under a group action}\label{subsec=action}

\begin{prop}[constructing an invariant bilinear form]\label{prop=inv}
Let $(\Gg,\Ng,\Gamg)$ be a triple fitting in \eqref{eq=shortex}. Let $\Theta$ be a group. Assume that $\Theta$ acts on $\Gg$ by group automorphisms, so that the $\Theta$-action on $\Gg$ induces that on $\Gamg=\Gg/\Ng$. Let $\BbG\colon \QQQ(\Ng)^{\Gg}\to\AAA(\Gamg)$ be the map defined in Definition~\textup{\ref{defn=bilinear_form_map}}. Let $\muf\in \QQQ(\Ng)^{\Gg}$, and assume that $\muf$ is $\Theta$-invariant, meaning that
\[
\muf(\theta \cdot \xl)=\muf(\xl)
\]
for every $\theta\in \Theta$ and every $\xl\in \Ng$. Then, $\BbG(\muf)$ is $\Theta$-invariant, meaning that
\[
\BbG(\muf)(\theta\cdot\gaml_1,\theta\cdot\gaml_2)=\BbG(\muf)(\gaml_1,\gaml_2)
\]
for every $\theta\in \Theta$ and all $\gaml_1,\gaml_2\in \Gamg$.
\end{prop}

Here, observe that the $\Theta$-action on $\Gg$ leaves $\Ng=[\Gg,\Gg]$ invariant (setwise) because this action is by group automorphisms. 

\begin{proof}[Proof of Proposition~\textup{\ref{prop=inv}}]
Let $\theta\in \Theta$. Let $\gaml_1,\gamma_2\in \Gamg$. Take $\gl_1,\gl_2\in \Gg$ such that $\ppm(\gl_1)=\gaml_1$ and $\ppm(\gl_2)=\gaml_2$. Since the $\Theta$-action on $\Gg$ is by group automorphisms and $\muf$ is $\Theta$-invariant,  by limit formula \eqref{eq=limit_formula} we have
\begin{align*}
\BbG(\muf)(\theta\cdot\gaml_1,\theta\cdot\gaml_2)&=\lim_{k\to \infty}\frac{\muf([(\theta\cdot \gl_1)^k,\theta\cdot \gl_2])}{k}\\
&=\lim_{k\to \infty}\frac{\muf([\theta\cdot (\gl_1^k),\theta\cdot \gl_2])}{k}\\
&=\lim_{k\to \infty}\frac{\muf(\theta\cdot [\gl_1^k,\gl_2])}{k}=\BbG(\muf)(\gaml_1, \gaml_2),
\end{align*}
as desired.
\end{proof}

\subsection{A criterion for the $\RR$-bilinearity}\label{subsec=RRbilinear}

Let $(\Gg,\Ng,\Gamg)$ be a triple fitting in \eqref{eq=shortex}
In certain cases, $\Gamg$ is equipped with a structure of an $\RR$-linear space; for instance, the one as in Example~\ref{exa=symplectic}.  In such a case, it is natural to ask for which $\muf\in \QQQ(\Ng)^{\Gg}$, $\bb_{\muf}$ belongs to $\AAA_{\RR}(\Gamg)$, that is, $\bb_{\muf}$ is $\RR$-bilinear. The following proposition provides a criterion for proving the $\RR$-bilinearity of $\bb_{\muf}$.

\begin{prop}[criterion for the $\RR$-bilinearity of $\bb_{\muf}$]\label{prop=RRbilinear}

Let $(\Gg,\Ng,\Gamg)$ be a triple fitting in \eqref{eq=shortex}. Assume the following two conditions.
\begin{enumerate}[label=\textup{(\arabic*)}]
 \item $\Gamg$ is equipped with a structure of an $\RR$-linear space;
 \item for every $\gaml\in \Gamg$, there exists  a group homomorphism  $c_{\gaml}\colon \RR \to \Gg$ such that for every $t\in \RR$, $(\ppm \circ c_{\gaml})(t)=t\gaml$.
\end{enumerate}
Let $\muf\in \QQQ(\Ng)^{\Gg}$. Assume the following condition:
\begin{enumerate}
\item[\textup{(3)}] for every $\gl\in \Gg$ and every $\gaml\in\Gamg$, 
\[
\sup\{|\muf([c_{\gaml}(s),\gl])|\;|\; s\in [0,1]\}<\infty.
\]
\end{enumerate}
Then, $\BbG(\muf)=\bb_{\muf}$ belongs to $\AAA_{\RR}(\Gamg)$.
\end{prop}

\begin{proof}
Let $\gaml_1,\gaml_2\in\Gamg$. Take $\gl_2\in \Gg$ with $\ppm(\gl_2)=\gaml_2$. Take $c_{\gaml_1}\colon \RR\to \Gg$ as in condition (2). Set $\gl_1=c_{\gaml_1}(1)$. Let $t\in \RR$. If $t=0$, then $\bb_{\muf}(0\cdot \gaml_1,\gaml_2)=0=0\cdot \bb_{\muf}(\gaml_1,\gaml_2)$. Hence, we may assume that $t\ne 0$. Let $k\in \NN$, and write $tk=\lfloor tk \rfloor +\{tk\}$. Here, for $u\in \RR$, $\lfloor u \rfloor$ is the largest integer not exceeding $u$, and $\{u\}=u-\lfloor u\rfloor\in [0,1)$. Then, by limit formula \eqref{eq=limit_formula}, we have
\begin{align*}
\bb_{\muf}(t\gaml_1,\gaml_2)&=\lim_{k\to \infty}\frac{\muf([c_{\gaml_1}(t)^k,\gl_2])}{k}\\
&=\lim_{k\to \infty}\frac{\muf([c_{\gaml_1}(tk),\gl_2])}{k}\\
&=\lim_{k\to \infty}\frac{\muf([c_{\gaml_1}(\lfloor tk \rfloor +\{tk\}),\gl_2])}{k}\\
&=\lim_{k\to \infty}\frac{\muf([\gl_1^{\lfloor tk \rfloor} c_{\gaml_1}(\{tk\}),\gl_2])}{k}.
\end{align*}
By \eqref{eq=commutatorcal} and $\muf\in \QQQ(\Ng)^{\Gg}$, we have
\[
\muf([\gl_1^{\lfloor tk \rfloor} c_{\gaml_1}(\{tk\}),\gl_2])\sim_{\DD(\muf)}\muf([\gl_1^{\lfloor tk \rfloor} ,\gl_2])+\muf([c_{\gaml_1}(\{tk\}),\gl_2]).
\]
By condition~(3), we can set $C=\sup\{|\muf([c_{\gaml_1}(s),\gl_2])|\;|\; s\in [0,1]\}<\infty$. Then, we have
\[
\frac{\muf([\gl_1^{\lfloor tk \rfloor} c_{\gaml_1}(\{tk\}),\gl_2])}{k}\sim_{(\DD(\muf)+C)/k}\frac{\muf([\gl_1^{\lfloor tk \rfloor} ,\gl_2])}{k}.
\]
By recalling \eqref{eq=limit_formula}, we have
\begin{align*}
\lim_{k\to\infty}\frac{\muf([\gl_1^{\lfloor tk \rfloor} ,\gl_2])}{k}
&=\lim_{k\to\infty}\left(\frac{\muf([\gl_1^{\lfloor tk \rfloor} ,\gl_2])}{\lfloor tk \rfloor}\cdot\frac{\lfloor tk \rfloor}{k}\right)\\
&=t \bb_{\muf}(\gaml_1,\gaml_2).
\end{align*}
Therefore, we conclude that
\[
\bb_{\muf}(t\gaml_1,\gaml_2)=t \bb_{\muf}(\gaml_1,\gaml_2),
\]
as desired.
\end{proof}

\begin{rem}\label{rem=unhomog}
In the setting of Proposition~\ref{prop=RRbilinear}, assume that a (not necessarily homogeneous) quasimorphism $\muf'$ on $\Ng$ satisfies $(\muf')_{\mathrm{h}}=\muf$, namely, the homogenization of $\muf'$ equals $\muf$. Then, condition (3) can be replaced by
\begin{enumerate}
\item[$(3')$] for every $\gl\in \Gg$ and every $\gaml\in\Gamg$, 
\[
\sup\{|\muf'([c_{\gaml}(s),\gl])|\;|\; s\in [0,1]\}<\infty.
\]
\end{enumerate}
Indeed, recall from Lemma~\ref{lem=homogenization}~(2) that $\|\muf-\muf'\|_{\infty}\leq \DD(\muf')$. Now, the equivalence between (3) and ($3'$) is straightforward.\end{rem}

\begin{cor}\label{cor=RRbilinear}

Let $(\Gg,\Ng,\Gamg)$ be a triple fitting in \eqref{eq=shortex}. Assume the following three conditions.
\begin{enumerate}[label=\textup{(\arabic*)}]
 \item $\Gamg$ is equipped with a structure of an $\RR$-linear space; 
\item $\Gg$ is equipped with a topology, and $\Ng$ is equipped with the relative topology of it;
 \item for every $\gaml\in \Gamg$, there exists  a continuous group homomorphism  $c_{\gaml}\colon \RR \to \Gg$ such that for every $t\in \RR$, $(\ppm \circ c_{\gaml})(t)=t\gaml$.
\end{enumerate}
Let $\muf\in \QQQ(\Ng)^{\Gg}$. Assume the following condition.
\begin{enumerate}
\item[\textup{(4)}] There exists a continuous function $\Phi\colon \Ng\to \RR$ such that for all $\gl_1,\gl_2\in \Gg$, 
\[
|\muf([\gl_1,\gl_2])|\leq |\Phi([\gl_1,\gl_2])|.
\]
\end{enumerate}
Then, $\BbG(\muf)=\bb_{\muf}$ belongs to $\AAA_{\RR}(\Gamg)$.
\end{cor}

In Corollary~\ref{cor=RRbilinear}, we do not assume that $\Ng$ is a closed subgroup of $\Gg$. Neither do we assume that the structure of an $\RR$-linear space on $\Gamg$ is compatible with the quotient topology on $\Gamg=\Gg/\Ng$.

\begin{proof}[Proof of Corollary~$\ref{cor=RRbilinear}$]
By Proposition~\ref{prop=RRbilinear}, it suffices to show that 
\begin{equation}\label{eq=cbounded}
\sup\{|\muf([c_{\gaml}(s),\gl])|\;|\; s\in [0,1]\}<\infty
\end{equation}
for every $\gl\in \Gg$ and every $\gaml\in \Gamg$. Consider the  function
\[
[0,1]\to\RR;\quad s\mapsto \Phi([c_{\gaml}(s),\gl]).
\]
By conditions~(2) and (3), this map is continuous. By compactness of $[0,1]$, the image of this function is bounded. Hence, we obtain \eqref{eq=cbounded} by condition (4).
\end{proof}

\subsection{The bilinear form $\bb_{\muf}$ and the center of $\Gg$}\label{subsection=center}

Results in this subsection will be employed in Section~\ref{sec=Reznikov} to prove Theorem~\ref{mthm=Reznikov}. We first state the following lemma; for the reader's convenience, we include its proof.

\begin{lem}\label{lem=descend}
Let $\Gg$ be a group and  $\Ng$ a normal subgroup of $\Gg$. Let $\Zg$ be a normal subgroup of $\Gg$. Let $\cGg=\Gg/\Zg$ and $\cNg=\Ng/(\Ng\cap \Zg)$. Let $\qqm\colon \Gg\twoheadrightarrow \cGg$ be the group quotient map. Then for $\muf\in \QQQ(\Ng)^{\Gg}$, the following two conditions are equivalent.
\begin{enumerate}[label=\textup{(\roman*)}]
  \item The quasimorphism $\muf$ descends to an element in $\QQQ(\cNg)^{\cGg}$. Namely, there exists $\cmuf\in \QQQ(\cNg)^{\cGg}$ such that $\cmuf\circ \left(\qqm|_{\Ng}\right)=\muf$;
  \item the restriction of $\muf$ to $\Ng\cap \Zg$ is the zero map.
\end{enumerate}
\end{lem}

\begin{proof}
We first show that condition (i) is equivalent to 
\begin{equation}\label{eq=normal}
\muf(\xl\zl)=\muf(\xl)
\end{equation}
for every $\xl\in \Ng$ and every $\zl\in \Ng\cap \Zg$. Indeed, if \eqref{eq=normal} holds, then $\muf$ descends to an element $\cmuf$ in $\QQQ(\cNg)$; since $\muf\in \QQQ(\Ng)^{\Gg}$, this $\cmuf$ is $\cGg$-invariant. The converse is trivial. In particular, we obtain that (i) implies (ii). 

In what follows, we prove that (ii) implies (i). Assume (ii), and let $\xl\in \Ng$ and $\zl\in\Ng\cap  \Zg$. For every $k\in \NN$, there exists $\zl^{(k)}\in \Ng\cap \Zg$ such that $(\xl\zl)^k=\xl^k\zl^{(k)}$. Hence, we have
\[
\muf(\xl\zl)=\frac{\muf((\xl\zl)^k)}{k}=\frac{\muf(\xl^k\zl^{(k)})}{k}\sim_{\frac{\DD(\muf)}{k}} \frac{\muf(\xl^k)+\muf(\zl^{(k)})}{k}=\muf(\xl).
\]
By letting $k\to \infty$, we obtain \eqref{eq=normal}, as desired.
\end{proof}

For a group $\Gg$, $Z(\Gg)$ denotes the center of $\Gg$. The following result immediately follows from limit formula \eqref{eq=limit_formula}.

\begin{lem}\label{lem=centernull}
Let $(\Gg,\Ng,\Gamg)$ be a triple fitting in \eqref{eq=shortex}. Let $\muf\in \QQQ(\Ng)^{\Gg}$. Then, for every $\laml\in \ppm(Z(\Gg))$ and every $\gaml\in \Gamg$,
\[
\bb_{\muf}(\laml,\gaml)=0.
\]
\end{lem}

Now we are in a position to state the key proposition to the proof of Theorem~\ref{mthm=Reznikov}; we will apply this to the setting of Example~\ref{exa=Reznikovsetting} below.

\begin{prop}[bilinear form and modding out by the center]\label{prop=center}
Let $(\Gg,\Ng,\Gamg)$ be a triple fitting in \eqref{eq=shortex}. Let $\cGg=\Gg/Z(\Gg)$, $\cNg=\Ng/(\Ng\cap Z(\Gg))$ and $\cGamg=\Gamg/\ppm(Z(\Gg))$, which give the following short exact sequence:
\begin{equation}\label{eq=shortex_center}
1 \longrightarrow \cNg=[\cGg,\cGg] \longrightarrow \cGg \stackrel{\cppm}{\longrightarrow} \cGamg=\cGg/\cNg \longrightarrow 1.
\end{equation}
Let $\qqm\colon \Gg\twoheadrightarrow \cGg$ be the group quotient map. 
Let $\BbG\colon \QQQ(\Ng)^{\Gg}\to \AAA(\Gamg)$ and $\BbcG\colon \QQQ(\cNg)^{\cGg}\to \AAA(\cGam)$ be the maps defined in Definition~$\ref{defn=bilinear_form_map}$ for  $\Gg$ and $\cGg$, respectively. 
Let $\muf\in \QQQ(\Ng)^{\Gg}$. Assume that the restriction of $\muf$ to $\Ng\cap Z(\Gg)$ equals the zero map, and take the element  $\cmuf$ in $\QQQ(\cNg)^{\cGg}$ with $\muf=\cmuf\circ \left(\qqm|_{\Ng}\right)$. Let $\Pg$ be a subgroup of $\Gg$ with $\Pg\geqslant \Ng$, and set $\cPg=\qqm(\Pg)$. Then, the following are all equivalent:
\begin{enumerate}[label=\textup{(\roman*)}]
  \item $\muf$ is extendable to an element in $\QQQ(\Pg)$;
  \item $\cmuf$ is extendable to an element in $\QQQ(\cPg)$;
  \item the restriction of $\BbG(\muf)$ to $\ppm(\Pg)\times \ppm(\Pg)$ is the zero form;
  \item the restriction of $\BbcG(\cmuf)$ to $\cppm(\cPg)\times \cppm(\cPg)$ is the zero form.
\end{enumerate}
\end{prop}

Here, the existence of $\cmuf$ follows from Lemma~\ref{lem=descend}. We note that  (ii) trivially implies (i), as we can compose an extension of $\cmuf$ and $\qqm|_{\Pg}$. The fact that (iii) implies (ii) will be employed in Section~\ref{sec=Reznikov}.

\begin{proof}[Proof of Proposition~$\ref{prop=center}$]
The equivalence between (i) and (iii) follows from Proposition~\ref{prop=extendability}; so does the one between (ii) and (iv). By Lemma~\ref{lem=centernull}, $\BbG(\muf)\colon \Gamg\times \Gamg\to \RR$ descends to a $\ZZ$-bilinear form $\cbb\colon \cGamg\times \cGamg\to \RR$. By limit formula \eqref{eq=limit_formula}, $\cbb$ coincides with $\BbcG(\cmuf)$. Therefore, (iii) and (iv) are equivalent. This ends our proof.
\end{proof}

\begin{exa}\label{exa=Reznikovsetting}
In Section~\ref{sec=Reznikov}, we will apply Proposition~\ref{prop=center} to the following case. Let $(\Mm,\omega)$ be a closed symplectic manifold. Now, consider the setting where the short exact sequence \eqref{eq=shortex} corresponds to
\[
1 \longrightarrow \tHamc(\Mm,\omega) \longrightarrow \tSympc(\Mm,\omega) \xrightarrow{\tflux_{\omega}} \HHH^1_c(\Mm;\RR) \longrightarrow 1.
\]
Recall from Proposition~\ref{survey on flux} that $Z(\tSympc(\Mm,\omega))=\pi_1(\Sympc_0(\Mm,\omega))$ and that $Z(\tSympc(\Mm,\omega)) \cap \tHamc(\Mm,\omega)=\pi_1(\Hamc(\Mm,\omega))$; the short exact sequence \eqref{eq=shortex_center} in this setting corresponds to
\[
1 \longrightarrow \Hamc(\Mm,\omega) \longrightarrow \Sympc_0(\Mm,\omega) \xrightarrow{\flux_{\omega}} \HHH^1_c(\Mm;\RR)/\Gamma_{\omega} \longrightarrow 1.
\]
Here, $\Gamma_{\omega}=\tflux_{\omega}(\pi_1(\Sympc_0(\Mm,\omega)))$ is the symplectic flux group (recall Subsection~\ref{subsec=symp}).
\end{exa}


\section{Bilinear forms constructed out of Calabi quasimorphisms on surfaces}\label{sec=Py}

In this section, we study bilinear forms constructed out of Calabi quasimorphisms on surfaces. In Subsection~\ref{subsec=Py}, we will prove Theorem~\ref{thm=KKMM}. In Subsection~\ref{subsec=Calabi}, we recall the definition of Calabi quasimorphisms (Definition~\ref{definition of Calabi qm}). Then for the case of closed symplectic surfaces with  genus at least two, we show that $\ZZ$-bilinear forms constructed out of Calabi quasimorphisms are $\RR$-bilinear (Proposition~\ref{prop=RRbilinear_Calabi}), and  obtain Theorem~\ref{thm=findimCalabi}. For a closed symplectic surface $(\Mm,\omega)$ of genus at least two, we have $\tSympc(\Mm,\omega)=\Sympc_0(\Mm,\omega)$ and $\tHamc(\Mm,\omega)=\Hamc(\Mm,\omega)$ (Proposition~\ref{survey on flux}~(6)). Nevertheless, in Subsection~\ref{subsec=Py} we use symbols $\tSympc(\Mm,\omega)$ and $\tHamc(\Mm,\omega)$ because these groups appear in a general situation of $\mushf$.

\subsection{Constructing the symplectic pairing out of Py's Calabi quasimorphism}\label{subsec=Py}
Here the proof of Theorem~\ref{thm=KKMM} is provided. In terms of the map $\BbSymp$ defined in Definition~\ref{defn=bilinear_form_map} 
for a closed symplectic surface $(\Mm,\omega)$ of genus at least two, Theorem~\ref{thm=KKMM} can be rephrased as follows.
\begin{thm}\label{thm=Py}
Let $(\Mm,\omega)$ be a closed surface of genus at least two  equipped with a symplectic form.  Let $\Bb=\BbSymp\colon \QQQ(\tHamc(\Mm,\omega))^{\tSympc(\Mm,\omega)}\to \AAA(\HHH^1_c(\Mm;\RR))$ be the map defined in Definition~\textup{\ref{defn=bilinear_form_map}}.
Let $\mupyf$ be Py's Calabi quasimorphism on $\tHamc(\Mm,\omega)$. Then, $\bb_{\mupyf}=\Bb(\mupyf)$ equals the symplectic pairing $\bb_{\omega}$.
\end{thm}

\begin{rem}\label{rem=intersection}
In the setting of Theorem~\ref{thm=Py}, since $\Mm$ is compact and the dimension $2n$ equals $2$, the form $\bb_{\omega}\colon \HHH^1_c(\Mm;\RR)\times \HHH^1_c(\Mm;\RR)\to \RR$ can be rewritten as 
\[
\bb_{\omega}(\vv,\wv)=(\vv \smile \wv)[\Mm],
\]
where $\smile$ is the cup product on the cohomology ring $\HHH^{\bullet}(\Mm;\RR)$ of $\Mm$ and $[\Mm]\in \HHH_2(\Mm;\RR)$ denotes the fundamental class of $\Mm$. We also mention that the bilinear form $\bb_{\omega}$ on $\HHH^1_c(\Mm;\RR)$ in this case coincides with the intersection form $\bb_I$.

In a manner similar to one above, if $(\Mm,\omega)$ is a closed connected symplectic manifold of dimension $2n$, then $\bb_{\omega}\colon \HHH^1_c(\Mm;\RR)\times \HHH^1_c(\Mm;\RR)\to \RR$ can be rewritten as 
\[
\bb_{\omega}(\vv,\wv)=(\vv \smile \wv \smile [\omega]^{n-1})[\Mm],
\]
where $[\Mm]\in \HHH_{2n}(\Mm;\RR)$ denotes the fundamental class of $\Mm$.
\end{rem}

The key ingredient to the proof of Theorem~\ref{thm=Py} was already in the previous work \cite{KKMM2} by some of the authors, although they were unaware of the fact that this would lead Theorem~\ref{thm=Py}. 

\begin{prop}[{\cite[Lemma~4.5 and Proposition~4.7]{KKMM2}}]\label{prop=KKMM}
Let $(\Mm,\omega)$ be a closed surface of genus at least two  equipped with a symplectic form. Let $\mupyf$ be Py's Calabi quasimorphism on $\tHamc(\Mm,\omega)$. Let $\gaml_1,\gaml_2\in \HHH^1_c(\Mm,\omega)$. Then, there exist $m\in \NN$, a sequence  $(\fl_1^{(k)})_{k\in \NN}$ in $\tSympc(\Mm,\omega)$ and $\fl_2^+,\fl_2^-\in \tSympc(\Mm,\omega)$ such that the following hold.
\begin{enumerate}[label=\textup{(\arabic*)}]
 \item $\tflux_{\omega}(\fl_2^+)+\tflux_{\omega}(\fl_2^-)=\frac{\gaml_2}{m}$.
 \item For every $k\in \NN$, $\tflux_{\omega}(\fl^{(k)})=k\gaml_1$.
 \item For every $k\in \NN$,
\[
\mupyf([\fl_1^{(k)},\fl_2^+][\fl_1^{(k)},(\fl_2^-)^{-1}]^{-1})\sim_{\DD(\mupyf)} \frac{k}{m}\cdot \bb_{\omega}(\gaml_1,\gaml_2).
\]
\end{enumerate}
\end{prop}

\begin{proof}[Proof of Theorem~$\ref{thm=Py}$]
Let $\gaml_1,\gaml_2\in \HHH^1_c(\Mm;\RR)$, and take $m\in \NN$, $(\fl_1^{(k)})_{k\in \NN}$, $\fl_2^+$ and $\fl_2^-$ as in Proposition~\ref{prop=KKMM}. By the refined limit formula \eqref{eq=refined_limit_formula_left} in Theorem~\ref{thm=refined_limit_formula} and Proposition~\ref{prop=KKMM}~(2), we have
\[
\lim_{k\to\infty}\frac{\mupyf([\fl_1^{(k)},\fl_2^+][\fl_1^{(k)},(\fl_2^-)^{-1}]^{-1})}{k}=\bb_{\mupyf}(\gaml_1,\tflux_{\omega}(\fl_2^+))+\bb_{\mupyf}(\gaml_1,\tflux_{\omega}(\fl_2^-)).
\]
By Proposition~\ref{prop=KKMM}~(1), the right-hand side of the equality above equals $\frac{1}{m}\bb_{\mupyf}(\gaml_1,\gaml_2)$. Therefore, Proposition~\ref{prop=KKMM}~(3) implies that
\[
\frac{1}{m}\bb_{\omega}(\gaml_1,\gaml_2)=\frac{1}{m}\bb_{\mupyf}(\gaml_1,\gaml_2);
\]
hence $\bb_{\mupyf}\equiv \bb_{\omega}$.
\end{proof}

\subsection{Calabi quasimorphisms on surfaces}\label{subsec=Calabi}

In this subsection,  we prove Theorem~\ref{thm=findimCalabi}. We rephrase it as follows.

\begin{thm}[paraphrase of Theorem~\ref{thm=findimCalabi}]\label{thm=Calabi}
Let $(\Mm, \omega)$ be a closed  surface  equipped with a symplectic form whose genus $l$ is at least two. Let $\Gg=\Sympc_0 (\Mm, \omega)$ and $\Ng=\Hamc(\Mm,\omega)$. Let $\Wv\subset \QQQ(\Ng)^{\Gg}$ be the $\RR$-span of Calabi quasimorphisms. Then the image of $\Wv$ under the projection $\QQQ(\Ng)^{\Gg}\twoheadrightarrow \VV(\Gg, \Ng)$ has real dimension  at most $l(2l-1)$.
\end{thm}

The key to the proof of Theorem~\ref{thm=Calabi} is the $\RR$-bilinearity of the $\ZZ$-bilinear form constructed from a Calabi quasimorphism on surfaces. We first recall the definition of Calabi quasimorphisms.

\begin{defn} \label{definition of displaceable subset}
A subset $A$ of a symplectic manifold $(\Mm,\omega)$ is said to be \emph{displaceable} if there exists $\varphi \in \Hamc(\Mm,\omega)$ satisfying $\varphi(A) \cap \overline{A} = \emptyset$. Here, $\overline{A}$ is the topological closure of $A$  in $\Mm$.
\end{defn}

Let $(M, \omega)$ be an exact symplectic manifold. Then the \emph{descended Calabi homomorphism} of $(M,\omega)$
$\underline{\mathrm{Cal}}_{\Mm} \colon \Hamc(\Mm, \omega) \to \RR$  is defined by
\begin{align} \label{eq=dCalabi}
\underline{\mathrm{Cal}}_{\Mm} =\int_0^1\left(\int_{\Mm} H_t \omega^n\right)dt,
\end{align}
where $H \colon [0,1] \times \Mm \to \RR$ is a normalized smooth function generating $\varphi \in \Hamc(\Mm, \omega)$. It is well known that the value of the right-hand side of equation \eqref{eq=dCalabi} does not depend on the choice of $H$ when $\Mm$ is exact (see \cite[Section~10.3]{MS} for example). By the definition, the composition of
\[ \tSympc(\Mm,\omega) \to \Sympc_0(M,\omega) \xrightarrow{\underline{\mathrm{Cal}}_{\Mm}} \RR\]
coincides with the Calabi homomorphism $\mathrm{Cal}_{\Mm} \colon \tSympc(\Mm, \omega) \to \RR$.

\begin{defn}[\cite{EP03} and \cite{PR}]\label{definition of Calabi qm}
Let $(\Mm, \omega)$ be a $2n$-dimensional  symplectic manifold.
A \emph{Calabi quasimorphism} is defined as a homogeneous quasimorphism $\muf \colon \Hamc(\Mm, \omega) \to \RR$ such that 
for every non-empty open displaceable subset $U$ of $\Mm$ with exact $\omega|_U$, 
 the restriction of $\muf$ to $\Hamc(U,\omega)$ coincides with the descended Calabi homomorphism $\underline{\mathrm{Cal}}_U$.
\end{defn}

Let $(\Mm, \omega)$ be a closed symplectic surface of genus at least two. Then Brandenbursky \cite{Bra} constructed a Calabi quasimorphism $\mubrf$ on $\Hamc(\Mm,\omega)$. We only need the following property of his quasimorphism.

\begin{prop}[see \cite{Bra}]\label{thm=Brabur}
Brandenbursky's Calabi quasimorphism $\mubrf \colon \Hamc(\Mm, \omega) \to \RR$ is extendable to $\mathrm{Symp}_0(\Mm,\omega)$.
\end{prop}

Proposition \ref{thm=Brabur} follows from the very construction of  $\mubrf$: it is constructed as the restriction of a quasimorphism on $\mathrm{Symp}_0(\Mm,\omega)$.

\begin{thm}[{\cite[Theorem~1.7]{EPP}}]\label{theorem EPP}
Let $(\Mm, \omega)$ be a closed surface with symplectic form and $\muf$ a homogeneous quasimorphism on $\Hamc(\Mm,\omega)$.
Then, $\muf$ is continuous in the $C^0$-topology if and only if the following condition is satisfied: there exists a positive number $A$ such that for every disk $D\subset S$ of area less than $A$, the restriction of $\muf$ to $\Hamc(D,\omega|_D)$ vanishes.
\end{thm}

\begin{lem}\label{lemma minus Brandenbursky}
Let $(\Mm, \omega)$ be a closed surface of genus at least two  equipped with a symplectic form. Let $\muf$ be a Calabi quasimorphism on $\Hamc(\Mm, \omega)$. Then $\muf - \mubrf$ is continuous in the $C^0$-topology.
\end{lem}
\begin{proof}
Let $\muf$ be a Calabi quasimorphism. Let $A > 0$ such that every disk whose area is less than $A$ is displaceable. Then, by the definition of the Calabi property and Theorem \ref{theorem EPP}, $\muf - \mubrf$ vanishes on every disk whose area is less than $A$. Then Theorem~\ref{theorem EPP} implies that $\muf - \mubrf$ is continuous in $C^0$-topology.
\end{proof}

\begin{prop}\label{prop=RRbilinear_Calabi}
Let $(\Mm, \omega)$ be a closed surface of genus at least two  equipped with a symplectic form. Let 
\[
\Bb=\BbSymp\colon \QQQ(\Hamc(\Mm, \omega))^{\Sympc_0(\Mm,\omega)}\to \AAA(\HHH^1_c(\Mm;\RR))
\]
be the map defined in Definition~$\ref{defn=bilinear_form_map}$. 
Let $\muf \colon \Hamc(\Mm, \omega) \to \RR$ be a Calabi quasimorphism on $\Hamc(\Mm, \omega)$. Then, $\Bb(\muf)$ belongs to $\AAA_{\RR}(\HHH^1_c(\Mm;\RR))$, namely, the \textup{(}$\ZZ$-bilinear form\textup{)} $\Bb(\muf)$ is $\RR$-bilinear.
\end{prop}
\begin{proof}
We show that the quasimorphism $\muf - \mubrf$ satisfies the hypothesises of Corollary~\ref{cor=RRbilinear}. In this proof, we endow $\Sympc(M,\omega)$ with $C^0$-topology. It is straightforward to show (1) and (2) of Corollary~\ref{cor=RRbilinear}. Item (4) follows from Lemma~\ref{lemma minus Brandenbursky}.

We now prove that (3) is satisfied. This proof is almost the same as the proof of the surjectivity of $\tflux_\omega \colon \tSympc(M, \omega) \to \HHH^1_c(\Mm, \omega)$ (see \cite[Proposition~3.1.6]{Ban97}). Let $v \in \HHH^1_c(\Mm, \omega)$. Then we can take a closed 1-form $\alpha$ with compact support on $\Mm$ that represents $v$. Let $X$ be a vector field on $\Mm$ satisfying $\iota_X \omega = \alpha$. Since $\alpha$ is closed, the Cartan formula implies that $X$ is a symplectic vector field. Let $\{ \psi^t\}$ be the flow generated by $X$. By the definition of the flux homomorphism, we have $\tflux_\omega([\{\psi^t\}_{t\in [0,1]}]) = v$.  We also recall from Proposition~\ref{survey on flux} that $\tSympc(\Mm, \omega)=\Sympc_0(\Mm, \omega)$ and that $\tflux_{\omega}$ can be regarded as $\flux_{\omega}$. Thus, we can define a continuous map $c_{v}\colon \RR \to \Sympc_0(\Mm, \omega)$ by $c_{v}(t)=\psi^t$ for $t\in \RR$. By construction, $c_v$ is a group homomorphism, and for every $t\in \RR$ we have
\[
\flux_\omega(c_{v}(t)) = tv.
\]
This ends the proof of (3). 

Therefore, $\Bb(\muf-\mubrf)$ belongs to $\AAA_{\RR}(\HHH^1_c(\Mm;\RR))$ by Corollary~\ref{cor=RRbilinear}. By Theorem~\ref{thm=Brabur}, $\muf$ and $\muf-\mubrf$ represent the same class in $\VV(\Sympc_0 (\Mm, \omega), \Hamc(\Mm,\omega))$, so that $\Bb(\muf)=\Bb(\muf-\mubrf)$ (recall Proposition~\ref{prop=V}). Hence, we conclude that $\Bb(\muf)\in \AAA_{\RR}(\HHH^1_c(\Mm;\RR))$.
\end{proof}

\begin{proof}[Proof of Theorem~\textup{\ref{thm=Calabi}}]
We have $\dim_{\RR}\HHH^1_c(\Mm;\RR)=2l$. Then, Proposition~\ref{prop=RRbilinear_Calabi} together with Proposition~\ref{prop=finitedim} ends the proof.
\end{proof}

\section{$\RR$-bilinearity of $\bb_{\mushf}$ and the proof of Theorem~\ref{mthm=RRbilinear}}\label{sec=RRbilinear_Shelukhin}

In this section, we first prove that the $\ZZ$-bilinear form $\bb_{\mushf}$ constructed out of Shelukhin's quasimorphism $\mushf$ is in fact $\RR$-bilinear. Then, we present proofs of Theorem~\ref{mthm=RRbilinear} and Theorem~\ref{mthm=Shelukhin_extendable}.

\subsection{Proof of the $\RR$-bilinearity of $\bb_{\mushf}$}

Let $(\Mm,\omega)$ be a closed symplectic manifold. We first recall the definition of continuity via smooth isotopy from Definition~\ref{continuity}: a function $\muf \colon \tHamc(\Mm,\omega) \to\RR$ is said to be \textit{continuous via smooth isotopy} if for every smooth isotopy $\gamma\colon [0,1] \to \Hamc(M,\omega)$ from identity, 
the function $s \mapsto \mu\left(\left[\{\gamma(t)\}_{t\in[0,s]}\right]\right)$ is continuous. 


Here is one of the key propositions in this subsection.

\begin{prop}[criterion for $\RR$-bilinearity in terms of continuity via isotopy]\label{prop=RRbilinear_nonhomog}
  Let $(\Mm, \omega)$ be a symplectic manifold of finite volume. 
  Set $\Gg=\tSympc(\Mm, \omega)$, $\Ng=\tHamc(\Mm, \omega)$ and $\Gamg=\HHH^1_c(\Mm ; \RR)$.
  Let $\Bb_G \colon \QQQ(N)^{G}\to \Gamma$ be the map defined in Definition~\textup{\ref{defn=bilinear_form_map}}. 
Let $\muf'$ be a quasimorphism on $N$ and let $\muf = (\muf')_{\rm h}$ its homogenization.
If $\muf'$ is continuous via smooth isotopy and $\muf \in \QQQ(N)^{G}$, then the $\ZZ$-bilinear form $\bb_{\muf}=\Bb_G(\muf)$ is $\RR$-bilinear.
\end{prop}
  \begin{proof}
We will prove this proposition by checking the following three conditions: (1) and (2) of Proposition \ref{prop=RRbilinear}, and ($3'$) of Remark~\ref{rem=unhomog}. Then, Proposition \ref{prop=RRbilinear}, together with Remark~\ref{rem=unhomog}, ensures the $\RR$-bilinearity of $\bb_{\muf}$. Condition (1) of Proposition \ref{prop=RRbilinear} is obvious.

In what follows, we confirm condition (2) of Proposition \ref{prop=RRbilinear}. Recall that we have checked condition (3) of Corollary~\ref{cor=RRbilinear} in the proof of Proposition~\ref{prop=RRbilinear_Calabi} for a closed symplectic surface $(\Mm,\omega)$ of genus at least two. The last part of that proof employed Proposition~\ref{survey on flux}~(6), which is a special property for such $(\Mm,\omega)$. Nevertheless, the other part of the proof remains to work for  a general closed symplectic manifold $(\Mm,\omega)$, and we have the following: given $v\in \Gamg=\HHH^1_c(\Mm;\RR)$, there exists a symplctic vector field $X$ on $\Mm$ such that the flow $\{ \psi^t\}_{t\in \RR}$ generated by $X$ satisfies $\tflux_\omega( [\{\psi^t\}_{t\in [0,1]}] ) = v$. Now, define a map $c_v\colon \RR\to \tSympc(\Mm,\omega)$ by
\begin{equation}\label{eq=tflow}
c_v(t)=[\{\psi^s\}_{s\in [0,t]}]
\end{equation}
for every $t\in \RR$. Then, $c_v$ is a group homomorphism, and for every $t\in \RR$ we have
\[
\tflux_{\omega}(c_v(t))=tv.
\]
Hence, we have verified condition (2) of Proposition \ref{prop=RRbilinear}.

Finally, we prove condition ($3'$) of Remark~\ref{rem=unhomog}. Take $\gl\in\Gg$ and $v\in \Gamg$. Take the group homomorphism $c_v\colon \RR\to \tSympc(\Mm,\omega)$ defined by \eqref{eq=tflow}. To see condition ($3'$) of Remark~\ref{rem=unhomog}, it suffices to show that
\begin{align} \label{eq=conti.via.iso}
\sup\{|\muf'([c_v(t),\gl])|\;|\; t\in [0,1]\}<\infty.
\end{align}
Note that $[c_v(t), \gl]$ is contained in $\tHamc(\Mm, \omega)$ by Proposition~\ref{survey on flux} (4). Set $\bar{\gl} = \qqm (g)$, where $\qqm \colon \tSympc(\Mm, \omega) \to \Sympc_0(\Mm, \omega)$ is the universal covering map. Then the path
\[ [0,1] \to \tHamc(\Mm, \omega);\quad t \mapsto [c_v(t), \gl] = c_v(t) \gl c_v(t)^{-1} \gl^{-1}\]
in $\tHamc(\Mm, \omega)$ is a lift of the smooth isotopy from identity
\[ [0,1] \to \Hamc(\Mm, \omega); \quad t \mapsto [\psi^t, \bar{\gl}] = \psi^t \bar{\gl} \psi^{-t} \bar{\gl}^{-1}.\]
%
%
Therefore, by continuity  via smooth isotopy of $\muf'$, the function defined by
\[ 
[0,1] \to \RR;\quad t \mapsto \muf'([c_v(t), \gl]) 
\]
is continuous. By compactness of $[0,1]$, we conclude \eqref{eq=conti.via.iso}.
Now we have verified the three conditions, as desired.
\end{proof}


On the proof of the $\RR$-bilinearity of $\bb_{\mushf}$, the following proposition is the second key.

%
%
%

\begin{prop}\label{prop=Shelukhin_conti}
For every symplectic manifold $(\Mm, \omega)$ of finite volume and  an $\omega$-compatible almost complex structure $J$ on $\Mm$, 
the map $\nuf_{J} \colon \tHamc(\Mm,\omega) \to \RR$ defined by \eqref{before_homogenize}
is continuous via smooth isotopy.
\end{prop}

We remark that the proof below of this proposition goes along a line similar to that of \cite[Equation (8)]{Shelukhin}.

\begin{proof}[Proof of Proposition~\textup{\ref{prop=Shelukhin_conti}}]
Take a smooth isotopy $\gamma\colon [0,1] \to \Hamc(M,\omega)$ from identity.
Then, there exists a normalized Hamiltonian function $H_t\colon [0,1]\times M\to\RR$ such that $\varphi_H^t=\gamma(t)$ (for example, see \cite[Remark 4.2.3]{Ban} and \cite[Theorem 2.3.16]{O15a}).
Define a function $\Psi\colon[0,1]\to\RR$ by $\Psi(s)= \nu_J\left(\left[\{\gamma(t)\}_{t\in[0,s]}\right]\right)$.
Recall that we defined the map $\nu_J\colon \tHamc (\Mm, \omega) \to\RR$ by
\[\nu_J(\tg) := \int_D \Omega - \int_0^1 \mom(H_t)(h_t\cdot J).\]
Here, we define maps $\Psi^1, \Psi^2\colon [0,1] \to\RR$ by 
\[\Psi^1(s) := \int_{D(s)} \Omega, \, \Psi^2(s) := \int_0^s \mom(H_t)(h_t\cdot J).\]
Here, $D(s)$ is a disk in $\mathcal{J}$ whose boundary is equal to the loop $\{ h_t \cdot J \}_{t \in [0,s]} \ast [h_s\cdot J, J]$.
Then, it is sufficient to prove that the maps $\Psi^1, \Psi^2$ are continuous.
First, note that $\Psi^2$ is continuous because
\[\Psi^2(s) := \int_0^s \mom(H_t)(h_t\cdot J) = \int_0^s S(h_t\cdot J)H_t(x)\omega^n dt = \int_0^s S(J)H_t(h_t x)\omega^n dt.\]

Next, we prove that $\Psi^1$ is continuous.
Note that $[0,1]\times M\to \mathcal{S}$, $(t,x)\mapsto (h_t\cdot J)_x$ is a smooth map.
Hence the map $[0,1]\times M\to T\mathcal{S}$, $(t,x)\mapsto \frac{d}{dt}(h_t\cdot J)_x$ is a continuous map.
Therefore, since $\bigcup\limits_{t\in [0,1]} \mathrm{supp}(H_t)$ is compact, 
$\max\limits_{s\in[0,1],x\in M}\left\|\frac{d}{dt}|_{t=s}(h_t\cdot J)_x\right\|_x<\infty$.
Here, $\|\cdot\|_x$ is the norm induced from the fiberwise-K\"{a}hler form $\sigma_x$ on $\mathcal{J}_x$.
Thus, for every $s_0\in [0,1]$ and every $\epsilon>0$, there exists $\delta>0$ such that for every $s\in[0,1]$ with $|s-s_0|<\delta$, $\left|\int_{D(s)} \Omega-\int_{D(s_0)} \Omega \right|<\epsilon$.
Hence $\Psi^1$ is continuous.
\end{proof}

Now we are in a position to present the $\RR$-bilinearity of $\bb_{\mushf}$.

\begin{prop}[$\RR$-bilinearity of $\bb_{\mushf}$]\label{prop=RRbilinear_Shelukhin}
  Let $(\Mm, \omega)$ be a symplectic manifold of finite volume. Let $\Bb=\BbSymp\colon \QQQ(\tHamc(\Mm,\omega))^{\tSympc(\Mm,\omega)}\to \AAA(\HHH^1_c(\Mm;\RR))$ be the map defined in Definition~\textup{\ref{defn=bilinear_form_map}}. 
Then the $\ZZ$-bilinear form $\bb_{\mushf}=\Bb(\mushf)$, constructed from Shelukhin's quasimorphism $\mushf$, is $\RR$-bilinear.
\end{prop}

\begin{proof}
Recall that Shelukhin's quasimorphism $\mushf$ is the homogenization of $\nuf_{J}$.
Therefore, the assertion follows from Propositions \ref{prop=RRbilinear_nonhomog} and \ref{prop=Shelukhin_conti}.
\end{proof}

In addition, we prove Theorem~\ref{thm=findimCVSI}. In fact, we will show the following stronger statement.

\begin{thm}[stronger statement of Theorem~\ref{thm=findimCVSI}]\label{thm=findimCVSI_strong}
Let $(\Mm, \omega)$ be a closed  symplectic manifold. Let $m=\dim_{\RR}\HHH^1_c(\Mm;\RR)$. Let $\Gg=\tSympc(\Mm,\omega)$ and $\Ng=\tHamc(\Mm,\omega)$. Define the $\RR$-linear subspace $\Qv^{\mathrm{ch}}_{(\Mm,\omega)}$ of $\QQQ(\Ng)^{\Gg}$ by
\[
\Qv^{\mathrm{ch}}_{(\Mm,\omega)}=\{\muf\in \QQQ(\Ng)^{\Gg}\;|\; \textrm{there exists $\muf'$ such that $(\muf')_{\mathrm{h}}=\muf$}\},
\]
where $\muf'$ is assumed to be a \textup{(}not necessarily homogeneous\textup{)} quasimorphism on $\Ng$ that is continuous via smooth isotopy.
Then the image of $\Qv^{\mathrm{ch}}_{(\Mm,\omega)}$ under the projection $\QQQ(\Ng)^{\Gg}\twoheadrightarrow \VV(\Gg, \Ng)$ has real dimension  at most $\dfrac{m(m-1)}{2}$.
\end{thm}

\begin{proof}
Combine Propositions~\ref{prop=RRbilinear_nonhomog} and \ref{prop=finitedim}.
\end{proof}

The $\RR$-linear space $\Qv^{\mathrm{ch}}_{(\Mm,\omega)}$ might be of interest due to the following fact. 

\begin{prop}\label{prop=Shelukhin_ch}
Let $(\Mm,\omega)$ be a closed symplectic manifold. Then, Shelukhin's quasimorphism $\mushf$ is an element of the $\RR$-linear space $\Qv^{\mathrm{ch}}_{(\Mm,\omega)}$, defined in Theorem~\textup{\ref{thm=findimCVSI_strong}}.
\end{prop}

\begin{proof}
Since $(\nuf_J)_{\mathrm{h}}=\mushf$, it immediately follows from Proposition~\ref{prop=Shelukhin_conti}.
\end{proof}

\subsection{Proof of Theorem~\ref{mthm=RRbilinear}}

In the following proof, for a closed symplectic manifold  $(\Mm,\omega)$, we  set $\Gg=\tSympc(\Mm, \omega)$, $\Ng=\tHamc(\Mm, \omega)$ and $\Gamg=\HHH^1_c(\Mm ; \RR)$. 

\begin{proof}[Proof of Theorem~\textup{\ref{mthm=RRbilinear}}]
Recall from Subsection~\ref{subsec=symp} that $\Ng=[\Gg,\Gg]$. Therefore, we may appeal to the general theory developed in Section~\ref{sec=bilinear} and obtain the map
\[
\BbG\colon \QQQ(\Ng)^{\Gg}\to \AAA(\Gamg)
\]
defined in Definition~\ref{defn=bilinear_form_map}.
We set $\bb_{\mushf}=\BbG(\mushf)$. Then, all properties of $\bb_{\mushf}$ asserted in Theorem~\ref{mthm=RRbilinear}, \emph{except} $\Sympc(\Mm,\omega)$-invariance and $\RR$-bilinearity, follow from Theorem~\ref{thm=bilinear_main}. 

To show $\Sympc(\Mm,\omega)$-invariance, combine Proposition~\ref{prop=Sympc-inv} with Proposition~\ref{prop=inv} for the $\Sympc(\Mm,\omega)$-action \eqref{eq=Sympc-action} on $\Gg$ by group automorphisms.

Finally, the $\RR$-bilinearity of $\bb_{\mushf}$ is proved in Proposition~\ref{prop=RRbilinear_Shelukhin}. Now our proof has been completed.
\end{proof}

\subsection{Extendability and $\bb_{\mushf}$}\label{subsec=extendablityRR}

Here, we prove Theorem~\textup{\ref{mthm=Shelukhin_extendable}} and Theorem~\ref{thm=extendabilityRR}.

\begin{proof}[Proof of Theorem~\textup{\ref{mthm=Shelukhin_extendable}}]
This equivalence follows from the general theory in Section~\ref{sec=bilinear} (Proposition~\ref{prop=extendability}). 
\end{proof}

\begin{proof}[Proof of Theorem~\textup{\ref{thm=extendabilityRR}}]
By Theorem~\textup{\ref{mthm=Shelukhin_extendable}}, (i) is equivalent to saying that $\bb_{\mushf}$ vanishes on $\tflux_{\omega}(\Pg)\times \tflux_{\omega}(\Pg)$. Since $\bb_{\mushf}$ is $\RR$-bilinear, this condition is equivalent to the vanishing of $\bb_{\mushf}$ on $\Vv_{\Pg}\times \Vv_{\Pg}$. Again by Proposition~\ref{prop=extendability}, this is equivalent to (ii). 
\end{proof}

\subsection{Constraint on the fluxes of commuting two symplectomorphisms}
In this subsection, we prove Theorems~\ref{thm=commutingtSymp} and \ref{thm=commutingSymp}. We will only show Theorem~\ref{thm=commutingSymp}, as this theorem is a refinement of Theorem~\ref{thm=commutingtSymp}.

\begin{proof}[Proof of Theorem~\textup{\ref{thm=commutingSymp}}]
By assumption, there exists $z\in \pi_1(\Hamc(\Mm,\omega))$ such that $[f,g]=z$. 
By Proposition~\ref{survey on flux}, $z$ commutes with every element in $\tSympc(\Mm,\omega)$.  
By \eqref{eq=commutatorcal} we have $[f^k,g]=z^k$ for every $k\in \NN$. By limit formula \eqref{eq=Shelukhin_form_limit}, we obtain
\[
\mushf(z)=\bb_{\mushf}(\tflux_{\omega}(f),\tflux_{\omega}(g)).
\]
Since $z\in \pi_1(\Hamc(\Mm,\omega))$, Proposition~\ref{prop=ShelukhinHam} ends our proof of \eqref{eq=commutingcenter}.
\end{proof}

\begin{cor}[obstruction to commutativity of symplectomorphisms in terms of the fluxes]\label{cor=obs_flux}
Let $(\Mm,\omega)$ be a closed symplectic manifold. Let $\qqm\colon \tSympc(\Mm,\omega)\twoheadrightarrow \Sympc_0(\Mm,\omega)$ be the universal covering map.
\begin{enumerate}[label=\textup{(\arabic*)}]
  \item Let $f,g\in \tSympc(\Mm,\omega)$ with $\bb_{\mushf}(\tflux_{\omega}(f),\tflux_{\omega}(g))\ne 0$. Then, there do \emph{not} exist $\psi,\psi'\in \tHamc(\Mm,\omega)$ such that $f\psi$ and $g\psi'$ commute.
  \item Let $f,g\in \tSympc(\Mm,\omega)$ such that 
\[
\bb_{\mushf}(\tflux_{\omega}(f),\tflux_{\omega}(g))\in \RR\setminus I_{c_1}(\pi_1(\Hamc(\Mm,\omega))).
\]
Then, there do \emph{not} exist $\varphi,\varphi'\in \Hamc(\Mm,\omega)$ such that $\qqm(f)\varphi$ and $\qqm(g)\varphi'$ commute in $\Sympc_0(\Mm,\omega)$.
\end{enumerate}
\end{cor}

\begin{proof}
Recall that $\tHamc(\Mm,\omega)=\Ker(\tflux_{\omega})$. Now, (1) and (2) both follow from the contraposition of Theorem~\ref{thm=commutingtSymp}.
\end{proof}



\begin{rem}\label{rem=countable}
The subgroup $I_{c_1}(\pi_1(\Hamc(\Mm,\omega)))$ of $\RR$ is `small,' in the sense that this set is at most countable. Indeed, by \cite[Exercise~2.4.6]{O15a}, $\pi_1(\Sympc_0(\Mm,\omega))$ is at most countable; we also have $\pi_1(\Hamc(\Mm,\omega))=\tHamc(\Mm,\omega) \cap \pi_1(\Sympc_0(\Mm,\omega))$ by Proposition~\ref{survey on flux}. Therefore, $\pi_1(\Hamc(\Mm,\omega))$ is at most countable.
\end{rem}

\section{Characterization of the triviality of the Reznikov class}\label{sec=Reznikov}
In this Section, we prove Theorem~\ref{mthm=Reznikov} and Theorem~\ref{thm=ReznikovRR}. One key here is Proposition~\ref{prop=ShelukhinHam}. Another key is the following result in  \cite{KKMMMReznikov}, where the triviality of the Reznikov class is characterized in terms of the extendability of Shelukhin's quasimorphism $\mushf$.

\begin{prop}[{\cite[Lemma 5.2]{KKMMMReznikov}}]\label{prop=KKMMMReznikov}
Let $(\Mm,\omega)$ be a closed  symplectic manifold. Let $\Pg$ be a subgroup of $\tSympc(\Mm,\omega)$ with $\Pg\geqslant \tHamc(\Mm,\omega)$. Let $\overline{\Pg}$ be the image of $\Pg$ by the universal covering map $\qqm\colon \tSympc(\Mm,\omega)\twoheadrightarrow \Sympc_0(\Mm,\omega)$. Then, the following two conditions are equivalent.
\begin{enumerate}[label=\textup{(\roman*)}]
  \item The Reznikov class $R$ is trivial on $\overline{\Pg}$;
  \item there exists $\phf\in \QQQ(\overline{\Pg})$ such that $\left(\phf|_{\Hamc(\Mm,\omega)} \right)\circ \left(\qqm|_{\tHamc(\Mm,\omega)}\right)=\mushf$.
\end{enumerate}
\end{prop}

\begin{proof}[Proof of Theorem~\textup{\ref{mthm=Reznikov}}]
We focus on the setting as in Example~\ref{exa=Reznikovsetting}. Namely,  the short exact sequence \eqref{eq=shortex} corresponds to
\[
1 \longrightarrow \tHamc(\Mm,\omega) \longrightarrow \tSympc(\Mm,\omega) \xrightarrow{\tflux_{\omega}} \HHH^1_c(\Mm;\RR) \longrightarrow 1
\]
and the short exact sequence \eqref{eq=shortex_center} corresponds to
\[
1 \longrightarrow \Hamc(\Mm,\omega) \longrightarrow \Sympc_0(\Mm,\omega) \xrightarrow{\flux_{\omega}} \HHH^1_c(\Mm;\RR)/\Gamma_{\omega} \longrightarrow 1.
\]
Let $\qqm\colon \tSympc(\Mm,\omega)\twoheadrightarrow \Sympc_0(\Mm,\omega)$ be the universal covering map.

Assume that $R|_{\overline{\Pg}}\in \HHH^2(\overline{\Pg})$ is trivial. Then, by Proposition~\ref{prop=KKMMMReznikov} there exists $\phf\in \QQQ(\overline{\Pg})$ such that $\left(\phf|_{\Hamc(\Mm,\omega)} \right)\circ \qqm=\mushf$. In particular, $\mushf$ must descend to an element in $\QQQ(\Hamc(\Mm,\omega))^{\Sympc_0(\Mm,\omega)}$. By Lemma~\ref{lem=descend}, the restriction of $\mushf$ to $\pi_1(\Hamc(\Mm,\omega))$ must be the zero map. By Proposition~\ref{prop=ShelukhinHam}, this implies (1). Since $\phf\circ \left(\qqm|_{\Pg}\right)\in\QQQ(\Pg)$ is an extension of $\mushf$, we also obtain (2) by Proposition~\ref{prop=extendability}.

In what follows, we show the converse. Assume (1) and (2). Then, by Proposition~\ref{prop=ShelukhinHam}, we can apply Proposition~\ref{prop=center} (``(iii) implies (ii)'') to the setting of  Example~\ref{exa=Reznikovsetting} with $\muf=\mushf$. Hence, there exists $\phf\in \QQQ(\overline{L})$ as in condition (ii) of Proposition~\ref{prop=KKMMMReznikov}. By Proposition~\ref{prop=KKMMMReznikov}, we conclude that $R|_{\overline{\Pg}}\in \HHH^2(\overline{\Pg})$ is trivial.
\end{proof}

\begin{proof}[Proof of Theorem~$\ref{thm=ReznikovRR}$]
Note that (ii) trivially implies (i). In what follows, we prove the converse implication. Assume (i). Then, by Theorem~\ref{mthm=Reznikov}, $I_{c_1}\equiv 0$ and the restriction of $\bb_{\mushf}$ to $\tflux_{\omega}(\Pg)\times \tflux_{\omega}(\Pg)$ equals the zero-form. Since $\bb_{\mushf}$ is $\RR$-bilinear, we conclude that the restriction of $\bb_{\mushf}$ to $\tflux_{\omega}(\Gg_{\Pg}^{\RR})\times \tflux_{\omega}(\Gg_{\Pg}^{\RR})$ equals the zero-form as well. Again by Theorem~\ref{mthm=Reznikov}, we obtain (ii).
\end{proof}

\section{Axiomatized theorem toward the explicit expression of $\bb_{\mushf}$}\label{sec=kari1}

In this section and the next section, we use the symbol $\Nm$ for a symplectic manifold, not for a group. Recall that we equip the product of two symplectic manifolds with the product symplectic structure (Subsection~\ref{subsec=notation}).

As we mentioned in Subsection~\ref{subsec=outline}, we will obtain Theorem~\ref{mthm=explicitShelukhin} by proving an axiomatized theorem. In this section, we state the axiomatized theorem  (Theorem~\ref{thm=XQ} below), and provide the proof of it.

\subsection{Axiomatized theorem for Theorem~\ref{mthm=explicitShelukhin}}\label{subsec=axiom}

To state the axiomatized theorem for Theorem~\ref{mthm=explicitShelukhin}, we introduce the $(\Xm,\Qm)$-condition as follows. 
 We set $D_{\ee}=([0,1]\times[0,1])\setminus([2\ee, 1-2\ee] \times [2\ee, 1-2\ee])$ for fixed $\ee \in (0,1/4)$.
  Let $p\colon\RR^2\to\RR^2/\ZZ^2$ be the natural projection
  and set $P_{\ee}=p(D_{\ee})$.
  We consider $P_{\ee}$ as a  symplectic manifold with the symplectic form $\omega_0 = dx \wedge dy$, where $(x,y)$ is the standard coordinates on $P_{\ee} \subset \RR^2 / \ZZ^2$.
  
\begin{defn}[$(\Xm,\Qm)$-condition]\label{definition:MQ}
Let $(\Xm,\omega_\Xm)$ and $(\Qm,\omega_{\Qm})$ be two closed symplectic manifolds. We say that a symplectic manifold $(\Mm,\omega)$ satisfies the \emph{$(\Xm,Q)$-condition} if the following conditions are all satisfied.
\begin{enumerate}[label=\textup{(\arabic*)}]
  \item The symplectic manifold $(\Mm,\omega)$ is the direct product of $(\Xm,\omega_\Xm)$ and $(\Qm,\omega_{\Qm})$.
  \item There exists a positive integer $\nn$ such that $\Mm$ admits a system of symplectic embeddings $I=\{\iPN_i \colon P_{\epsilon_i} \times \Nm_i \to \Xm \}_{i=1, \dots, \nn}$, where $(\Nm_i, \omega_i)$ is a
  $(\dim \Xm - 2)$-dimensional closed symplectic manifold and $0< \ee_i < \frac14$.
  \item For each $i=1,\dots, \nn$, let $\ioP_{i \ast}^{P_{\ee_i},X}  \colon \widetilde{\Sympc_0}(P_{\epsilon_i},\omega_0) \to \widetilde{\Sympc_0}(\Xm,\omega_\Xm)$ be the map defined from $\ioP_{i}$ as in Subsection \ref{subsec=notation}  $\ref{item:emb}$, and set $\Gg_i := \mathop{\mathrm{Im}} \left( \ioP_{i \ast}^{P_{\ee_i},X} \right)$.
  Then, $\Gg_i$ commutes with $\Gg_j$ whenever $i$ and $j$ are distinct elements in $\{1,\ldots,\nn\}$.
  \item The flux homomorphim $\widetilde{\flux}_{\omega_{\Qm}}$ has a section homomorphism.
\end{enumerate}
\end{defn} 



Now we are ready to present the statement of the axiomatized theorem. 
Recall that $A(\Mm,\omega)$ denotes the average Hermitian scalar curvature (Definition~\ref{defn=aHsc}).

\begin{thm}[axiomatized theorem for Theorem~\textup{\ref{mthm=explicitShelukhin}}] \label{thm=XQ}
  Let $(\Xm,\omega_\Xm)$ and $(\Qm,\omega_{\Qm})$ be two closed symplectic manifolds. Let $(\Mm,\omega)$ be a $2n$-dimensional  symplectic manifold. Assume that $(\Mm,\omega)$ satisfies the $(\Xm,\Qm)$-condition, and take $\nn$ and $I=\{\iPN_i \colon P_{\epsilon_i} \times N_i \to \Xm\}_{i=1, \dots, \nn}$ as in Definition~\textup{\ref{definition:MQ}}. 
 Set $V_I := \bigoplus\limits_{i=1}^{\nn}  \ioP_{i \ast}(\HHH^1_c(P_{\ee_i} ; \RR)) \subset \HHH^1_c(\Xm ; \RR)$. 
Let $\bb_{I,\Qm}$ be a $\RR$-bilinear form on $V_I \oplus \HHH^1_c(\Qm;\RR) \subset \HHH^1_c(M ; \RR)$ defined by
  \[ \bb_{I,\Qm} =  n \cdot \left( \sum_{i=1}^{\nn} \vol(\Nm_i \times \Qm, \omega_{\Nm_i \times \Qm}) \left( A(\Mm, \omega) - A(\Nm_i \times Q , \omega_{\Nm_i \times \Qm}) \right) \ioP_{i \ast}^{P_{\ee_i},X} \bb_i \right) \oplus 0_{\Qm},\]
where $\bb_{i}$ denotes the intersection form on $\HHH^1_c( P_{\epsilon_i} ; \RR)$.
Then, we have
\[
\bb_{\mushf}|_{(V_I \oplus \HHH^1_c(\Qm;\RR))\times (V_I \oplus \HHH^1_c(\Qm;\RR))} \equiv\bb_{I,Q}.
\]
\end{thm}

In many applications of Theorem \ref{thm=XQ} in this paper, we only deal with the case where $\Qm$ is the one-point set; we will see in Section~\ref{sec=kari2}.


To prove Theorem \ref{thm=XQ}, we will first show the following proposition.

\begin{prop} \label{prop=XQ}
  Let $(\Mm,\omega)$ be a symplectic manifold satisfying the $(\Xm,\Qm)$-condition for some $\Xm$ and $\Qm$. Let $\mushf$ be Shelukhin's quasimorphism on $\tHamc(\Mm,\omega)$. Let $\gaml_1,\gaml_2\in V_I \oplus \HHH^1_c(\Qm;\RR)$. Then, there exist $m\in \NN$, a sequence  $(\fl_1^{(k)})_{k\in \NN}$ in $\tSympc(\Mm,\omega)$ and $\fl_2^+,\fl_2^-\in \tSympc(\Mm,\omega)$ such that the following hold.
  \begin{enumerate}[label=\textup{(\arabic*)}]
   \item $\tflux_{\omega}(\fl_2^+)+\tflux_{\omega}(\fl_2^-)=\frac{\gaml_2}{m}$.
   \item For every $k\in \NN$, $\tflux_{\omega}(\fl^{(k)})=k\gaml_1$.
   \item For every $k\in \NN$,
  \[
  \mushf([\fl_1^{(k)},\fl_2^+][\fl_1^{(k)},(\fl_2^-)^{-1}]^{-1})\sim_{\DD(\mushf)} \frac{k}{m}\cdot \bb_{I,Q}(\gaml_1,\gaml_2).
  \]
  \end{enumerate}
  \end{prop}


\subsection{Proof of Theorem~\ref{thm=XQ}}

In order to prove Theorem~\ref{thm=XQ}, we will use symplectomorphisms constructed in \cite{KKMM2}, so we briefly review them.
For $\qd =(a,b,c,d)\in \mathbb{R}^4$ satisfying $|c|\leq \ee$ and $|d|\leq \ee$,
the authors constructed ${\sigma}_{\qd}, {\tau}_{\qd}, {\sigma}'_{\qd}, {\tau}'_{\qd} \in \tSympc(P_\ee)$ 
satisfying the following property.

\begin{lem}[{\cite[Lemma 3.3]{KKMM2}}]\label{lem:local_flux_calculation}
  Let $\qd =(a,b,c,d)\in \mathbb{R}^4$ satisfy $|c|\leq \ee$ and $|d|\leq \ee$. 
Let $\alpha,\beta\colon[0,1]\to P_{\ee}$ be curves on $P_{\ee}$ defined by $\alpha(t)=p(0,t)$, $\beta(t)=p(t,0)$. Then,  we have
  \begin{gather*}
    \widetilde{\flux}_{\omega_{0}}\left( {\sigma}_{\qd} \right)=b [ \beta]^\ast , \quad
    \widetilde{\flux}_{\omega_{0}}\left( {\sigma}'_{\qd} \right)=a [ \alpha]^\ast,\\
    \widetilde{\flux}_{\omega_{0}}\left( {\tau}_{\qd} \right)=c  [ \alpha]^\ast, \quad
    \widetilde{\flux}_{\omega_{0}}\left( {\tau}'_{\qd} \right)=d[ \beta]^\ast.
  \end{gather*}
  \end{lem}

These elements ${\sigma}_{\qd}, {\tau}_{\qd}, {\sigma}'_{\qd}, {\tau}'_{\qd}$ are defined as the flow of the vector fields as shown in Figure \ref{vec_field} (see \cite{KKMM2} and \cite{KKMMMReznikov} for more details on the construction).

\begin{figure}[htbp]
  \begin{minipage}[c]{0.24\hsize}
    \centering
    \includegraphics[width=4truecm]{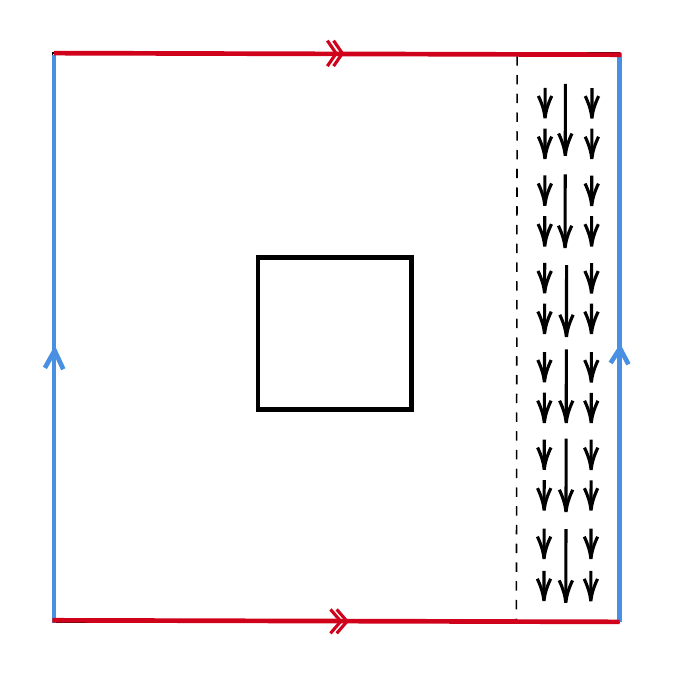}
  \end{minipage}
  \begin{minipage}[c]{0.24\hsize}
    \centering
    \includegraphics[width=4truecm]{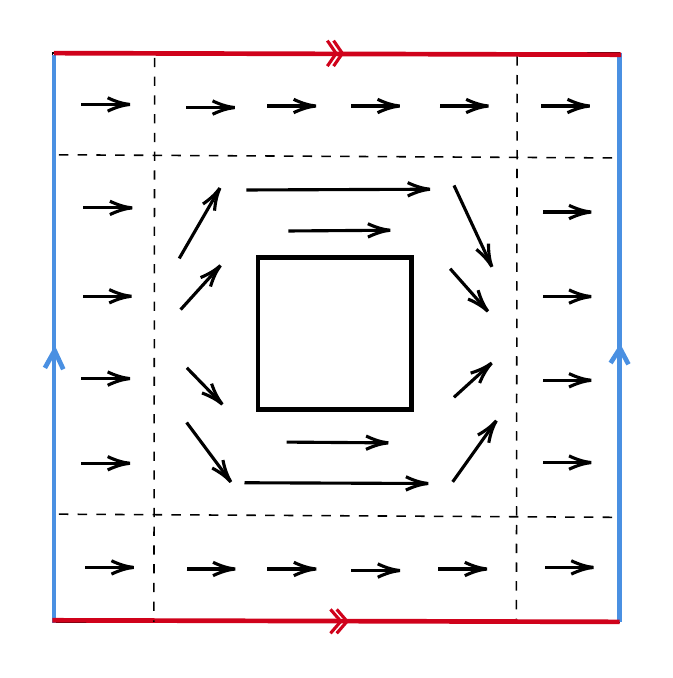}
  \end{minipage}
  \begin{minipage}[c]{0.24\hsize}
    \centering
    \includegraphics[width=4truecm]{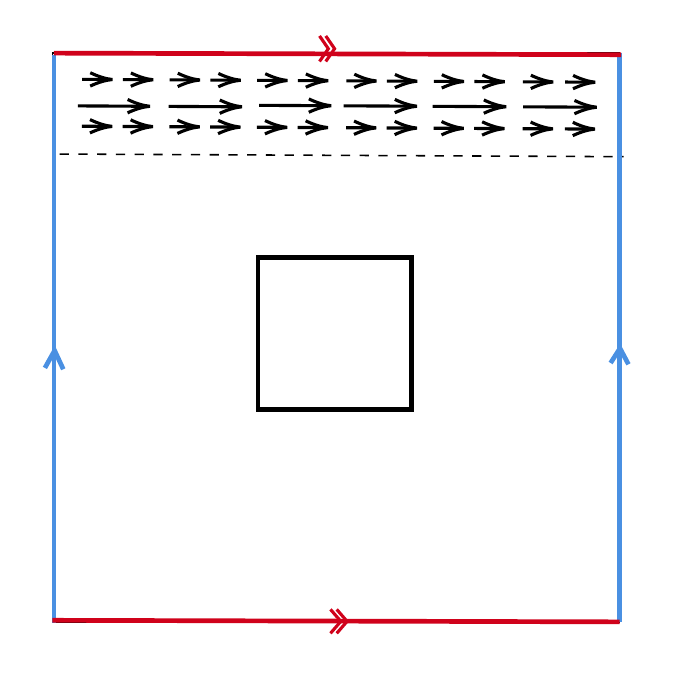}
  \end{minipage}
  \begin{minipage}[c]{0.24\hsize}
    \centering
    \includegraphics[width=4truecm]{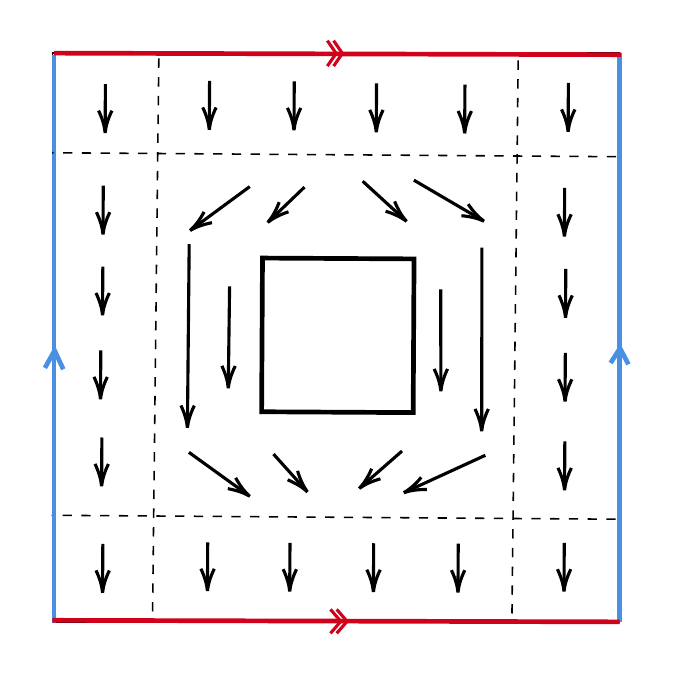}
  \end{minipage}
  \caption{Vector field defining $\sigma_{\qd}$, $\tau_{\qd}$, $\sigma'_{\qd}$ and $\tau'_{\qd}$}
  \label{vec_field}
\end{figure}

  We will use the following calculations in \cite{KKMMMReznikov}.

  \begin{lem}[{\cite[Lemma 4.3]{KKMMMReznikov}}]  \label{average scalar curvature}
  Let $(S,\omega_S)$ be a closed surface with a symplectic form $\omega_S$ and $(N, \omega_N)$ a closed symplectic manifold. Then,
  \[A (S \times N , \omega_{S\times N} ) = A(S,\omega_S) + A(N, \omega_N). \]
  \end{lem}

  \begin{prop}[{\cite[Proposition 4.6]{KKMMMReznikov}}]  \label{prop:flux}
  Let $(N, \omega_N)$ and $(M, \omega)$ be closed symplectic manifolds of dimensions $2n-2$ and $2n$, respectively.
    Assume that  $\iPN\colon (P_{\ee} \times N,\omega_{P_{\ee}\times N})\to (M ,\omega)$ is an open  symplectic embedding. 
   Let $\qd =(a,b,c,d)\in \mathbb{R}^4$ satisfy $|c|\leq \ee$ and $|d|\leq \ee$. 
   Then the following hold:
   \begin{enumerate}[label=\textup{(\arabic*)}]
     \item $ \mushf \left(\ioP^{P_{\ee},M}_{\ast}\left([\sigma_{\qd},\tau_{\qd}]\right)\right)=-bc  n \cdot \vol(N, \omega_N) \left( A(N,\omega_N) - A(M,\omega) \right)$;
     \item $ \mushf  \left(\ioP^{P_{\ee},M}_{\ast}\left([\sigma^{\prime}_{\qd},\tau^{\prime -1}_{\qd}]\right)\right)=-ad  n \cdot \vol(N, \omega_N) \left( A(N,\omega_N) - A(M,\omega) \right)$.
   \end{enumerate}
   \end{prop}
  We note that in \cite{KKMMMReznikov} only (1) of Proposition \ref{prop:flux} is shown, but (2) can be proved similarly.

Now we are ready to prove Theorem~\ref{thm=XQ}.
Fix a point $y \in N$ and take an embedding $P_\ee \to P_\ee \times N; x \mapsto (x,y)$. Let $\hat\ioP_j \colon P_\ee \to M$ be the composition of the embedding $P_\ee \to P_\ee \times N$ and $\iPN_j \colon P_\ee \times N \to M$. 
Set $\alpha_j = \hat\ioP_j \circ \alpha$ and $\beta_j = \hat\ioP_j \circ \beta$ for $j=1,\ldots, \nn$.
For $m \in \NN$, we set $v=\gamma_1$ and $w=\gamma_2/m$ and represent them as
    \begin{equation*}\label{eq:vw}
      v = \sum_{j=1}^{\nn}(a_j  [\alpha_j]^{\ast}+b_j [\beta_j]^{\ast}) + \delta_v
     \quad \text{and} \quad w = \sum_{j=1}^{\nn}(c_j [\alpha_j]^{\ast}+d_j [\beta_j]^{\ast}) + \delta_w,
    \end{equation*}
    where $\delta_v , \delta_w \in \HHH^1_c(Q;\RR)$. 
    We take sufficiently large $m$ such that $|c_j|< \epsilon_j$ and $|d_j|< \epsilon_j$ for every $j =1 ,\dots, l$.
    Set $\qd_j = (a_j,b_j,c_j,d_j)$. 

    Set $\iota^P_{j \ast}=\iota^{P_{\ee_j}, \Mm}_{j \ast}$. For $k \in \NN$, we can define $\fkp, \gkp, \hkp \in \tSympc(\XX,\omega_\XX)$ by
  \begin{equation*}\label{eq:fg}
    \fkp = \prod_{j=1}^l \iota^P_{j \ast} (\sigma_{\qd_j}^k \sigma_{\qd_j}^{\prime k}), \quad
   \gkp = \prod_{j=1}^l \iota^P_{j \ast} (\tau_{\qd_j}),  \quad \text{and} \quad
   \hkp = \prod_{j=1}^l \iota^P_{j \ast} (\tau'_{\qd_j})
  \end{equation*}
    since the groups $G_i$ $(i=1,\dots,\nn)$ are pairwise commutative.
    Let $s \colon \HHH^1_c(Q;\RR) \to \tSympc(Q,\omega_Q)$ be a section homomorphism of the flux homomorphism $\tflux_{\omega_Q}$.
    Define $\fk, \gk, \hk \in \tSympc(M,\omega)$ by
\begin{equation*}\label{eq:fg^prime}
  \fk= \eta\left(\fkp, s(k \cdot \delta_v)\right), \quad
   \gk = \eta\left(\gkp, s\left(\frac12 \cdot \delta_w\right)\right),  \quad  \text{and} \quad
   \hk = \eta\left(\hkp, s\left(\frac12 \cdot \delta_w\right)\right),
     \end{equation*}
where $\eta \colon \tSympc(\XX,\omega_\XX) \times \tSympc(Q,\omega_Q) \to \tSympc(M,\omega)$
is the natural homomorphism. 
The following is a counterpart of \cite[Lemma 4.5]{KKMM2}.

\begin{lem} \label{lem:inG}
The following hold.
\begin{enumerate}[label=\textup{(\arabic*)}]
 \item For every $k \in \NN$, $\tflux_{\omega}(\fk)=k \gamma_1$.
 \item $\tflux_{\omega}(\gk \hk) = \gamma_2/m$.
\end{enumerate}
\end{lem}

\begin{proof}

For every $k\in \NN$,
\begin{align*}
\tflux_{\omega}(\fk)&=\tflux_{\omega} \left( \eta \left(\prod_{j=1}^l  \iota^P_{j \ast} (\sigma_{\qd_j}^k \sigma_{\qd_j}^{\prime k}) , s( k \cdot \delta_v) \right) \right)\\
&=\sum_{j= 1 }^l\tflux_{\omega} \left(\iota^P_{j \ast} (\sigma_{\qd_j}^k \sigma_{\qd_j}^{\prime k}) \right) + \tflux_{\omega_Q}( s( k \cdot \delta_v))  \\
&=k\left( \sum_{j= 1 }^l\tflux_{\omega} \left( \iota^P_{j \ast} (\sigma_{\qd_j}) \right)+\sum_{j= 1 }^l\tflux_{\omega}\left( \iota^P_{j \ast} (\sigma'_{\qd_j}) \right) + \delta_v \right).
\end{align*}
By Lemma~\ref{lem:local_flux_calculation}, we have
\[
\tflux_{\omega} \left( \iota^P_{j \ast} (\sigma_{\qd_j}) \right)=b_j[\beta_j]^{\ast},\quad \tflux_{\omega} \left( \iota^P_{j \ast}(\sigma'_{\qd_j}) \right)=a_j[\alpha_j]^{\ast}
\]
for every $1\leq j\leq l$. 
Hence, we conclude that
\[
\tflux_{\omega}(\fk)=kv =k \gamma_1,
\]
which proves (1). In a manner similar to one above, we have
\[
\tflux_{\omega}(\gk \hk) = w =\frac{\gamma_2}{m};
\]
hence we have shown (2).
\end{proof}

\begin{prop}\label{prop:intersection}
The following holds true for every $k \in \NN$:
\[\mushf([\fk,\gk][\fk,(\hk)^{-1}]^{-1})\sim_{D(\mushf)}k \cdot \bb_{I,Q} (v,w).\]
\end{prop}

\begin{proof}

By the proof of \cite[Proposition 4.7]{KKMM2}, we have
\[ [\fkp,\gkp] = \prod_{j=1}^{\nn} \iota^P_{j \ast} ( [\sigma_{\qd_j} , \tau_{\qd_j} ]^k )
\quad \text{and} \quad [\fkp,(\hkp)^{-1}]=\prod_{j=1}^l \iota^P_{j \ast} ([\sigma'_{\qd_j}, \tau^{\prime-1}_{\qd_j}]^k ). \]
Let $\kappa_j \colon P_{\epsilon_j} \times N_j \times Q  \to M(=X\times Q) $ denote the symplectic embedding defined by 
$\kappa_j (x,y,z)=( \iota_j (x,y), z )$. Set $\kappa^P_{j \ast} := \kappa^{P_{\ee_j}, M}_{i \ast} $.
Then, 
\begin{align*}
  [\fk,\gk] & = \eta ([\fkp,\gkp], [ s(k \cdot \delta_v), s(\delta_w) ]  ) \\
&  =\eta \left( \prod_{j=1}^l \iota^P_{j \ast} ( [\sigma_{\qd_j}, \tau_{\qd_j}]^k ) , \mathrm{id}_Q  \right) 
 =\prod_{j=1}^l \kappa^P_{j \ast} ([\sigma_{\qd_j}, \tau_{\qd_j}]^k ).
\end{align*}
Hence, by Lemma \ref{lem:qm}~(2), and by applying Proposition \ref{prop:flux}~(1) to $\kappa_i$,
\begin{equation} \label{eq:bc}
  \begin{aligned}
    \mushf ([ \fk , \gk ])
    & = \sum_{i=1}^{\nn} \mushf \Bigl(\kappa^P_{i \ast} ([\sigma_{\qd},\tau_{\qd}]^k ) \Bigr)
    = k \sum_{i=1}^{\nn} \mushf \Bigl(\kappa^P_{i \ast}([\sigma_{\qd},\tau_{\qd}] ) \Bigr) \\
    & = -k n \sum_{i=1}^{\nn} b_jc_j \cdot \vol(N_i \times Q, \omega_{N_i \times Q}) \cdot ( A(N_i \times Q ,\omega_{N_i \times Q}) - A(M,\omega)) .     
  \end{aligned}
\end{equation}
Similarly, by Proposition \ref{prop:flux}~(2), we can prove that 
\begin{equation} \label{eq:ad}
  \mushf([{\fk},(\hk)^{-1}])=-k n \sum_{i=1}^{\nn}  a_i d_i \cdot \vol(N_i \times Q, \omega_{N_i \times Q}) \cdot (A(N_i \times Q ,\omega_{N_i \times Q}) - A(M,\omega)).
\end{equation}
Therefore, by \eqref{eq:bc} and \eqref{eq:ad},
\begin{align*}
 & \mushf([\fk,\gk][\fk,\hk]^{-1}) \\
  \sim_{D(\mushf)}{} & \mushf( [ \fk , \gk ] ) + \mushf( [ \fk , (\hk)^{-1} ]^{-1} ) \\
 ={} & \mushf( [ \fk , \gk ] ) - \mushf( [ \fk, (\hk)^{-1} ] ) \\
 ={} & k n \sum_{i=1}^{\nn} (a_id_i-b_ic_i) \cdot \vol(N_i \times Q, \omega_{N_i \times Q}) \cdot ( A(N_i \times Q,\omega_{N_i \times Q}) - A(M,\omega)) \\
 ={} & k \cdot \bb_{Q,I} (v,w). \qedhere
  \end{align*}
\end{proof}

  \begin{proof}[Proof of Proposition~\textup{\ref{prop=XQ}}]
  Items (1) and (2) follow from Lemma \ref{lem:inG}, and (3) follows from Proposition \ref{prop:intersection}.
\end{proof}

Now we are in a position to prove Theorem~\ref{thm=XQ}.

\begin{proof}[Proof of Theorem~\textup{\ref{thm=XQ}}]
It follows from the combination of Proposition~\ref{prop=XQ} and the refined limit formula \eqref{eq=refined_limit_formula_left}. More precisely, our proof goes along a line similar to the proof of Theorem~\ref{thm=Py} by Proposition~\ref{prop=KKMM} and \eqref{eq=refined_limit_formula_left}.
\end{proof}

\section{Results on the explicit expression of $\bb_{\mushf}$}\label{sec=kari2}
In this section, we prove Theorem~\ref{mthm=explicitShelukhin}. We recall the statement  by splitting it into three theorems, as follows.

\begin{thm}\label{surface product with vanishing H1_triple}
Let $S$ be a closed orientable surface whose genus $l$ is at least two and $N$ a closed manifold.
Let $\omega_S$, $\omega_N$ be symplectic forms on $S$, $N$, respectively. Set $(M,\omega)=(S,\omega_S)\times (N,\omega_N)$.
Let $2n$ be the dimension of $M$. Then we have
\[\bb_{\mushf}|_{\HHH^1_c(S;\RR)\times \HHH^1_c(S;\RR)} = n(2-2l)\frac{\vol(N,\omega_N)}{\mathrm{Area}(S,\omega_S)} \cdot \bb_{\omega_S} = \vol(M,\omega) \frac{(2-2l)}{\mathrm{Area}(S,\omega_S)^2} \cdot \bb_{\omega_S}.\]
 Here, we regard $\HHH^1_c(S;\RR)$ as a subspace of $\HHH^1_c(M;\RR) \cong \HHH^1_c(S;\RR) \oplus \HHH^1_c(N;\RR)$. 
\end{thm}

\begin{thm}\label{surface product_triple}
Let $S_i$ be a closed orientable surface whose genus $l_i$ is at least  one and $\omega_i$ is a symplectic form on $S_i$ \textup{(}$i=1,\ldots,n$\textup{)}.
Set $(M,\omega)=(S_1,\omega_1)\times\cdots\times(S_n,\omega_n)$.
Then, we have
\[\bb_{\mushf}(v,w) =\vol(M,\omega) \sum_{i=1}^n \frac{2-2l_i}{{\rm Area}(S_i,\omega_i)^2} \bb_{\omega_{S_i}}(v_i,w_i) \]
for every
$v=(v_1,\ldots,v_n)$ and $w=(w_1,\ldots,w_n)$ in $\HHH^1_c(M;\RR) \cong \HHH^1_c(S_1;\RR) \oplus \cdots \oplus \HHH^1_c(S_n;\RR)$.
Here, for every $i = 1,\ldots, n$, we regard $v_i$ and $w_i$ as elements in $\HHH^1_c(S_i;\RR)$. 
\end{thm}

\begin{thm}\label{torus blow up_triple}
Let $n\geq2$ and $r_1,\ldots,r_n$ be real numbers with $0<r_1<r_2<\cdots<r_n$.
For real numbers $\rho,r$ with $0<\rho<r < r_1$,
let $(\hat{M}_r,\omega_\rho)$ be the blow-up of $(T^{2n}(r_1,\ldots,r_n),\omega_0)$ with respect to $\iota(r)$, $J_0$ and $\rho$.
Then, 
\[\bb_{\mushf}(v,w) = \vol(T^{2n}(r_1,\ldots,r_n),\omega_0) A(\hat{M}_r,\omega_{\rho}) \sum_{i=1}^n \frac{1}{{\rm Area}(T^2(r_i),\omega_0)} \bb_{\omega_0}(v_i,w_i) \]
for every $v=(v_1,\ldots,v_n)$ and $w=(w_1,\ldots,w_n)$ in 
\begin{align*} 
 &\HHH^1_c(\hat{M}_r;\RR)\cong \HHH^1_c(T^{2n}(r_1,\ldots,r_n);\RR) \\ 
\cong{}  &\HHH^1_c(T^2(r_1);\RR) \oplus \HHH^1_c(T^2(r_2);\RR) \oplus \cdots \oplus \HHH^1_c(T^2(r_n);\RR).
\end{align*}
where we regard $v_i$ and $w_i$ as elements in $\HHH^1_c(T^2(r_i);\RR)$ for every $i = 1,\ldots, n$. 
\end{thm}

We show these three theorems by checking the $(X,Q)$-condition (Definition~\ref{definition:MQ}) for $(\Mm,\omega)$ and then by applying the axiomatized theorem (Theorem~\ref{thm=XQ}); the first step is treated in Subsection~\ref{subsec:XQ}, and the second step is discussed in Subsection~\ref{subsec=explicit}.

\subsection{Symplectic manifolds satisfying the $(X,Q)$-condition}\label{subsec:XQ}

Recall that $S$ is a closed orientable surface whose genus $l$ is at least two. We take curves $\alpha_1$, $\beta_1$, $\alpha_2$, $\beta_2$, $\ldots$, $\alpha_\nn$, $\beta_\nn \colon [0,1] \to S$ on $S$ as depicted in Figure~\ref{curves}.

\begin{figure}[htbp]
\centering
\includegraphics[width=7truecm]{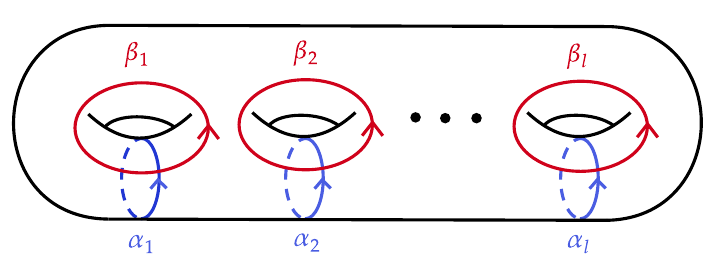}
\caption{Curves $\alpha_1,\beta_1, \alpha_2, \beta_2, \ldots,\alpha_\nn,\beta_\nn$ on $S$}
\label{curves}
\end{figure}

For a sufficiently small $\epsilon>0$, we take symplectic embeddings $\iota^{S}_1,\ldots,\iota^{S}_{\nn} \colon (P_\epsilon,\omega_0)\to (S,\omega)$ satisfying the following two conditions:
\begin{enumerate}[label=\textup{(\arabic*)}]
  \item $\iota^{S}_j \circ \alpha = \alpha_{j}$, $\iota^{S}_j \circ \beta = \beta_{ j }$ for $j = 1,\ldots,\nn$.
  \item $\iota^{S}_j (P_\epsilon)\cap \iota^{S}_j (P_\epsilon)=\emptyset$ if $i\neq j$.
\end{enumerate}

Let $\ast$ denote the one-point space. 

\begin{prop}\label{surface product with vanishing H1_MQ}
Let $S$ be a closed orientable surface whose genus $l$ is at least two and $N$ a closed manifold.
Let $\omega_S$, $\omega_N$ be symplectic forms on $S$, $N$, respectively and set $(M,\omega)=(S,\omega_S)\times (N,\omega_N)$.
 Then $M$ satisfies the $(S\times N,\ast)$-condition. 
\end{prop}

\begin{proof}
Let $I=\{ \iota_i \colon P_{\epsilon}\times N \to S\times N= M \}_{i=1,\dots,\nn}$ be the system of embeddings induced by $\{\iota^S_i \colon  P_{\epsilon} \to S\}_{i=1,\dots,\nn}$. 
Then, by condition (2) above, $G_i$ and $G_j$ commute if $i\neq j$. 
\end{proof}

\begin{prop}\label{surface product_MQ}
Let $S_i$ be a closed orientable surface whose genus $l_i$ is at least one and $\omega_i$ is a symplectic form on $S_i$ \textup{(}$i=1,\ldots,n$\textup{)}.
Assume that $l_i > 1$ if $i\leq n_0$ and $l_{n_0+1}=\cdots=l_n=1$ for some $n_0$.
Set $(M,\omega)=(S_1,\omega_1)\times\cdots\times(S_n,\omega_n)$, $X=S_1 \times \cdots \times S_{n_0}$ and $Q= S_{n_0+1} \times \cdots \times S_{n}$.
Then $M$ satisfies the $(X,Q)$-condition.
\end{prop}

\begin{proof}
For each surface $S_i$, we take sufficiently small $\epsilon_i$ and symplectic embeddings $\iota^{S}_{i,1},\ldots,\iota^{S}_{i,\nn_i} \colon (P_{\epsilon_i},\omega_0)\to (S_i,\omega_i)$ as above. Let 
\[I=\{ \iota_{i,j} \colon S_1 \times \cdots \times S_{i-1} \times P_{\epsilon_i} \times S_{i+1} \times \cdots \times S_{n_0} \to X \}_{1\leq i \leq n_0, 1\leq j \leq l_i}  \]
be the system of embeddings induced by $\iota^{S}_{i,j} \colon P_{\epsilon_i} \to S_i$. 
Then, by condition (2) above again, each $G_{i,j} := \mathrm{Im} \left( \iota^{P_{\epsilon_i}, M}_{i,j \ast} \right)$ commutes with others. Since $Q$ is a $2(n-n_0)$-dimensional torus, $\tflux_Q$ admits a section homomorphism.
\end{proof}


\begin{prop}\label{torus blow up_MQ}
Let $n\geq2$ and $r_1,\ldots,r_n$ be real numbers with $0<r_1<r_2<\cdots<r_n$.
For real numbers $\rho,r$ with $0<\rho<r < r_1$,
let $(\hat{M}_r,\omega_\rho)$ be the blow-up of $(T^{2n}(r_1,\ldots,r_n),\omega_0)$ with respect to $\iota(r)$, $J_0$ and $\rho$.
Then $\hat{M}_r$ satisfies the $(\hat{M}_r,\ast)$-condition.
\end{prop}

\begin{proof}  
For $\epsilon >0$ and a system of symplectic embeddings $\{\kappa_i \colon P_{\epsilon} \to T^2(r_i)\}_{i=1,\dots,n}$, let 
  \[\hat\kappa_i \colon T^{2(i-1)}(r_1,\dots, r_{i-1})\times P_{\epsilon} \times T^{2(n-i)}(r_{i+1},\dots, r_{n}) \to T^{2n}(r_1,\dots, r_n) \]
denote the symplectic embedding defined by 
\[\hat\kappa_i (z_1,\dots,z_i,\dots,z_n)=(z_1,\dots,\kappa(z_i),\dots,z_n).\]
Let $p_i\colon \RR^2\to T^{2}(r_i)$, $p \colon \RR^{2n}\to T^{2n}(r_1,\dots, r_n)$ be the standard projections.
Take $\epsilon>0$ such that the area of $P_\epsilon$ is smaller than the area of $T^2(r_i) \setminus p(B^2(r))$ for each $i$.
Then, we can take a system of symplectic embeddings $\{ \kappa_i \colon P_{\epsilon} \to T^2(r_i) \}_{i=1,\dots,n}$ such that the image of $\kappa_i$ and $p_i(B^2(r))$ are disjoint.
Thus, the image of $\hat\kappa_i$ and $p(B^{2n}(r))$ are disjoint.
Hence, the symplectic embeddings $\hat\kappa_1, \dots, \hat\kappa_n$ induce the system of symplectic embeddings
\[I=\{ \iota_i \colon T^{2(i-1)}(r_1,\dots, r_{i-1})\times P_{\epsilon} \times  T^{2(n-i)}(r_{i+1},\dots, r_{n}) \to \hat{M}_r\}_{i=1,\dots,n}.\] 
By the construction of $I$, we conclude that $G_i$ and $G_j$ commute if $i\neq j$. 
\end{proof}


\subsection{Explicit expression of the bilinear form} \label{subsec=explicit}

In this subsection, we prove Theorem~\ref{mthm=explicitShelukhin}, which determines fully or partially the explicit expression of the bilinear form $\bb_{\mushf}$ for a certain $(\Mm,\omega)$. First, we compute the average Hermitian scalar curvature (Definition~\ref{defn=aHsc}) of closed symplectic surfaces.

\begin{lem}\label{lem=Asurface}
Let $(S,\omega)$ be a closed surface of genus $l$  equipped with a symplectic form. Then, we have
\[
A(S,\omega)=\frac{\chi(S)}{\mathrm{Area}(S,\omega)}=\frac{2-2l}{\mathrm{Area}(S,\omega)}.
\]
Here $\chi$ is the Euler characteristic.
\end{lem}

\begin{proof}
By applying \eqref{eq:aHsc} with $n=1$, we have
\[
A(S,\omega)= \left. \int_{S} c_1(S)  \middle/ \int_{S} \omega \right. =\frac{\chi(S)}{\mathrm{Area}(S,\omega)}. \qedhere
\]
\end{proof}

We also recall that $\bb_{\omega}$ denotes the symplectic pairing for a symplectic manifold $(\Mm,\omega)$, which is defined by \eqref{eq=symplecticpairing}.

\begin{proof}[Proof of Theorem~\textup{\ref{surface product with vanishing H1_triple}}]
By Proposition~\ref{surface product with vanishing H1_MQ}, $M$ satisfies the $(S \times N,\ast)$-condition. 
Take $I$ as in the proof of Proposition~\ref{surface product with vanishing H1_MQ} and set  $Q=\ast$.
Then, $V_I \oplus \HHH^1_c(Q;\RR) = \HHH^1_c(S;\RR) \oplus 0 = \HHH^1_c(S;\RR)$.
By Lemmas \ref{average scalar curvature}  and  \ref{lem=Asurface},  
\begin{align*}
  \bb_{I,Q} &=n \sum_{i=1}^l  \vol(N, \omega_N) ( A(M,\omega) -A(N,\omega_N) )  \iota^{P_\ee, M}_{i \ast} \bb_i = n  \sum_{i=1}^l  \vol(N, \omega_N)  A(S,\omega_S) \iota^{P_\ee, M}_{i \ast}\bb_i  \\ 
&  =   n (2-2l) \frac{ \vol(N, \omega_N) }{\mathrm{Area}(S,\omega_S)} \cdot\bb_{\omega_S} = \vol(M, \omega)\frac{ 2-2 l }{\mathrm{Area}(S,\omega_S)^2} \bb_{\omega_S}
\end{align*}
since $\vol(M,\omega)=n  \vol(N ,\omega_N)  \vol(S, \omega_S)$. 
Therefore, the assertion follows from Theorem \ref{thm=XQ}.
\end{proof}


\begin{proof}[Proof of Theorem~\textup{\ref{surface product_triple}}]
We can assume that $l_i > 1$ if $i\leq n_0$ and $l_{n_0+1}=\cdots=l_n=1$ for some $n_0$ without loss of generality.
Set $X=S_{1} \times \cdots \times S_{n_0}$ and $Q= S_{n_0+1} \times \cdots \times S_{n}$. 
By Proposition~\ref{surface product_MQ}, $M$ satisfies the $(X,Q)$-condition. 
Take $I$ as in the proof of Proposition~\ref{surface product_MQ}. 
Then, $V_I \oplus \HHH^1_c(Q;\RR) =\HHH^1_c(X;\RR) \oplus \HHH^1_c(Q;\RR)=\HHH^1_c(M;\RR)$.
Let $(N_i, \omega_{N_i})$ denote $(S_1 \times \cdots \times S_{i-1} \times S_{i+1} \times \cdots \times S_{n_0} \times Q, \omega_0)$.
Note that $\vol(M,\omega)=n \vol(N_i, \omega_{N_i}) {\rm Area}( S_i, \omega_{i})$.
By Lemmas \ref{average scalar curvature} and \ref{lem=Asurface}, 
\begin{multline*}
  \bb_{I,Q}  = n \Biggl(  \sum_{\substack{1\leq i\leq n_0 \\ 1\leq j \leq l_i}}
 \vol(N_i,\omega_{N_i}) 
 ( A(M,\omega) -A(N_i, \omega_{N_i})  ) \iota^{P_\ee, \Xm}_{i,j \ast}  \bb_i \Biggr) \oplus 0_Q \\
 = \Biggl( \sum_{\substack{1\leq i\leq n_0 \\ 1\leq j \leq l_i}}  \frac{\vol(\Mm ,\omega)}{{\rm Area}(S_i,\omega_i)}   A(S_i, \omega_i) \iota^{P_\ee, \Xm}_{i,j \ast} \bb_i \Biggr) \oplus 0_Q 
 = \vol(\Mm ,\omega) \Biggl(  \sum_{\substack{1\leq i\leq n_0 \\ 1\leq j \leq l_i}}  \frac{2-2l_i}{{\rm Area}(S_i,\omega_i)^2}  \iota^{P_\ee, \Xm}_{i,j \ast} \bb_i \Biggr) \oplus 0_Q.
\end{multline*}
  Thus, we have
 \begin{align*}
\bb_{I,Q}(v,w) =  \vol(M,\omega)  \sum_{i=1}^{n_0} \frac{2-2l_i}{{\rm Area}(S_i,\omega_i)^2}  \bb_{\omega_{S_i}}(v_i,w_i)=  \vol(M,\omega)  \sum_{i=1}^{n} \frac{2-2l_i}{{\rm Area}(S_i,\omega_i)^2}  \bb_{\omega_{S_i}}(v_i,w_i).
 \end{align*}
 Therefore, the assertion follows from Theorem \ref{thm=XQ}.
\end{proof}


\begin{proof}[Proof of Theorem~\textup{\ref{torus blow up_triple}}]
By Proposition~\ref{torus blow up_MQ}, $M$ satisfies the $(\hat{M}_r,\ast)$-condition. 
Take $I$ as in the proof of Proposition~\ref{torus blow up_MQ} and $Q=\ast$.  
Then $V_I=\HHH_c^1(\hat{M}_r;\RR)$. 
For every $i=1,\ldots,n$, let $(N_i,\omega_i)$ denote $\left(T^{2(i-1)}(r_1,\dots, r_{i-1})\times T^{2(n-i)}(r_{i+1},\dots, r_{n}) ,\omega_0\right)$. 
Note that $\vol(M,\omega)=n \vol(N_i, \omega_{i}) {\rm Area}( T^2(r_i), \omega_{0})$.
Since $A(N_i,\omega_i)=0$ for each $i=1,\dots, n$,
\begin{align*}
\bb_{I,Q} &= n \sum_{i=1}^{n}  \vol(N_i,\omega_i)  \left( A(\hat{M}_r,\omega_{\rho}) -A(N_i,\omega_i)  \right) \iota^{P_\ee, \hat{M}_r}_{i \ast} \bb_i \\
& =  \sum_{i=1}^{n} \frac{\vol(T^{2n}(r_1,\ldots,r_n),\omega_0) }{{\rm Area}(T^2(r_i),\omega_0)}   A(\hat{M}_r,\omega_{\rho}) \iota^{P_\ee, \hat{M}_r}_{i \ast}  \bb_i. 
\end{align*}
Therefore, the assertion follows from Theorem \ref{thm=XQ}.
\end{proof}

\begin{proof}[Proof of Theorem~\textup{\ref{mthm=explicitShelukhin}}]
Combine Theorems~\ref{surface product with vanishing H1_triple}, \ref{surface product_triple} and \ref{torus blow up_triple}.
\end{proof}

\subsection{Applications}

The following result corresponds to \cite[Theorem~A]{KKMMMReznikov} via Theorem~\ref{mthm=Shelukhin_extendable}.

\begin{cor}[non-vanishing results of $\bb_{\mushf}$]\label{cor=non-vanishing_Shelukhin}
Assume that $(M,\omega)$ is one of the following:
\begin{enumerate}[label=\textup{(\arabic*)}]
 \item The direct product $(S,\omega_S)\times (N,\omega_N)$ of a closed surface $(S,\omega_S)$ whose genus is at least two equipped with the symplectic form $\omega_S$, and a closed symplectic manifold $(N,\omega_N)$.
 \item The blow-up $(\hat{M}_r,\omega_\rho)$ of the torus $(T^{2n}(r_1,\ldots,r_n),\omega_0)$ with respect to $\iota(r)$, $J_0$ and $\rho$.
\end{enumerate}
Then, $\bb_{\mushf}$ is \emph{not} equal to the zero-form.
\end{cor}

\begin{proof}
Observe that $2-2l\ne 0$ for every integer $l$ at least two. Now, (1)  immediately follows from Theorem~\ref{surface product with vanishing H1_triple}. For (2), combine Theorem~\ref{torus blow up_triple} with Lemma~\ref{torus blow up lemma}.
\end{proof}

Corollary~\ref{cor=non-vanishing_Shelukhin} yields the following non-zero bound from below of the real dimension of the $\RR$-linear space appearing in Theorem~\ref{thm=findimCVSI_strong}.

\begin{prop}\label{prop=dimatleast1}
Let $(\Mm, \omega)$ be a closed  symplectic manifold satisfying the assumption of Corollary~\textup{\ref{cor=non-vanishing_Shelukhin}}. Let $\Gg=\tSympc(\Mm,\omega)$ and $\Ng=\tHamc(\Mm,\omega)$. Let $\Qv^{\mathrm{ch}}_{(\Mm,\omega)}$ be the $\RR$-linear subspace of $\QQQ(\Ng)^{\Gg}$ defined in Theorem~\textup{\ref{thm=findimCVSI_strong}}.
Then the image of $\Qv^{\mathrm{ch}}_{(\Mm,\omega)}$ under the projection $\QQQ(\Ng)^{\Gg}\twoheadrightarrow \VV(\Gg, \Ng)$ has real dimension  at least $1$.
\end{prop}

\begin{proof}
Let $\Gamg=\Gg/\Ng=\HHH^1_c(\Mm;\RR)$. Let $\bBbG\colon \VV(\Gg,\Ng)\hookrightarrow \AAA(\Gamg)$ be the map defined in Definition~\ref{defn=bilinear_form_map}. Then, by Corollary~\ref{cor=non-vanishing_Shelukhin} we have
\[
\bBbG([\mushf])=\bb_{\mushf}\ne 0 \in \AAA(\Gamg).
\]
Now, recall from Proposition~\ref{prop=Shelukhin_ch} that $\mushf\in \Qv^{\mathrm{ch}}_{(\Mm,\omega)}$. Therefore, we obtain the desired dimension estimate from below.
\end{proof}

For a closed symplectic surface $(\Mm,\omega)$,  we have the following result. Here recall from Remark~\ref{rem=intersection} that $\bb_{\omega}$ coincides with the intersection form $\bb_I$ for a closed symplectic surface $(\Mm,\omega)$.

\begin{cor}\label{cor=bsurface}
Let $(\Mm,\omega)$ be a closed surface  equipped with a symplectic form whose genus $l$ is at least two. Then we have
\[
\bb_{\mushf}=\frac{2-2l}{\mathrm{Area}(\Mm)}\cdot \bb_{\omega}=\frac{2-2l}{\mathrm{Area}(\Mm)}\cdot \bb_{I}.
\]
\end{cor}

\begin{proof}
Apply Theorem~\ref{surface product with vanishing H1_triple} with $N=\ast$, or apply Theorem~\ref{surface product_triple} with $n=1$.
\end{proof}


Now we are in a position to prove Theorem~\ref{thm=ShelukhinPy}.

\begin{proof}[Proof of Theorem~\textup{\ref{thm=ShelukhinPy}}]
Set
\begin{equation}\label{eq=ShelukhinPy}
\muf_{\mathrm{SP}}=\mushf-A(\Mm,\omega)\mupyf=\mushf-\frac{2-2l}{\mathrm{Area}(\Mm,\omega)}\mupyf.
\end{equation}
Let $\bBb=\bBbSymp\colon \VV(\tSympc(\Mm,\omega),\tHamc(\Mm,\omega))\to$$\AAA(\HHH^1_c(\Mm;\RR))$ be the map defined in Definition~\ref{defn=bilinear_form_map}. 
Then, by Theorem~\ref{thm=Py} and Corollary~\ref{cor=bsurface}, we have
\[
\bBb([\muf_{\mathrm{SP}}])=\frac{2-2l}{\mathrm{Area}(\Mm,\omega)}\bb_{\omega}-\frac{2-2l}{\mathrm{Area}(\Mm,\omega)}\bb_{\omega}=0.
\]
Since $\bBb$ is injective, we conclude that $[\muf_{\mathrm{SP}}]=0$ in $\VV(\tSympc(\Mm,\omega),\tHamc(\Mm,\omega))$; in other words, $\muf_{\mathrm{SP}}\in i^{\ast}\QQQ(\tSympc(\Mm,\omega))$.
\end{proof}

\begin{rem}\label{rem=local_Calabi}
Let $(\Mm,\omega)$ a symplectic manifold of finite volume and
$\iota \colon B=B^{2n}(r) \to \Mm$ be an open symplectic embedding. 
Then, Shelukhin's quasimorohism $\mushf$ has the following local form by \cite[Theorem 3]{Shelukhin}:
\begin{equation}\label{eq=Shelukhin_local}
  \mushf \circ  \iota_\ast = \frac12 \tau_B - A(\Mm,\omega) \cdot \cal_{B}, 
\end{equation}
where $\tau_B \colon \tHamc (B, \omega_0) \to \RR$ denotes the average Maslov quasimorphism of Barge--Ghys (we refer the reader to \cite{Bra} or \cite[Definition 1.7.3]{Shelukhin} for more details) and $\iota_\ast \colon \tHamc(B,\omega_0) \to \tHamc(\Mm,\omega)$ is the natural homomorphism induced by $\iota$. 

Now, we stick to the setting of Theorem~\ref{thm=ShelukhinPy}, and set $\muf_{\mathrm{SP}}$ by \eqref{eq=ShelukhinPy}. Then, by \eqref{eq=Shelukhin_local} and the Calabi property of $\mupyf$, this $\muf_{\mathrm{SP}}$ has the following local form for an embedded ball $B$ in $(\Mm,\omega)$:
\begin{equation}\label{eq=mu_local}
\muf_{\mathrm{SP}} \circ  \iota_\ast = \frac{1}{2} \tau_B - 2A(\Mm,\omega)\cdot \cal_B.
\end{equation}
In particular, \eqref{eq=mu_local} implies that the element $\muf_{\mathrm{SP}}$  in $i^{\ast}\QQQ(\tSympc(\Mm,\omega))$ is \emph{not} the zero map.
\end{rem}

By taking the contraposition of Theorem~\ref{mthm=Reznikov}, we in particular have the following result: if condition (2) of Theorem~\ref{mthm=Reznikov} fails, then the Reznikov class $R|_{\overline{\Pg}}\in \HHH^2(\overline{\Pg})$ is \emph{non-trivial}. The special case of this statement for $\Pg=\tSympc(\Mm,\omega)$ was obtained in \cite{KKMMMReznikov}. In the last part of this subsection, we exhibit the following corollary to Theorem~\ref{mthm=Reznikov}, which completely determines the triviality of the Reznikov class. We note that the following two sepcial cases of Corollary~\ref{cor=surfacetrivial} correspond to \cite[Corollary~5.6]{KKMMMReznikov}. By setting $V=\{0\}$, we have the triviality of $R|_{\Hamc(\Mm,\omega)}$; by setting $V=\HHH^1_c(\Mm;\RR)$, we obtain the non-triviality of $R|_{\Sympc_0(\Mm,\omega)}$.

\begin{cor}\label{cor=surfacetrivial}
Let $(\Mm,\omega)$ be a closed surface of genus at least two  equipped with a symplectic form. Let $\bb_I\colon \HHH^1_c(\Mm;\RR)\times \HHH^1_c(\Mm;\RR)\to \RR$ be the intersection form on $\HHH^1_c(\Mm;\RR)$. Let $V$ be an $\RR$-linear subspace of $\HHH^1_c(\Mm;\RR)$, and set $\Gg_{V}=\flux_{\omega}^{-1}(V)$. Then the Reznikov class $R$ is trivial on $\Gg_V$ if and only if $V$ is an isotropic subspace of the symplectic vector space $(\HHH^1_c(\Mm;\RR),\bb_I)$.
\end{cor}

\begin{proof}
In this case, $\pi_1(\Hamc(\Mm,\omega))$ is trivial (Proposition~\ref{survey on flux}~(6)). In particular, $I_{c_1}\equiv 0$. Let $l$ be the genus of $\Mm$. Then, Corollary~\ref{cor=bsurface} shows that
\[
\bb_{\mushf}=\frac{2-2l}{{\rm Area}(\Mm,\omega)} \cdot\bb_{I}.
\]
Since $l\geq 2$, Theorem~\ref{mthm=Reznikov} implies the desired equivalence.
\end{proof}


\section*{Acknowledgment}
The fifth-named author wishes to express his deepest gratitude to Yash Lodha for the kind invitation to the summer school `Topological and Analytic methods in group theory' at the University of Hawai`i at M\={a}noa in July, 2024; there, he came up with one of the key ideas of the present work. The first-named author, the second-named author, the fourth-named author and the fifth-named author are partially supported by JSPS KAKENHI Grant Number JP21K13790, JP24K16921, JP23K12975 and JP21K03241, respectively.
The third-named author is partially supported by JSPS KAKENHI Grant Number JP23KJ1938 and JP23K12971.

\bibliography{reference}
\bibliographystyle{abbrv}
\end{document}